\documentclass{amsart}

\usepackage{amsmath}
\usepackage{amssymb}

\usepackage{algpseudocode}
\usepackage[hmargin=30mm,vmargin=30mm]{geometry}
\usepackage{graphicx}
\usepackage{hyperref}
\usepackage{apacite}

\newtheorem{theorem}{Theorem}[section]
\newtheorem{lemma}[theorem]{Lemma}
\newtheorem{prop}[theorem]{Proposition}
\newtheorem{corollary}[theorem]{Corollary}

\theoremstyle{definition}
\newtheorem{definition}[theorem]{Definition}

\theoremstyle{remark}
\newtheorem{remark}[theorem]{Remark}

\newcommand{\mres}{\mathbin{\vrule height 1.6ex depth 0pt width
0.13ex\vrule height 0.13ex depth 0pt width 1.3ex}}

\newcommand{\argmin}{\mathop{\mathrm{argmin}}\limits}   % ASdeL
\newcommand{\parallelsum}{\mathbin{\!/\mkern-5mu/\!}}

\def\N{\mathbb{N}}
\def\R{\mathbb{R}}

\def\subclassname{{\bfseries Mathematics Subject Classification
(2010)}\enspace}
\def\subclass#1{\par\addvspace\medskipamount{\rightskip=0pt plus1cm
\def\and{\ifhmode\unskip\nobreak\fi\ $\cdot$
}\noindent\subclassname\ignorespaces#1\par}}

\providecommand{\keywords}[1]
{
  \small	
  \textbf{\textit{Keywords---}} #1
}

\begin{document}

\title{Metrics, quantization and registration in varifold spaces}

%    Information for first author
\author{Hsi-Wei Hsieh}
%    Address of record for the research reported here
\address{Department of Applied Mathematics and Statistics, Johns Hopkins University}
%    Current address
%\curraddr{Department of Mathematics and Statistics,
%Case Western Reserve University, Cleveland, Ohio 43403}
\email{hhsieh9@jhu.edu}
%    \thanks will become a 1st page footnote.
%\thanks{NSF grant No 1819131}

%    Information for second author
\author{Nicolas Charon}
\address{Department of Applied Mathematics and Statistics, Johns Hopkins University}
\email{charon@cis.jhu.edu}
%\thanks{NSF grant No 1819131.}

%    General info
%\subjclass[2000]{Primary 54C40, 14E20; Secondary 46E25, 20C20}

\date{}

\dedicatory{}

\begin{abstract}
This paper is concerned with the theory and applications of varifolds to the representation, approximation and diffeomorphic registration of shapes. One of its purpose is to synthesize and extend several prior works which, so far, have made use of this framework mainly in the context of submanifold comparison and matching. In this work, we instead consider deformation models acting on general varifold spaces, which allows to formulate and tackle diffeomorphic registration problems for a much wider class of geometric objects and lead to a more versatile algorithmic pipeline. We study in detail the construction of kernel metrics on varifold spaces and the resulting topological properties of those metrics, then propose a mathematical model for diffeomorphic registration of varifolds under a specific group action which we formulate in the framework of optimal control theory. A second important part of the paper focuses on the discrete aspects. Specifically, we address the problem of optimal finite approximations (quantization) for those metrics and show a $\Gamma$-convergence property for the corresponding registration functionals. Finally, we develop numerical pipelines for quantization and registration before showing a few preliminary results for one and two-dimensional varifolds.    

\subclass{49Q20 \and 49M25 \and 58E50 \and 68T10}
\end{abstract}

\maketitle

%\section*{This is an unnumbered first-level section head}
%This is an example of an unnumbered first-level heading.

%% The correct journal style for \specialsection is all uppercase; a known bug
%% in amsart.cls prevents this, so input must be uppercase until it is fixed.
%\specialsection*{This is a Special Section Head}
%\specialsection*{THIS IS A SPECIAL SECTION HEAD}
%This is an example of a special section head%
%%%%%%%%%%%%%%%%%%%%%%%%%%%%%%%%%%%%%%%%%%%%%%%%%%%%%%%%%%%%%%%%%%%%%%%%
%\footnote{Here is an example of a footnote. Notice that this footnote
%text is running on so that it can stand as an example of how a footnote
%with separate paragraphs should be written.
%\par
%And here is the beginning of the second paragraph.}%
%%%%%%%%%%%%%%%%%%%%%%%%%%%%%%%%%%%%%%%%%%%%%%%%%%%%%%%%%%%%%%%%%%%%%%%%
.

\section{Introduction}
Shape is a bewildering notion: while simultaneously intuitive and ubiquitous to many scientific areas from pure mathematics to biomedicine, it remains very challenging to pin down and analyze in a systematic way. The goal of the research field known as shape/pattern analysis is precisely to provide solid mathematical and algorithmic frameworks for tasks such as automatic comparison or statistical analysis in ensembles of shapes, which is key to many applications in computer vision, speech and motion recognition or computational anatomy, among many others.

What makes shape analysis such a difficult and still largely open problem is, on the one hand, the numerous modalities and types of objects that can fall under this generic notion of shape but also the fundamental nonlinearity that is an almost invariable trait to most of the shape spaces encountered in applications. As a result, the seemingly simple issue of defining and computing distances or means on shapes is arguably a research topic of its own, which has generated countless works spanning several decades and involving concepts from various subdisciplines of mathematics. Among many important works, the model of shape space laid out by Grenander in \cite{Grenander1993} is especially relevant to the present paper. The underlying principle is to build distances between shapes which are induced by metrics on some deformation groups acting on those shapes. This approach has the advantage (at a theoretical level at least) of shifting the problem of metric construction from the many different cases of shape spaces to the single setting of deformation groups. One of the fundamental requirement is the right-invariance of the metrics on those groups; finding the induced distance between two given shapes then reduces to determining a deformation of minimal cost in the group, in other words to solving a \textit{registration} problem.

Besides usual finite-dimensional groups like rigid of affine transformations, there is in fact a lot of practical interest in applying such an approach with groups of "large deformations", specifically groups of \textit{diffeomorphisms}. This has triggered the exploration of right-invariant metrics over diffeomorphism groups. The Large Deformation Diffeomorphic Metric Mapping (LDDMM) model pioneered in \cite{Beg2005,younes2019shapes} is one of such framework that defines Riemannian metrics for diffeomorphic mappings obtained as flows of time-dependent vector fields (c.f. the brief presentation of Section \ref{ssec:diffeom_reg}). In this setting, registering two shapes can be generically formulated as an \textit{optimal control} problem, the functionals to optimize being typically a combination of a deformation regularization term given by the LDDMM metric on the group and a fidelity term that enforces (approximate) matching between the two shape objects. Applications of this model have been widespread in particular within the field of computational anatomy, due to the ability to adapt it to various data structures including landmarks, 2D and 3D images, tensor fields... see e.g. \cite{Miller2009,Miller2015} for recent reviews.

Interestingly, this line of work has also been drawing many useful concepts from the seemingly distant area of mathematics known as \textit{geometric measure theory} \cite{Federer}. The key idea of representing shapes (submanifolds) as measures or distributions has been instrumental in the theoretical study of Plateau's problem on minimal surfaces and more generally in calculus of variations. It can also prove effective for computational purposes, in problems such as discrete curvature approximations \cite{Cohen-Steiner2003,buet2018discretization} or estimation of shape medians \cite{Hu2018}. With regard to the aforementioned deformation analysis problems, the potential interest of geometric measure theory has been identified early on in the works of \cite{Glaunes2004,Glaunes2006}. Indeed, LDDMM registration of objects like geometric curves or surfaces requires fidelity terms independent of the parametrization of either of the two shapes. On the practical side, this means that one cannot usually rely on predefined pointwise correspondences between the vertices of two triangulated surfaces for instance, which makes the registration problem significantly harder than in the case of labelled objects such as landmarks or images. 

The embedding of unparametrized shapes into measure spaces provides one possible way to address the issue, by constructing parametrization-invariant fidelity metrics as restrictions of metrics on those measure spaces themselves. Several competing approaches have been introduced, each relying on embeddings into different spaces of generalized measures: \cite{Glaunes2006,Glaunes2008,Durrleman4} are based on the representation of oriented curves and surfaces as \textit{currents}, \cite{Charon2} and \cite{Charon2017,bauer2019relaxed} extended this model to the setting of unoriented and oriented \textit{varifolds}, while \cite{Roussillon2016,Roussillon2017} considers the higher-order representation of \textit{normal cycles}, see also the recent survey \cite{charon2020fidelity}. One common feature to all those works, however, is that they are focused primarily on registration of curves or surfaces. In other words, the use of current, varifolds or normal cycles confines to the computation of a fidelity metric to guide registration algorithms but the deformation model itself remains tied to the curve/surface setting or equivalently, in the discrete situation, to objects described by point set meshes. 

The guiding theme and main objective of this paper is to investigate an alternative framework that, in contrast with those prior works, would formulate the deformation model as well as tackle the registration problem directly in these generalized measure spaces: we focus specifically on the (oriented) varifold setting of \cite{Charon2017}. There are several arguments for the interest of such an approach %First, in curve or surface registration with LDDMM, deforming the associated varifolds instead of the shape themselves should save the need for constant     
but in our point of view, the primary motivation lies in the fact that, varifolds being more general than submanifolds, the proposed framework allows to extend large deformation analysis methods to a range of new geometric objects while giving more flexibility to deal with some of the flaws which are commonplace in shapes segmented from raw data. As a proof of concept, our recent work \cite{hsieh2019diffeomorphic} considered the simple case of registration of discrete one-dimensional varifolds. Building on these preliminary results, the present paper intends to provide a thorough and general study of the framework.

The specific contributions and organization of this paper are the following. First, we propose a comprehensive study of the class of kernel metrics on varifold spaces initiated in \cite{Charon2,Charon2017}, in particular by examining the required conditions to recover true distances between all varifolds (as opposed to the subset of rectifiable varifolds) and comparing the resulting topologies with some standard metrics on measures. This is presented in Section \ref{sec:metrics_varifold} after the brief introduction to the notion of oriented varifold of Section \ref{sec:varifold_space}. In Section \ref{diff_matching}, we discuss the action of diffeomorphisms and from there derive a formulation of LDDMM registration of general varifolds, for which we show the existence of solutions and derive the Hamiltonian equations associated to the corresponding optimal control problem. Section \ref{sec:approximation_discrete_var} addresses the issue of quantization in varifold space, namely of approximating any varifold as a finite sum of Dirac masses. We consider a novel approach in this context, that consists in computing projections onto particular cones of discrete varifolds. We then prove the $\Gamma$-convergence of the corresponding approximate registration functionals. In Section \ref{sec:numerics}, we derive the discrete version of the optimal control problem and optimality equations, from which we deduce a geodesic shooting algorithm for the diffeomorphic registration of discrete varifolds. Finally, results on $1$- and $2$- varifolds are presented in Section \ref{sec:results}, emphasizing the potentiality of the approach to tackle data structures which are typically challenging for previous algorithms that are designed for point sets and meshes. The MATLAB code used in this paper is made available under the GNU general public license through the Github repository \url{https://github.com/charoncode/Var_LDDMM}.

\section{The space of oriented varifolds}
\label{sec:varifold_space}
The concept of varifold was originally developed in the context of geometric measure theory by \cite{Young42}, \cite{almgren1966plateau} and \cite{allard1972first} for the study of Plateau's problem on minimal surfaces. The interest in registration and shape analysis was evidenced in \cite{Charon2,Charon2017}. In those works, varifolds provide a convenient representation of geometric shapes such as rectifiable curves and surfaces and an efficient approach to define and compute fidelity terms for registration, or to perform clustering, classification in those shape spaces. The main purpose of this section is to introduce varifolds in this latter context. The case of non-oriented shapes was thoroughly investigated in \cite{Charon2}. Later on, the generalized framework of oriented varifold was proposed in \cite{Charon2017} but only for objects of dimension or co-dimension one. In the following, we provide a fully general presentation of oriented varifolds and their properties, that also does not specifically focus on the case of rectifiable varifolds as these previous works did. Although we assume here that all the considered shapes are oriented, we emphasize that the non-oriented framework of \cite{Charon2} can be recovered almost straightforwardly through adequate choices of orientation-invariant kernels as we shall briefly point out later on.

\subsection{Definition} 
\label{subsec:defofvf}
The underlying principle of varifolds is to extend measures of $\R^n$ by incorporating an additional tangent space component. In this work, we will consider such spaces to be oriented. Thus, for a given dimension $0\leq d \leq n$, we first need to introduce the set of all possible $d$-dimensional oriented tangent spaces in $\mathbb{R}^n$:
\begin{definition}
 The $d$-dimensional oriented Grassmannian $\widetilde{G}_d^n$ is the set of all oriented $d$-dimensional linear subspaces of $\mathbb{R}^n$.
 %\begin{align*}
 %    \widetilde{G}_d^n := \{(T,\mathcal{O}): T \in G^n_d, \ \mathcal{O} = \textrm{orientation}\}
 %\end{align*}
 %which is set of all oriented $d$-dimensional subspaces in $\mathbb{R}^n$. 
\end{definition}
The oriented Grassmannian is a compact manifold of dimension $d(n-d)$ which can be identified to the quotient $SO(n)/(SO(d) \times SO(n-d))$. It is also a double cover of the (non-oriented) Grassmannian $G_d^n$ of $d$-dimensional subspaces of $\R^n$. For practical purposes, a more convenient representation of $\widetilde{G}_d^n$ is the one detailed in the following remark.
\begin{remark}
\label{rem:Grassmannian}
%Given $(T,\mathcal{O}) \in \widetilde{G}_d^n$, there exists a basis $\{u_i \}_{i=1}^d$ of $T$ such that  %$[u_1,\cdots,u_d] = \mathcal{O}$. Then the following map is well defined and injective,
Given $T \in \widetilde{G}_d^n$, there exists a basis $\{u_i \}_{i=1,\ldots,d} \in \R^{n\times d}$ of $T$ such that $[u_1,\cdots,u_d]$ has consistent orientation with $T$. Then the following map, called the oriented Pl\"{u}cker embedding, is well defined and injective,
\begin{align*}
    i_P: \ \widetilde{G}_d^n &\mapsto \{\xi \in \Lambda^d(\mathbb{R}^n) : |\xi| =1\} \\
    T &\mapsto \frac{u_1 \wedge \cdots \wedge u_d}{|u_1 \wedge \cdots \wedge u_d |}.
\end{align*}
This allows to identify $\widetilde{G}^n_d$ as a subset of the unit sphere of $\Lambda^d(\mathbb{R}^n)$ which inherits the topology of the inner product on $\Lambda^d(\mathbb{R}^n)$. We remind that this inner product is defined for any $\xi=\xi_1 \wedge \ldots \wedge \xi_d$, $\eta=\eta_1 \wedge \ldots \wedge \eta_d$ in $\Lambda^d(\mathbb{R}^n)$ by the determinant of the Gram matrix: 
\begin{align}\label{eq:wedge_inner_product}
    \langle \xi,\eta \rangle = \det(\xi_i \cdot \eta_j)_{i,j=1,\ldots,d} 
\end{align}
Through this identification, one can also define the action of linear transformations on $\widetilde{G}^n_d$ as follows 
\begin{align} \label{eq:act_lt_grass}
    A \cdot T := \frac{A u_1 \wedge \cdots \wedge Au_d}{|A u_1 \wedge \cdots \wedge Au_d |}
\end{align}
for any $T \in \widetilde{G}^n_d$ and $A: \mathbb{R}^n \mapsto \mathbb{R}^n$ a linear invertible map.
\end{remark}

Similar to the definition of classical varifolds in \cite{Simon}, we define oriented varifolds as measures on $\mathbb{R}^n \times \widetilde{G}^n_d$. 

\begin{definition}
 An oriented $d$-varifold $\mu$ on $\mathbb{R}^n$ is a nonnegative finite Radon measure on the space $\mathbb{R}^n \times \widetilde{G}^n_d$. Its weight measure $|\mu|$ is defined by $|\mu|(A) := \mu(A\times \widetilde{G}^n_d)$ for all Borel subset $A$ of $\R^n$. We denote by $\mathcal{V}_d$ the space of all oriented $d$-varifolds.
\end{definition}
In the rest of the paper, with a slight abuse of vocabulary, we will often use the word varifold instead of oriented varifold for the sake of concision. Recall that from the Riesz representation theorem, we can alternatively view any varifold $\mu$ as a distribution, i.e. an element of the dual space $C_0(\mathbb{R}^d \times \widetilde{G}^n_d)^*$, where $C_0(\mathbb{R}^d \times \widetilde{G}^n_d)$ denotes the set of continuous functions vanishing at infinity on $\mathbb{R}^d \times \widetilde{G}^n_d$. It is defined for any test function $\omega \in C_0(\mathbb{R}^d \times \widetilde{G}^n_d)$ by:
\begin{equation}
\label{eq:var_mu_distribution}
    (\mu |\omega) \doteq \int_{\mathbb{R}^n \times \widetilde{G}^n_d} \omega(x,T) d\mu(x,T).
\end{equation}
As an additional note, another useful representation of a general varifold in $\mathcal{V}_d$ can be obtained by the disintegration theorem (see \cite{Ambrosio2000} Chap. 2). Namely, if $\mu \in \mathcal{V}_d$, for $|\mu|$-almost every $x$ in $\R^n$, there exists a probability measure $\nu_x$ on $\widetilde{G}^n_d$ such that $x \mapsto \nu_x$ is $|\mu|$-measurable and we can write
\begin{equation}
\label{eq:var_mu_disintegration}
    (\mu |\omega) = \int_{\R^n} \int_{\widetilde{G}^n_d} \omega(x,T) d\nu_x(T) d|\mu|(x).
\end{equation}
In other words, the varifold $\mu$ can be decomposed as its weight measure on $\R^n$ together with a family of tangent space probability measures on the Grassmannian at the different points in the support of $|\mu|$. This is usually referred to as the Young measure representation of $\mu$.

\subsection{Diracs and rectifiable varifolds}
There are a few important families of varifolds which will be relevant for the following. First of those are the Diracs. For $x \in \mathbb{R}^n$ and $T \in \widetilde{G}^n_d$, the associated Dirac varifold $\delta_{(x,T)}$ acts on functions of $C_0(\mathbb{R}^n \times \widetilde{G}^n_d)$ by the relation
\begin{align*}
    (\delta_{(x,T)}|\omega) = \omega(x,T), \ \forall \omega \in C_0(\mathbb{R}^n \times \widetilde{G}^n_d).
\end{align*}
$\delta_{(x,T)}$ can be viewed as a singular particle at position $x$ that carries the oriented $d$-plane $T$.

A second particular class is the one of \textit{rectifiable varifolds}, which are in essence the varifolds representing an oriented shape of dimension $d$. More precisely, given an oriented $d$-dimensional submanifold $X$ of $\mathbb{R}^n$ of finite total $d$-volume, denoting by $T_X(x) \in \widetilde{G}^n_d$ the oriented tangent space at $x \in X$, one can associate to $X$ the varifold $\mu_X$, which is defined for all Borel subset $B\subset \R^n \times \widetilde{G}^n_d$ by $\mu_X(B) = \mathcal{H}^d(\{x \in X | (x,T_X(x)) \in B \})$. Here, $\mathcal{H}^d$ is the $d$-dimensional Hausdorff measure on $\R^n$, i.e. the measure of $d$-volume of subsets of $\R^n$ (we refer the reader to \cite{Simon} for the precise construction and properties of Hausdorff measures). It is then not hard to see that, as an element of $C_0(\mathbb{R}^n \times \widetilde{G}^n_d)^*$,  
\begin{align}\label{eq:vfasdistri}
(\mu_X|\omega) &= \int_{\mathbb{R}^d \times \widetilde{G}^n_d} \omega(x,T) d \mu_X(x,T) \nonumber \\
               &= \int_X \omega(x,T_X(x)) d \mathcal{H}^d(x).
\end{align}
Such a representation $X \mapsto \mu_X$ can be extended to slightly more general objects known as \textit{oriented rectifiable sets}. A subset $X$ of $\mathbb{R}^n$ is said to be a \textit{countably} $\mathcal{H}^d$-\textit{rectifiable set} if $\mathcal{H}^d(X \setminus \cup_{j=1}^{\infty}F_j(\mathbb{R}^d)) = 0$, where $F_j: \mathbb{R}^d \mapsto \mathbb{R}^n$ are Lipschitz function for all $j$ (c.f. \cite{Simon}). We say that $(X,T_X)$ is an \textit{oriented rectifiable set} if $X$ is a countably $d$-rectifiable set and $T_X:X \mapsto \widetilde{G}^n_d$ is a $\mathcal{H}^d$-measurable function such that for $\mathcal{H}^d-a.e.\ x \in X$, $T_X(x)$ is the approximate tangent space of $X$ at $x$ with specified orientation. Rectifiable subsets include both usual submanifolds but also piecewise smooth objects like polyhedra. Given any oriented rectifiable set $(X, T_X)$, we can associate a varifold that we also write $\mu_X$ given again by \eqref{eq:vfasdistri}. The set of those $\mu_X$ will be referred to as the rectifiable oriented varifolds in this paper (note that this is actually more restrictive than the standard definition of rectifiable varifold in the literature which also incorporates an additional multiplicity function). 
\begin{remark}
Rectifiable varifolds still make a very "small" subset of $\mathcal{V}_d$: indeed, in the Young measure representation of \eqref{eq:var_mu_disintegration}, we have in this case the very particular constraint that probability measures $\nu_x$ are Dirac masses, specifically $\nu_x = \delta_{T_X(x)}$. 
\end{remark}

\section{Metrics on varifolds}
\label{sec:metrics_varifold}
In this section, we address the issue of defining adequate metrics on the space $\mathcal{V}_d$. After reviewing some classical metrics and their limitations for the specific applications of this work, we turn to metrics defined through positive definite kernels, for which we extend previous constructions introduced in e.g. \cite{Charon2,Charon2017} and derive the most relevant properties of this class of distances.

\subsection{Standard topologies and metrics on $\mathcal{V}_d$}
\label{subsec:standard_metrics}
As a measure/distribution space, $\mathcal{V}_d$ can be equipped with various topologies and metrics, several of which have been regularly used in various contexts. We discuss a few of those below. 
\begin{itemize}
   \item[$\bullet$]   
\textit{mass norm}: with the previous identification of measures in $\mathcal{V}_d$ with elements of the dual $C_0(\mathbb{R}^n \times \widetilde{G}^n_d)^*$, one can define the following dual metric on $\mathcal{V}_d$:
\begin{align}
   d_{op}(\mu,\nu) \doteq \sup\limits_{|\omega|_{\infty} \leq 1} (\mu -\nu|\omega), \ \forall \mu \in \mathcal{V}_d.
\end{align}
where $|\omega|_{\infty} \doteq \sup_{\mathbb{R}^n \times \widetilde{G}^n_d} |\omega|$. This metric is generally too strong for applications in shape analysis and leads to a discontinuous behavior. Indeed, one can easily verify that for any two Dirac masses $\delta_{(x,T)}$ and $\delta_{(x',T')}$, $d_{op}(\delta_{(x,T)},\delta_{(x',T')}) = 2$ whenever $(x,T) \neq (x',T')$.   

\item[$\bullet$]
\textit{weak-*} topology: a sequence of $d$-varifolds $\{\mu_i\}_i$ converges to $\mu \in \mathcal{V}_d$ in the weak-* topology (denoted by $\mu_i \overset{\ast}{\rightharpoonup} \mu$) if and only if for all $\omega \in C_c(\mathbb{R}^d \times \widetilde{G}^n_d)$ (continuous compactly supported function) 
\begin{align}
    \lim_{i \rightarrow \infty}(\mu_i|\omega) = (\mu|\omega).
\end{align}
In fact, the weak-* topology on $\mathcal{V}_d$ can be metrized by the following distance:
\begin{align*}
    d_*(\mu,\nu) = \sum_{k \in \mathbb{N}} 2^{-k} |(\mu-\nu|\omega_k)|,
\end{align*}
where $\{\omega_k\}_{k\in \mathbb{N}}$ is a dense sequence in $C_c(\mathbb{R}^n \times \widetilde{G}^n_d)$.

\item[$\bullet$]
\textit{Wasserstein metric}: the Wasserstein-1 distance of optimal transport can be expressed in its Kantorovitch dual formulation \cite{villani2008optimal} as
\begin{equation}
    d_{Wass^1}(\mu,\nu) \doteq \sup\limits_{\textrm{Lip}(\omega) \leq 1} |(\mu-\nu|\omega)|.
\end{equation}
where the sup is taken over all Lipschitz regular functions on $\mathbb{R}^n \times \widetilde{G}^n_d$ with Lipschitz constant smaller than one. This metric is however well-suited for measures with the same total mass. Several recent works \cite{Piccoli2014,Chizat2018} have instead proposed generalized Wasserstein distances derived from unbalanced optimal transport. 

\item[$\bullet$]
\textit{Bounded Lipschitz metric}: similar to the previous, the bounded Lipschitz distance (sometimes referred to as the \textit{flat metric}) on $\mathcal{V}_d$ is defined by
\begin{equation}
    d_{BL}(\mu,\nu) \doteq \sup\limits_{\|\omega\|_{\infty}, \textrm{Lip}(\omega) \leq 1} |(\mu-\nu|\omega)|.
\end{equation}
It can be shown (cf. Ch 8 in \cite{bogachev2007measure}) that $d_{BL}$ metrizes the \textit{narrow topology} on $\mathcal{V}_d$, namely the topology for which a sequence $(\mu_i)$ converges to $\mu$ if and only if $\lim_{i \rightarrow \infty}(\mu_i|\omega) = (\mu|\omega)$ for all bounded continuous functions $\omega$.
\end{itemize}

Clearly, the narrow topology is stronger than the weak-* topology. Furthermore, it is also well known that $d_{BL}$ locally metrizes the weak-* topology on $\mathcal{V}_d$, namely:
\begin{prop}\label{prop:dBL_weakstar}
Let $\mu$ and $\{\mu_i\}_{i}$ be varifolds such that the sequence $\{\mu_i\}_{i}$ is tight. Then $\mu_i \overset{\ast}{\rightharpoonup} \mu$ if and only if $d_{BL}(\mu_i,\mu) \rightarrow 0$.
\end{prop}
\begin{proof}
Since $d_{BL}$ metrizes the narrow topology, it suffices to show that $\mu_i$ converges to $\mu$ in the narrow topology. Let $\omega$ be a bounded continuous function defined on $\mathbb{R}^n \times \widetilde{G}^n_d$ and $\varepsilon>0$. By the tightness property, we may choose a compact set $K\subset \mathbb{R}^n \times \widetilde{G}^n_d$ such that $\mu(K^c)+ \sup_i \mu_i(K^c)< \varepsilon/2\|\omega\|_{\infty}$. Let $B$ be an open ball that contains $K$. Define
\begin{align*}
    \eta(x,T)  \doteq  \left\{ \begin{array}{lc}
      \omega(x,T),  &  \ \text{if} \ (x,T) \in K  \\
      0 ,    & \ \text{if} \ (x,T) \in B^c
    \end{array} \right. 
\end{align*}
From Tietz extension theorem, there exists a continuous extension $\tilde{\omega}$ of $\eta$ on $\mathbb{R}^n \times \widetilde{G}^n_d$ such that $\tilde{\omega}|_K =\omega|_K$ and $\tilde{\omega} \in C_c(\mathbb{R}^n \times \widetilde{G}^n_d)$. This implies that
\begin{align*}
    \left| \int_{\mathbb{R}^n \times \widetilde{G}^n_d} \omega d(\mu_i-\mu) \right| 
  &\leq \int_{\mathbb{R}^n \times \widetilde{G}^n_d}|\omega- \tilde{\omega}| d (\mu_i+\mu)
    + \left| \int_{\mathbb{R}^n \times \widetilde{G}^n_d} \tilde{\omega} d (\mu_i-\mu) \right|\\
    &\leq  \left| \int_{\mathbb{R}^n \times \widetilde{G}^n_d} \tilde{\omega} d (\mu_i-\mu) \right| + \varepsilon.
\end{align*}
Taking $\limsup$ on both sides, we see that 
\begin{align*}
    \limsup_{i \rightarrow \infty} \left| \int_{\mathbb{R}^n \times \widetilde{G}^n_d} \omega d(\mu_i-\mu) \right| < \varepsilon.
\end{align*}
Since $\varepsilon$ is arbitrary, we obtain that $\mu_i$ converges to $\mu$ in the narrow topology. 
\end{proof}
As a direct consequence of Proposition \ref{prop:dBL_weakstar}, we have in particular that weak-* convergence and convergence in $d_{BL}$ are equivalent if one restricts to varifolds that are supported in a fixed compact subset of $\mathbb{R}^n \times \widetilde{G}^n_d$. Note also that a very similar result to Proposition \ref{prop:dBL_weakstar} holds when replacing the bounded Lipschitz distance by generalized Wasserstein metrics, as proved in \cite{Piccoli2014}. 

The above metrics on varifolds all originate from classical ones in standard measure theory. Unlike the mass norm, Wasserstein and bounded Lipschitz metrics have nice theoretical properties in terms of shape comparison. However, for the purpose of diffeomorphic registration that we shall tackle below, one needs metrics that are easy to evaluate numerically. This is typically not the case of $d_{Wass^1}$ and $d_{BL}$ expressed above as there is no straightforward way to compute the corresponding suprema over the respective sets of test functions. One line of work has been considering approximations of optimal transport distances with e.g. entropic regularizers for which Sinkhorn-based algorithms can be derived, see for instance the recent work \cite{Feydy2017}. In this paper, we focus on the alternative approach previously developed for currents in \cite{Glaunes2008} and unoriented varifolds in \cite{Charon2} which instead relies on particular Hilbert spaces of test functions, as we detail in the next section.

\subsection{Kernel metrics}
\label{ssec:kernel_metrics}
In this section, we start by defining a general class of pseudo-metrics on $\mathcal{V}_d$ based on positive definite kernels and their corresponding \textit{reproducing kernel Hilbert space} (RKHS). We will then study sufficient conditions on such kernels to recover true metrics before examining the relationship between those kernel metrics and the ones of Section \ref{subsec:standard_metrics}.

\subsubsection{Kernels for varifolds}
We refer the reader to \cite{Aronszajn1950,Hastie,Micheli2014} for a presentation of the construction and main properties of positive kernels and Reproducing Kernel Hilbert Spaces which we do not recall in detail here for the sake of concision. In the context of varifolds, we are interested in defining positive definite kernels on the product $\mathbb{R}^n \times \widetilde{G}^n_d$. Along the lines of previous works like \cite{Charon2,Charon2017}, we restrict to separable kernels for which we have:

\begin{prop} \label{prop:ctru_kernel}
Let $k^{pos}$ and $k^{G}$ be continuous positive definite kernels on $\mathbb{R}^n$ and $\widetilde{G}^n_d$ respectively. Assume in addition that for any $x \in \mathbb{R}^n $, $k^{pos}(x,\cdot) \in C_0(\mathbb{R}^n)$. Then $k:= k^{pos} \otimes k^G$ is a positive definite kernel on $\mathbb{R}^n \times \widetilde{G}^n_d$ and the RKHS $W$ associated to $k$ is continuously embedded in $C_0(\mathbb{R}^n \times \widetilde{G}^n_d)$ i.e. there exists $c_{W}>0$ such that for any $\omega \in W$, we have $\|\omega\|_{\infty} \leq c_W \|\omega\|_{W}$.
\end{prop}
We recall that the tensor product kernel has the exact expression $k((x,T),(x',T')) = k^{pos}(x,x') k^{G}(T,T')$. The proof of Proposition \ref{prop:ctru_kernel} is a straightforward adaptation of the same result for unoriented varifolds (cf. \cite{Charon2} Proposition 4.1).  

\begin{remark}\label{rk:assumption_kernel}
To simplify the rest of the presentation and in the perspective of later numerical considerations, we will also assume specific forms for $k^{pos}$ and $k^G$, namely that $k^{pos}$ is a translation/rotation invariant radial kernel $k^{pos}(x,y) = \rho(|x-y|^2), \ \forall x,y \in \mathbb{R}^n$, with $\rho(0)>0$, and $k^G(S,T) = \gamma(\langle S,T \rangle), \ \forall S,T \in \widetilde{G}^n_d$ where $\langle \cdot, \cdot \rangle$ is the inner product on $\widetilde{G}^n_d$ inherited from $\Lambda^d(\mathbb{R}^n)$ introduced in remark \ref{rem:Grassmannian}. These assumptions are quite natural as they will eventually induce metrics on varifolds invariant to the action of rigid motion, as we shall explain later. Note that the unoriented framework of \cite{Charon2} can be also recovered in this setting by simply restricting to orientation-invariant kernels $k^{G}$ i.e. such that $\gamma(-t)=\gamma(t)$ for all $t$.    
\end{remark}

Now, if we let $\iota_W:W \hookrightarrow C_0(\mathbb{R}^d \times \widetilde{G}^n_d)$ be the continuous embedding given by Proposition \ref{prop:ctru_kernel} and $\iota_W^*$ its adjoint, for any $\mu \in C_0(\mathbb{R}^n \times \widetilde{G}^n_d)^*$, we have 
\begin{align} \label{eq:vf_as_wstar}
    (\iota^* \mu| \omega) = \int_{\mathbb{R}^d \times \widetilde{G}^n_d} \omega(x,T) d\mu(x,T), \ \forall \omega \in W.
\end{align}
With \eqref{eq:vf_as_wstar}, we may identify $\mu$ as an element of the dual RKHS $W^*$. Note that $\iota_W^*$ is not injective in general, in other words one can have $\mu = \mu'$ in $W^*$ but $\mu \neq \mu'$ in $C_0(\mathbb{R}^n \times \widetilde{G}^n_d)^*$.

In any case, one can compare any two varifolds $\mu,\mu' \in \mathcal{V}_d$ through the Hilbert norm of $W^*$ by defining: 
\begin{equation}
    d_{W^*}(\mu,\mu')^2 = \|\mu-\mu'\|_{W^*}^2 = \|\mu\|_{W^*}^2 -2\langle \mu,\mu' \rangle_{W^*} + \|\mu'\|_{W^*}^2 
\end{equation}
where we use the small abuse of notation of writing $\mu$ and $\mu'$ instead of $\iota_W^* \mu$ and $\iota_W^* \mu'$ on the two right hand sides. Due to the potential non-injectivity of $\iota_W^*$, in general $d_{W^*}$ only induces a pseudo-metric on $\mathcal{V}_d$. 

The main advantage of this construction is that $d_{W^*}$ can be expressed more explicitly based on the reproducing kernel property of $W$. Indeed, given any $\mu$ and $\nu$ in $\mathcal{V}_d$, the inner product between them is given by 
\begin{align} \label{eq:rkp2}
    \langle \mu, \mu' \rangle_{W^*} &= \int_{(\mathbb{R}^d \times \widetilde{G}^n_d)^2} k^{pos}(x,x') k^G(T,T') d\mu(x,T) d \mu'(x',T') \nonumber \\
    &= \int_{(\mathbb{R}^d \times \widetilde{G}^n_d)^2} \rho(|x-x'|^2) \gamma(\langle T,T' \rangle) d\mu(x,T) d \mu'(x',T')
\end{align}
for kernels selected as in Remark \ref{rk:assumption_kernel}.

\subsubsection{Characterization of distances}
As mentioned above, $d_{W^*}$ is a priori a pseudo-distance between varifolds. It's a natural question to ask under which conditions it leads to an actual distance. 

Most past works have addressed this question focusing on the case of varifolds representing submanifolds and reunion of submanifolds \cite{Charon2,Charon2017}. We can first provide an extension of these results to the general case of oriented rectifiable varifolds. A key notion for the rest of this section is the one of $C_0$-universality of kernels:
\begin{definition}
A positive definite kernel $k$ on a metric space $\mathcal{M}$ is called $C_0$-universal when its RKHS is dense in $C_0(\mathcal{M})$ for the uniform convergence topology.  
\end{definition}
$C_0$-universality has been studied in great length in such works as \cite{Carmeli2010,sriperumbudur2011universality}. In particular, one can provide characterizations of $C_0$-universality for certain classes of kernels and spaces $\mathcal{M}$. In the case of translation-invariant kernels on $\mathcal{M}=\R^n$ for instance, it has been established that $C_0$-universal kernels are the ones which can be expressed through the Fourier transform of finite Borel measures with full support on $\R^n$, which includes: compactly-supported kernels, Gaussian kernels, Laplacian kernels... With the previous definition, we have the following sufficient condition: 
\begin{theorem}
\label{thm:dist_rectifiable_var}
Suppose $k^{pos}$ is a $C_0$-universal kernel on $\R^n$, $\gamma(1)>0$ and $\gamma(t) \neq \gamma(-t), \ \forall t \in [-1,1]$. Let $(X,T(\cdot))$ and $(Y,S(\cdot))$ be two oriented $\mathcal{H}^d$-rectifiable sets with $\mathcal{H}^d(X)$, $\mathcal{H}^d(Y)< \infty$. If $\left\| \mu_X-\mu_Y \right\|_{W'}=0$, then $\mathcal{H}^d(X\bigtriangleup Y) =0$ and $T = S \ \mathcal{H}^d$-$a.e$. 
\end{theorem}
The full proof can be found in the Appendix. Note that the first part of the proof directly gives an equivalent statement for unoriented rectifiable varifolds (if one instead assumes $\gamma(t) = \gamma(-t)$ for all $t$), generalizing the result of \cite{Charon2}.

However, the previous proposition does not necessarily lead to a distance on the full space $\mathcal{V}_d$. Counter-examples in the case $d=1$ are discussed for example in \cite{hsieh2019diffeomorphic}. To recover a true distance on $\mathcal{V}_d$, one needs the previous map $\iota_W^{*}$ or equivalently the map
\begin{align}\label{eq:featuremap}
    \mu \mapsto \int_{\mathbb{R}^d \times \widetilde{G}^n_d} k(\cdot,(y,T)) d \mu(y,T), \ \mu \in C_0(\mathbb{R}^d \times \widetilde{G}^n_d)^*
\end{align}
to be injective. As follows from Theorem 6 in \cite{sriperumbudur2011universality}, this is in fact guaranteed when the kernel $k$ on the product space $\mathbb{R}^n \times \widetilde{G}^n_d$ is $C_0$-universal, specifically
\begin{theorem}
\label{thm:dist_general_var}
The pseudo-distance $d_{W^*}$ induces a distance between signed measures of $\mathbb{R}^n \times \widetilde{G}^n_d$ if and only if $k$ is $C_0$-universal on $\mathbb{R}^n \times \widetilde{G}^n_d$. In particular, a sufficient condition for $d_{W^*}$ to be a distance on $\mathcal{V}_d$ is that $k^{pos}$ and $k^G$ are $C_0$-universal kernels on $\R^n$ and $\widetilde{G}^n_d$ respectively.
\end{theorem}
Note that these conditions are more restrictive than in Theorem \ref{thm:dist_rectifiable_var}. To our knowledge, there is no simple characterization for general $C_0$-universal kernels on the Grassmannian. However, within the setting of Remark \ref{rk:assumption_kernel}, one easily constructs $C_0$-universal kernels by restriction (based on the Pl\"{u}cker embedding) of $C_0$-universal kernels defined on the vector space $\Lambda^d(\R^n)$.

\subsubsection{Comparison with classical metrics}
We now study more precisely the topology induced by the (pseudo) distance $d_{W^*}$ on $\mathcal{V}_d$ in comparison with the ones defined in Section \ref{subsec:standard_metrics}. First of all, we observe that, for any $\omega \in W$ with $\|\omega\|_{W} \leq 1$, one must have $\| \omega \|_{\infty} \leq c_W$, where $c_W$ is the embedding constant of Proposition \ref{prop:ctru_kernel}. Thus, for any $\mu$ and $\mu'$ in $\mathcal{V}_d$, we have
\begin{equation} 
\label{eq:dbl_wstar}
    \| \mu-\mu'\|_{W^*} = \sup_{\omega \in W,\ \|\omega\|_{W \leq 1}} \int_{\mathbb{R}^d \times \widetilde{G}^n_d } \omega \ d(\mu-\mu') \leq c_W d_{op}(\mu,\mu').
\end{equation}
From the above inequalities we see that convergence in $d_{op}$ implies convergence in $d_{W^*}$.

\begin{remark}
\label{rmk:dom_W_BL}
With more assumptions on the regularity of the kernel $k$, namely if $W$ is continuously embedded in $C_0^1(\mathbb{R}^d \times \widetilde{G}^n_d)$, following a similar reasoning as above, one obtains the bound $\| \mu-\mu'\|_{W^*} \leq c_{W} d_{BL}(\mu,\mu')$.  
\end{remark}

Suppose $\mu_i$  converges to $\mu$ in narrow topology. Since the map $(\nu_1,\nu_2) \mapsto \nu_1 \otimes \nu_2$ is continuous with respect to the narrow topology, we have
\begin{align*}
    \|\mu_i\|_{W^*}^2 &= \int_{(\mathbb{R}^d \times \widetilde{G}^n_d)^2} k((x,S),(y,T)) d\mu_i(x,S) d \mu_i(y,T) \\
    &\rightarrow \int_{(\mathbb{R}^d \times \widetilde{G}^n_d)^2} k((x,S),(y,T)) d\mu(x,S) d \mu(y,T) \\
    &= \|\mu\|_{W^*}^2,
\end{align*}
as $i \rightarrow \infty$.
Also, it's clear that $\lim_{i \rightarrow \infty}\langle \mu_i, \mu \rangle_{W^*} \rightarrow \|\mu\|_{W^*}^2$ and hence $\mu_i \rightarrow \mu$ with respect to $d_{W^*}$. To summarize the discussion above:

\begin{prop}\label{prop:top_comp}
Let $\{\mu_i\}_i$ and $\mu$ be varifolds in $\mathcal{V}_{d}$ and assume that $\mu_i \rightarrow \mu$ with respect to the operator norm or the narrow topology, then $\mu_i \rightarrow \mu$ in $W^*$.
\end{prop}
\begin{remark}
We emphasize that the result of Proposition \ref{prop:top_comp} only requires the assumptions of Proposition \ref{prop:ctru_kernel} and thus holds whether $\iota$ is injective or not. 
\end{remark}

As for the weak-* topology, with the $C_0$-universality assumption of Theorem \ref{thm:dist_general_var} and restricting to varifolds with bounded total mass, we show that $d_{W^*}$ induces a topology stronger than weak-* convergence:

\begin{prop}\label{prop:dW_finer_weakstar}
If $k$ is $C_0$-universal, then the topology induced by $d_{W^*}$ is finer than the weak-* topology on $\mathcal{V}_{d,M} \doteq \{\mu \in \mathcal{V}_d \ \text{s.t} \ |\mu|(\mathbb{R}^n) \leq M \}$ for any fixed $M>0$.
%Moreover, if $\rho(0),\gamma(1)>0$, then the topology induced by $d_{W^*}$ is strictly finer.
\end{prop}

\begin{proof}
 Let $\{\mu_i\}_i$ and $\mu$ be varifolds in $\mathcal{V}_{d,M}$ and assume that $\lim_{i \rightarrow \infty}d_{W^*}(\mu_i,\mu) = 0$. For any $f \in C_0(\mathbb{R}^d \times \widetilde{G}^n_d)$ and $\varepsilon>0$, there exists a $g \in W$ such that $\|g-f\| < \varepsilon/2M$. Then we obtain that $\mu_i \overset{\ast}{\rightharpoonup} \mu$ from the following inequalities:
\begin{equation*}
    |(\mu_i-\mu|f)| \leq |(\mu_i| f-g)| + |(\mu| g-f)| + |(\mu_i - \mu|g)| \leq  \varepsilon + \|\mu_i -\mu \|_{W^*} \|g\|_W.
\end{equation*}

\end{proof}

Note that the topology induced by $d_{W^*}$ may be strictly finer on $\mathcal{V}_{d,M}$. Indeed, if $\rho(0),\gamma(1)>0$, consider $\mu_i = \delta_{(x_i,S)}$, where $\lim_{i \rightarrow \infty}|x_i| = \infty$ and $S\in \widetilde{G}^n_d$ fixed. Then $\mu_i \overset{\ast}{\rightharpoonup} 0$ while $\|\mu_i\|_{W^*}^2 =  \rho(0) \gamma(1) > 0$  for all $i$. Yet, by combining Propositions \ref{prop:dBL_weakstar}, \ref{prop:top_comp} and \ref{prop:dW_finer_weakstar}, we have the following 

\begin{corollary}
Let $M>0$ and $K \subset \R^n \times \widetilde{G}^n_d$ be a compact subset. If $k$ is $C_0$-universal, then $d_{W^*}$ metrizes the weak-* convergence of varifolds on $\mathcal{V}_{d,M,K} \doteq \{\mu \in \mathcal{V}_d \ \text{s.t} \ |\mu|(\mathbb{R}^n) \leq M, \ \text{supp}(\mu)\subset K \}.$
\end{corollary}

In summary, $C_0$-universality provides a sufficient condition to obtain actual distances between varifolds that can be expressed based on the kernel function. Furthermore, the resulting topology is locally equivalent to the weak-* topology as well as the topology induced by the bounded Lipschitz distance. This equivalence will be of importance in Section \ref{sec:approximation_discrete_var}.

\section{Deformation and registration of varifolds} 
\label{diff_matching}
Having defined a way of comparing general varifolds through the above kernel metrics $d_{W^*}$, our goal is now to focus on deformation models for those objects in order to formulate and study the diffeomorphic registration problem on $\mathcal{V}_d$. 

\subsection{Deformation models}
\label{ssec:deformation_models}
In this section, we discuss different models for how varifolds can be transported by a diffeomorphism of $\R^n$, in other words what are possible group actions of the diffeomorphism group $\textrm{Diff}(\R^n)$ on $\mathcal{V}_d$. 

Let us start by considering the case of an oriented rectifiable subset $(X,T_X)$. A diffeomorphism $\phi \in \textrm{Diff}(\mathbb{R}^n)$ transports $(X,T_X)$ as
\begin{align*}
    \phi \cdot (X,T_X) \doteq (\phi(X), T_{\phi(X)}),
\end{align*}
where the transported orientation map writes
\begin{align*}
    T_{\phi(X)}(y) \doteq d_{\phi^{-1}(y)} \phi \cdot T_X(\phi^{-1}(y))
\end{align*}
the above term being well-defined from \eqref{eq:act_lt_grass}. This suggests introducing the following \textit{pushforward action} on $\mathcal{V}_d$, which is defined for all $\mu \in \mathcal{V}_d$ and $\phi \in \textrm{Diff}(\mathbb{R}^n)$ by:
\begin{equation}
\label{eq:def_pushforward}
    (\phi_{\#} \mu | \omega) \doteq \int_{\mathbb{R}^d \times \widetilde{G}^n_d} \omega(\phi(x),d_{x} \phi \cdot T) J_T \phi(x) d \mu(x,T)
\end{equation}
in which $J_T \phi(x)$ denotes the determinant of the Jacobian of $\phi$ along $T$ (i.e. the change of d-volume induced by $\phi$ along $T$ at $x$) which is given by 
\begin{equation*}
  J_T \phi(x) = \det\left((d_x\phi(e_i) \cdot d_x\phi(e_j))_{i,j=1,\ldots,d} \right)   
\end{equation*}
for $(e_1,\ldots,e_d)$ an orthonormal basis of $T$. One easily verifies that $(\phi,\mu) \mapsto \phi_{\#} \mu$ defines a group action which commutes with the action on oriented rectifiable sets, namely
\begin{prop}
\label{prop:pushforward_rectifiable}
 For any oriented rectifiable set $(X,T_X)$ and diffeomorphism $\phi \in \textrm{Diff}(\mathbb{R}^n)$,\ $\phi_{\#} \mu_X = \mu_{\phi(X)}$.
\end{prop}
This follows from the area formula for integrals over rectifiable sets, c.f. \cite{Simon} Chapter 2. 
\begin{remark}
This pushforward action also extends the diffeomorphic transport of measures with densities on $\R^n$. Indeed if $\mu=\theta(x). \mathcal{L}^n$ with $\theta$ a measurable density function on $\R^n$ and $\mathcal{L}^n$ the Lebesgue measure, we can extend $\mu$ to a n-varifold in $\mathcal{V}_n$ by taking a constant global orientation in $\widetilde{G}_n^n = \{\pm 1\}$. Then, for any orientation-preserving diffeomorphism $\phi$, \eqref{eq:def_pushforward} writes in this case: $\phi_{\#} \mu = |J\phi(x)| \theta(x) . \mathcal{L}^n$ with $J\phi(x)$ is the full Jacobian determinant of $\phi$, leading to the usual action on densities $\phi \cdot \theta(x) = |J\phi(x)| \theta(x)$.  
\end{remark}

However, in contrast with past works on submanifold registration, this is not the only possible group action that could be considered on the space $\mathcal{V}_d$. For instance, one can define another action by removing the above volume change term, taking instead
\begin{align*}
    (\phi_{*} \mu | \omega) := \int_{\mathbb{R}^d \times \widetilde{G}^n_d} \omega(\phi(x),d_{x} \phi \cdot T) d \mu(x,T).
\end{align*}

This \textit{normalized action} has the property of preserving the total mass of the varifold, i.e.,
\begin{align*}
    | \phi_{*} \mu | (\mathbb{R}^n) = | \mu | (\mathbb{R}^n), \ \forall \mu \in \mathcal{V}_d \textrm{ and } \phi \in \textrm{Diff}(\mathbb{R}^n).
\end{align*}
Although this action is not consistent with the action on rectifiable sets as in Proposition \ref{prop:pushforward_rectifiable}, this model may be more adequate in applications to certain types of data in which mass preservation is natural. 

We refer the interested reader to \cite{hsieh2019diffeomorphic} for a more in depth discussion on the properties (orbits, isotropy subgroups...) of these group actions in the simpler case of 1-varifolds. In the rest of the paper, we will restrict ourselves to the pushforward action model of \eqref{eq:def_pushforward}, although we expect the following derivations to adapt to other cases as well, which precise study is for now left as future work.

\subsection{The diffeomorphic registration problem}
\label{ssec:diffeom_reg}

With the group action defined above, we are now ready to introduce the mathematical formulation of the diffeomorphic registration problem for general varifolds in $\mathcal{V}_d$. As deformation model, we will rely on the Large Deformation Diffeomorphic Metric Mapping (LDDMM) setting mentioned in the introduction. 

Let us briefly sum up the basic construction of LDDMM, which details can be found in \cite{Beg2005,younes2019shapes}. In this framework, deformations consist of diffeomorphisms generated by flowing time-dependent vector fields. Let $V$ be a fixed RKHS of vector fields on $\R^n$ and $L^2([0,1],V)$ be space of time dependent velocity fields $v$ such that for all $t\in [0,1]$, $v_t$ belongs to $V$. The flow map $t \mapsto \varphi_t^v$ is defined for all $t\in[0,1]$ by $\varphi_0^v = \text{id}$ and the ODE $\dot{\varphi}^v_t = v_t \circ \varphi^v_t$. If $V$ is continuously embedded in $C_0^1(\R^n,\R^n)$, one can show that for all $t$, $\varphi_t^v$ is a $C^1$-diffeomorphism of $\R^n$. Moreover, on the subgroup $\text{Diff}_V =\{\varphi_1^v \ | \ v \in L^2([0,1],V)\}$ of $\text{Diff}(\R^n)$, one can define the following right-invariant Riemannian metric:
\begin{equation*}
    d_{G_V}(\text{id},\phi) = \inf \left\{ \int_0^1 \|v_t\|_V^2 dt \ | \ \varphi_1^v = \phi \right\}
\end{equation*}

Let us now consider a source (or template) varifold $\mu_0 \in \mathcal{V}_d$ as well as a target $\mu_{tar} \in \mathcal{V}_d$. With the above deformation model and metric, registering $\mu_0$ to $\mu_{tar}$ consists in finding a deformation $\phi$ that minimizes $d_{G_V}(\text{id},\phi)$ with the constraint that $\phi_{\#} \mu_0$ is close to $\mu_1$ in the sense of a kernel metric $\|\cdot\|_{W^*}$ defined in Section \ref{ssec:kernel_metrics}. This can be reformulated as the following optimal control problem:
\begin{equation}
\label{eq:matching_var}
\argmin_{v \in L^2([0,1],V)}  \bigg\{E(v) = \frac{1}{2}\int_{0}^{1} \|v_t\|_V^2 dt + \lambda \|\mu(1) - \mu_{tar} \|_{W^*}^2 \bigg\}
\end{equation}
with $v$ being the control, $E$ the total cost and the state equation is given by $\mu(t) \doteq (\phi_t^v)_{\#} \mu_0$ for the pushforward model. The first term in \eqref{eq:matching_var} is the regularization term that constrains the regularity of the estimated deformation paths. The second term measures the similarity between the deformed varifold $\mu(1)$ and the target varifold $\mu_{tar}$. $\lambda$ is a weight parameter between the regularization and fidelity terms. Note that this is consistent with the generic inexact registration problem formulation in LDDMM that was proposed for objects like images, landmarks, submanifolds...

The well-posedness of the optimal control problem \eqref{eq:matching_var} holds under the following assumptions:

\begin{theorem}\label{thm:exist_opt_control}
If $V$ is continuously embedded in $C_0^2(\mathbb{R}^n,\mathbb{R}^n)$, $W$ is continuously embedded in $C_0^1(\mathbb{R}^n \times \widetilde{G}^n_d)$  and $\rm{supp}(\mu_0) \subset K$, for some compact subset $K$ of $\mathbb{R}^n \times \widetilde{G}^n_d$, then there exists a global minimizer to the problem \eqref{eq:matching_var}.
\end{theorem}
The proof is similar to previous results of the same type on rectifiable currents and varifolds. We give it in Appendix for the sake of completeness.  
\begin{remark}
One can derive necessary and sufficient conditions on the kernels of $W$ and $V$ for the two embedding assumptions of Theorem \ref{thm:exist_opt_control} to hold (see for instance Theorem 2.11 in \cite{Micheli2014}). In our context, in order to get $W \hookrightarrow C_0^1(\mathbb{R}^n \times \widetilde{G}^n_d)$ for instance, it is enough to assume that $\rho$ and $\gamma$ are $C^2$ functions such that all derivatives of $\rho$ up to order 2 vanish as $x \rightarrow +\infty$. 
\end{remark}

As an important note, the formulation of \eqref{eq:matching_var} extends registration of submanifolds or rectifiable subsets in the sense that if $\mu_0 = \mu_{X_0}$ and $\mu_{tar}=\mu_{X_{tar}}$ for two oriented d-rectifiable subsets of $\R^n$ then \eqref{eq:matching_var} becomes equivalent, thanks to Proposition \ref{prop:pushforward_rectifiable}, to registering rectifiable subsets, i.e. to the problem 
\begin{equation*}
  \argmin_{v \in L^2([0,1],V)}  \bigg\{\frac{1}{2}\int_{0}^{1} \|v_t\|_V^2 dt + \lambda \|\mu_{X(1)} - \mu_{X_{tar}} \|_{W^*}^2 \bigg\} 
\end{equation*}
with $X(t) = \varphi_t^v \cdot X_0$, which is the setting of many past works as for instance \cite{Glaunes2008,Charon2,Charon2017}.

\subsection{General optimality conditions}
\label{ssec:general_PMP}
A last important question we address in this section is the derivation of necessary optimality conditions for the solutions of \eqref{eq:matching_var}. In standard finite-dimensional optimal control problems, these are provided by the Pontryagin Maximum Principle (PMP) introduced originally in \cite{Pontryagin1962}. The approach generalizes, with a certain number of technicalities, to a broad class of infinite-dimensional shape matching problems, as developed in \cite{arguillere14:_shape}. 

We follow the same setting as well as related works such as \cite{Sommer2013} by first rewriting the above problem as an optimal control problem on diffeomorphisms, i.e.
\begin{equation*}
    \argmin_{v \in L^2([0,1],V)}  \bigg\{\frac{1}{2}\int_{0}^{1} \|v_t\|_V^2 dt + g(\varphi_1^v) \ | \ \text{s.t.} \ \dot{\varphi}_t^v = v_t \circ \varphi_t^v \bigg\} 
\end{equation*}
with $g(\varphi_1^v) \doteq \lambda \|(\varphi_1^v)_{\#} \mu_0 - \mu_{tar} \|_{W^*}^2$. The state variables are now given by the deformations $\varphi_t^v$ which we view as elements of the Banach space $\mathcal{B} \doteq \text{id} + C_0^1(\R^n,\R^n)$. Let us denote, for $\phi \in \text{Diff}(\R^n)$, $\xi_{\phi} : V \rightarrow C_0^1(\R^n,\R^n)$ the mapping $v \mapsto v \circ \phi$. We then introduce the Hamiltonian functional $H: \ C_0^1(\R^n,\R^n)^* \times \mathcal{B} \times V \rightarrow \R$ defined by:
\begin{equation}
    \label{eq:Hamiltonian_cont}
    H(p,\phi,v) = (p|v\circ \phi) - \frac{1}{2}\|v\|_V^2
\end{equation}
where $p$ is the costate variable which is a vector distribution of $C_0^1(\R^n,\R^n)^*$ and $(p|v\circ \phi)$ denotes the duality bracket in $C_0^1(\R^n,\R^n)^*$. With the assumptions of Theorem \ref{thm:exist_opt_control}, it follows from the maximum principle shown in \cite{arguillere14:_shape} that if $(v_t,\varphi_t^v)$ is a global minimum of the optimal control problem, there exists a path of costates $p \in H^1([0,1],C_0^1(\R^n,\R^n)^*)$ such that the following equations hold:
\begin{equation}
\left\{\begin{array}{ll}
&\dot{\varphi}_t^v = \partial_p H(p_t,\varphi_t^v,v_t) \\
&\dot{p}_t = -\partial_{\phi} H(p_t,\varphi_t^v,v_t) \\
&\partial_v H(p_t,\phi_t^v,v_t) = 0
\end{array}\right.
\label{eq:Hamiltonian_eqn_cont}
\end{equation}
with the end time boundary conditions $p_1 = -\partial_{\phi} g(\varphi_1^v)$. From the last equation in \eqref{eq:Hamiltonian_eqn_cont}, we can attempt to deduce the form of the optimal $v$. Introducing the Riesz isometry operator $\bold{K}_V : V^*\rightarrow V$ and its inverse $\bold{L}_V=\mathbf{K}_V^{-1}: V \rightarrow V^*$, we get:
\begin{equation}
\label{eq:opt_v_cont}
  \xi_{\varphi_t^v}^* p_t - \bold{L}_V v_t = 0 \ \Rightarrow \ v_t = \bold{K}_V \xi^*_{\varphi_t^v} p_t. 
\end{equation}

One additional consequence of \eqref{eq:Hamiltonian_eqn_cont} is the following conservation of momentum again proved in \cite{arguillere14:_shape}: for all $u \in C_0^1(\R^n,\R^n)$ and $t\in[0,1]$, 
\begin{equation}
    \label{eq:conservation_momentum}
    (p_t | d\varphi_t^v u) = (p_0 | u).
\end{equation}
Note that \eqref{eq:Hamiltonian_eqn_cont}, \eqref{eq:opt_v_cont} and \eqref{eq:conservation_momentum} are generic to the LDDMM model and so far independent of the nature of the deformed objects and of the term $g(\varphi_1^v)$ in the cost. This dependency is entirely encompassed by the boundary condition $p_1 = -\partial_{\phi} g(\varphi_1^v)$ which we may describe a little more precisely based on the following:

\begin{prop}
\label{prop:variation_g}
 The end-time momentum $p_1$ is a vector distribution in $C_0^1(\R^n,\R^n)^*$ of the form
 \begin{align*}
     (p_1|u) = &\int_{\R^n} \alpha(x) \cdot u(x) \ d|\mu_0|(x) \\ 
     +&\int_{\R^n \times \widetilde{G}^n_d} \beta(x,T)  du|_T(x) \ d\mu_0(x,T) \\
     +&\int_{\R^n \times \widetilde{G}^n_d} \gamma(x,T) \text{div}_T u(x) \ d\mu_0(x,T)
 \end{align*}
 where $\alpha: \R^n \rightarrow \R^n$, $\beta: \R^n \times \widetilde{G}^n_d \rightarrow (\R^{n\times d})^*$ and $\gamma: \R^n \times \widetilde{G}^n_d \rightarrow \R$ are continuous fields and for all $T \in \widetilde{G}^n_d$, $\text{div}_T u$ and $du|_{T}$ denote the divergence and differential of $u$ restricted to $T$. 
\end{prop}
A condensed proof of this proposition can be found in the Appendix, although we have left aside the technical derivations related to differential calculus on the Grassmannian (this will be discussed further in Section \ref{sec:numerics} in the discrete setting). This result extends in a way first variation formulas for varifolds proved in \cite{Charon2,Charlier2017} which considered variations of rectifiable varifolds resulting from variations of the underlying rectifiable sets. This corresponds to the special case in which $\mu_0 = \mu_{X_0}$. In that case, one can show, after some derivations, that the above expression of $p_1$ can be rewritten in the form of a vector distribution $u \mapsto \int_{\varphi_1^v(X_0)} u(x) \cdot h(x) d \mathcal{H}^d$ in $C_0^0(\R^n,\R^n)^*$ with vectors $h(x)$ normal to $\varphi_1^v(X_0)$ at each $x$. In our more general situation, this is however not possible and $p_1$ is a priori a distribution that involves first order derivatives of the test function $u$.

Now, the conservation law of \eqref{eq:conservation_momentum} gives that for all $t \in [0,1]$, 
\begin{equation*}
  (p_t | d\varphi_t^v u) = (p_1 | d\varphi_1^v u) = (p_0 | u).  
\end{equation*}
Using the expression of $p_1$ in Proposition \ref{prop:variation_g}, and grouping all $0$-th and $1$-st order terms in the resulting expressions, we may write $p_t$ in the general form:
\begin{align*}
   (p_t | u) &= \int_{\R^n \times \widetilde{G}_d^n} \alpha_t(x,T) \cdot u(x) \ d\mu_0(x,T) + \int_{\R^n \times \widetilde{G}_d^n} B_t(x,T) d_x u|_{T} \ d\mu_0(x,T)
\end{align*}
where $\alpha_t: \R^n \times \widetilde{G}_d^n \rightarrow \R^n$ and $B_t: \R^n \times \widetilde{G}_d^n \rightarrow (\R^{n\times d})^*$ are continuous fields, with $\alpha_1(x,T)=\alpha(x)$ and $B_1(x,T) du|_{T}(x) = \beta(x,T) du|_T(x) + \gamma(x,T) \text{div}_T u(x)$. Furthermore, optimal vector fields satisfy $v_t = K_V \xi^*_{\varphi_t^v} p_t$ and we have
\begin{align*}
   (\xi^*_{\varphi_t^v} p_t | u) &= (p_t | u\circ \varphi_t^v) \\
   &= \int_{\R^n \times \widetilde{G}_d^n} \alpha_t(x,T) \cdot u(\varphi_t^v(x)) \ d\mu_0(x,T) + \int_{\R^n \times \widetilde{G}_d^n} B_t(x,T) d_{\varphi_t^v(x)}u|_{d_x\varphi_t^v \cdot T} \ d\mu_0(x,T).
\end{align*}
Denoting $K_V: \R^n \times \R^n \rightarrow \R^{n\times n} $ the reproducing kernel of $V$, the reproducing kernel property implies that for all $u\in V$ and $x,h \in \R^n$, $u(x) \cdot h = \langle K_V(x,\cdot) h, u \rangle_V$. Moreover, the similar property on the kernel first order derivatives \cite{Micheli2014} gives that for any $h,h' \in \R^n$, 
$$d_x u (h) \cdot h' = \langle \partial_1 K_V(x,\cdot)(h) \cdot h', u \rangle_V.$$
Then, we rewrite the linear maps $B_t$ as $B_t(x,T) H = \sum_{i=1}^d b_{t,i}(x,T) \cdot H_i$ for any $H=(H_1,\ldots,H_d) \in \R^{n\times d}$ and where $b_i(x,T) \in \R^n$ are the component vector fields of $B_t$. By the above and the linearity of $\bold{K}_V$, we obtain the following general expression for optimal vector fields   
\begin{align}
   v_t &= \int_{\R^n \times \widetilde{G}_d^n} K_V(\varphi_t^v(x),\cdot) \alpha_t(x,T) \ d\mu_0(x,T) & \nonumber\\
   &+ \int_{\R^n \times \widetilde{G}_d^n} \left(\sum_{i=1}^d \partial_1 K_V(\varphi_t^v(x),\cdot)(d_x\varphi_t^v(t_i))\cdot b_{t,i}(x,T)\right) d\mu_0(x,T). 
\end{align}
In contrast with LDDMM registration of submanifolds or point clouds, the expression of optimal deformation fields involves in general both the kernel function and its first order derivatives. We do not explicit the vector fields $\alpha$ and $b_i$ at this point, it will be specified later in the discrete setting, see Section \ref{ssec:discrete_registration}.

\section{Approximations by discrete varifolds}
\label{sec:approximation_discrete_var}

The previous derivations were so far conducted for completely general measures in the space $\mathcal{V}_d$ which include objects of widely different natures. In the perspective of implementing numerically the above approach, which is the subject of Section \ref{sec:numerics}, we first need to build an adequate discretization framework in $\mathcal{V}_d$ with approximation guarantees, and even more importantly investigate the consistency of the discretized registration problems (Theorem \ref{thm:convergence_sol}), which is the main result of this section. 

\subsection{Discrete approximations}
\label{ssec:discrete_approx}
In what follows, we will consider the specific class of varifolds which can be written as finite combinations of Dirac masses:
\begin{align}\label{eq:discrete_varifolds}
    \mu = \sum_{i=1}^N r_i \delta_{(x_i,T_i)},\ r_i \in \mathbb{R}_+,\ x_i \in \mathbb{R}^n, \ T_i \in \tilde{G}^n_d.
\end{align}
for some $N\geq 1$. Throughout this paper, varifolds of this form will be called \textit{discrete varifolds}. It is quite natural to consider this type of varifolds for the purpose of representing discrete shapes, which has been exploited in previous works on piecewise linear curves and surfaces. For example, if $X = \bigcup_{i=1}^N X_i$ is a triangulated surface, with $X_i$ being the mesh triangles with specified orientations, one can write $\mu_X = \sum_{i=1}^N \mu_{X_i}$ and for each $i \in \{1,\cdots,N\}$ approximate $\mu_{X_i}$ by $r_i \delta_{(x_i,T_i)}$, where $x_i$ is the center of $X_i$, $T_i$ the oriented plane containing $X_i$ and $r_i = \mathcal{H}^d(X_i)$. This leads to the approximation $\widetilde{\mu}_{X} := \sum_{i=1}^N r_i \delta_{(x_i,d_i)}$. As proved in \cite{Charon2017}, this approximation provides an acceptable error bound for $d_{W^*}$:
\begin{align*}
    d_{W^*}(\mu_X,\widetilde{\mu}_{X}) \leq Cte \ \mathcal{H}^d(X) \max_{i} \textrm{diam}(X_i).
\end{align*}
The main interest of such discrete varifold approximations is that the expression of the metric \eqref{eq:rkp2} becomes particularly simple to compute numerically. Indeed, given two discrete varifolds $\mu = \sum_{i=1}^{N} r_i \delta_{(x_i,S_i)} $ and $\mu' = \sum_{j=1}^{M} r'_j \delta_{(x'_j,T'_j)}$, we have as a particular case of \eqref{eq:rkp2}:
\begin{align}\label{eq:norm_discrete_varifolds}
       \langle \mu,\mu' \rangle_{W^*} =  \sum_{i=1}^N\sum_{j=1}^M  r_i r'_j \rho(|x_i-x'_j|^2) \gamma(\langle T_i,T'_j \rangle). 
\end{align}

The above approximation scheme only applies to the case of piecewise linear shapes given by meshes such as polygonal curves or triangulated surfaces. In the more general context of this work, a key issue is to construct similar discrete varifold approximations for more general and less structured objects. Specifically, given a varifold $\mu$ with finite total weight, can it be approximated by discrete varifolds and will approximations converge as $N\rightarrow +\infty$? This is the problem known as \textit{quantization}, which has been studied intensively in the case of probability measures over Euclidean spaces \cite{graf2007foundations} or manifolds \cite{kloeckner2012approximation}, under specific regularity assumptions on those measures. In the situation of varifolds, an interesting recent work on this question is \cite{buet2018discretization}. The authors prove that any rectifiable varifold with finite mass can be approximated by a sequence of discrete varifolds for the bounded Lipschitz distance and propose a numerical approach to approximate mean curvature measures based on discrete varifolds. 

In this section, we first wish to extend approximation results to general oriented varifolds of finite mass for both $d_{BL}$ and $d_{W^*}$ metrics.
\begin{theorem} 
\label{thm:varifold_approximation}
Let
 \begin{align*}
     \mathcal{V}^N_d := \left\{ \sum_{i=1}^N r_i \delta_{(x_i,T_i)} | r_i \in \mathbb{R}_+,\ x_i \in \mathbb{R}^n, \ T_i \in \widetilde{G}^n_d \right\}
 \end{align*}
 be the (non-convex) cone of discrete varifolds with at most $N$ Diracs. For any oriented varifold $\mu \in \mathcal{V}_d$ with $|\mu |(\mathbb{R}^n) < \infty$, there exists a sequence $\mu_N \in \mathcal{V}_d^N$such that $\lim_{N \rightarrow \infty} d_{BL}(\mu_N,\mu) =0$. Moreover, if $\mu$ has compact support, then we can assume that for all $N$, $\rm{supp}(\mu_N) \subset K$ for some compact set $K \subset \mathbb{R}^n \times \widetilde{G}^n_d$ and 
 \begin{equation*}
     d_{BL}(\mu_N,\mu) < \frac{C}{N^{1/(n+d(n-d))}},
 \end{equation*}
 where $C$ is a constant that only depends on $n$, $d$ and $\rm{supp}(\mu)$.
\end{theorem}

\begin{proof}
 We first tackle the case of compactly supported $\mu$. Without loss of generality, we may also assume that $\mu$ is a probability measure. Let $D = n+ d(n-d)$ and $B \subset \mathbb{R}^n$ be a closed ball that contains $\rm{supp}|\mu|$. For brevity, we write $M \doteq B \times \widetilde{G}^n_d$. Since we can view $M$ as a compact $D$-dimensional submanifold of $\R^n \times \Lambda^d(\mathbb{R}^n)$ (using Pl\"{u}cker embedding), M is also regular of dimension $D$ (cf. \cite{graf2007foundations}), i.e., $0<\mathcal{H}^D(M)< \infty$ and there exist $c,r_0>0$, such that
  \begin{align*}
      \frac{1}{c}r^D \leq \mathcal{H}^D \mres M(B_r(a)) \leq c r^D,\ \forall a \in M, \ r \in (0,r_0).
  \end{align*}
  Given $\varepsilon \in (0,5 r_0)$, by the 5-Times Covering Lemma (cf. \cite{Simon}), these exists a subset $\mathcal{I} \subset M$, such that $M \subset \cup_{x \in \mathcal{I}} B_{\varepsilon}(x)$ and $B_{\varepsilon/5}(x) \cap B_{\varepsilon/5}(y) = \varnothing$ for all $x \neq y \in \mathcal{I}$. Therefore,
  \begin{align*}
      \mathcal{H}^D(M) \geq \sum_{x \in \mathcal{I}} \mathcal{H}^D(M \cap B_{\varepsilon/5}(x)) 
      \geq \frac{|\mathcal{I}| \varepsilon^D}{c 5^D}.
  \end{align*}
 We can thus obtain a partition $\{A_i\}_{i=1,\cdots,|\mathcal{I}|}$ of $M$ from the the collection $\{B_{\varepsilon}(x) \cap M\}_{x\in\mathcal{I}}$ which satisfies $\sup_i \textrm{diam}(A_i) < \varepsilon$ and 
 \begin{equation*}
     |\mathcal{I}| \leq \frac{c 5^D \mathcal{H}^D(M)}{ \varepsilon^D}.
 \end{equation*}
Let $r_i = \mu(A_i)$ and $(x_i,T_i) \in A_i$ and define $\nu=\sum_{i=1}^{|\mathcal{I}|} r_i \delta_{(x_i,T_i)}$. For any $\varphi \in \textrm{Lip}_1(\mathbb{R}^n \times \widetilde{G}^n_d)$, with $\|\varphi\|_{\infty} \leq 1$, we have

\begin{align*}
\left| \int_{\mathbb{R}^n \times \widetilde{G}_d^n} \varphi(x,T)d \nu - \int_{\mathbb{R}^n \times \widetilde{G}_d^n} \varphi(x,T) d\mu\right| 
&= \left| \sum_{i=1}^{|\mathcal{I}|} \left( \mu(A_i) \varphi(x_i,T_i) - \int_{A_i}\varphi(x,T) d\mu \right) \right| \\
&\leq \sum_{i=1}^{|\mathcal{I}|} \int_{A_i}| \varphi(x_i,T_i) - \varphi(x,T)| d\mu \\
&<\sum_{i=1}^{|\mathcal{I}|} \varepsilon \mu(A_i) = \varepsilon.
\end{align*}
Taking the supremum over all $\omega \in \textrm{Lip}_1(\mathbb{R}^n \times \widetilde{G}^n_d)$ with $\|\varphi\|_{\infty} \leq 1$, we obtain $d_{BL}(\mu,\nu)< \varepsilon$. Then for each $N \in \mathbb{N}$, we can choose $\varepsilon_N = 5(C \mathcal{H}^D(M)/N)^{1/D}$ and we obtain $\mu_N \in \mathcal{V}_d^N$ such that
\begin{equation*}
    d_{BL}(\mu,\mu_N)< \frac{5 C^{1/D} (\mathcal{H}^D(M))^{1/D}}{N^{1/D}}
\end{equation*}
and in particular $\lim_{N \rightarrow +\infty} d_{BL}(\mu,\mu_N) = 0$.  

Suppose now that $\rm{supp}(\mu)$ is not compact: we show that for any $\varepsilon>0$, there exists a discrete varifold $\nu$ such that $d_{BL}(\mu,\nu)< \varepsilon$. Choose a compact set $K \subset \mathbb{R}^n \times \widetilde{G}^n_d$ such that $\mu(\mathbb{R}^n \times \widetilde{G}^n_d \setminus K)< \varepsilon/2$. From the previous case, we can find a discrete varifold $\nu$ such that $d_{BL}(\mu \mres K,\nu )<\varepsilon/2$, and hence $d_{BL}(\mu,\nu )<\varepsilon$.
\end{proof}
Note that the proposition clearly holds for non-oriented varifolds as well. Another direct consequence, thanks to proposition \ref{prop:top_comp} and remark \ref{rmk:dom_W_BL}, is the following corresponding statement for $d_{W^*}$:
\begin{corollary}
\label{cor:var_approx_W}
With the assumptions from proposition \ref{prop:ctru_kernel}, one also has $\lim_{N \rightarrow \infty} d_{W^*}(\mu_N, \mu) =0$. If in addition $W \hookrightarrow C_0^1(\R^n\times \tilde{G}_d^n)$, an equivalent upper bound as in Theorem \ref{thm:varifold_approximation} holds for $d_{W^*}(\mu,\mu_N)$.    
\end{corollary}
We should point out that the asymptotic convergence rate given by the previous upper bound is rather slow, especially as the dimensions $d$ and $n$ grow. This is however under very mild assumptions on the varifold $\mu$. We expect much better convergence properties for certain specific classes of varifolds, for instance assuming Alfors regularity as in \cite{kloeckner2012approximation}, although we leave such questions for future investigation.

\subsection{Optimal approximating sequence}
In addition to the asymptotic approximation results of the previous section, we now want to construct such sequences of discrete approximating varifolds. Given any $\mu \in \mathcal{V}_d$ and $N \in \N$, a natural idea is to look for the optimal discrete varifold in $\mathcal{V}_d^N$ that approximates $\mu$ in terms of the metric $d_{W^*}$. Due to the intricate structure of the set $\mathcal{V}_d^N$ (infinite-dimensional non-convex cone), this is far from a straightforward problem. Several different approaches in some simpler contexts have been proposed to circumvent this issue, which we briefly recap. One possibility is to restrict to finite-dimensional vector spaces of $\mathcal{V}_d$ (e.g. generated by finite sets of Diracs). Works such as \cite{durrleman2009statistical,Gori2016} for instance, which are focused on the model of currents, consider dictionaries of Diracs defined on a predefined grid of point positions in $\R^n$. Then the problem can be recast as the one of finding sparse approximations of $\mu$ in such a dictionary. It remains a NP hard problem but solutions can be approached either through greedy algorithms like \textit{orthogonal matching pursuit} as proposed in \cite{durrleman2009statistical} or by considering the $L^1$ relaxation formulation leading to a standard convex LASSO program. Such ideas apply well to the specific situation of currents mainly as a result of the inherent linearity of this model: indeed, at any iteration of a matching pursuit procedure, once the optimal position of a Dirac is found, the corresponding direction vector and weight are explicitly determined. This allows to limit the search over grid of points in the spatial domain only. Unfortunately, for the general oriented varifold metrics we consider in this paper, such a property no longer holds and, as a result, these methods would involve very large dictionaries defined on grids on the product $\R^n \times \widetilde{G}_d(\R^n)$. Such an increase in dimension makes the approach numerically impractical as soon as $n\geq 3$ and $d\geq2$. Another downside is that the use of finite dictionaries and greedy algorithms like matching pursuit is not guaranteed to give an optimal approximation of varifolds for a given number $N$ of Diracs. The approach we develop in this section consists instead in directly tackling the non-convex problem of computing the projection onto $\mathcal{V}_d^N$ for the class of kernel metrics $d_{W^*}$. It shares some connection with the recent work of \cite{Chauffert2017} that considers a related problem for standard measures defined on the torus $\R^n / \mathbb{Z}^n$. 

Fix a varifold $\mu_* \in \mathcal{V}_d$. For any $N \in \mathbb{N}, \ N\geq 1$, we seek $\mu_N \in \mathcal{V}_d^N$ that is closest to $\mu_*$ for $d_{W^*}$, namely
\begin{align}\label{eq:projection_problem}
    \mu_{N} = \operatorname*{argmin}_{\mu \in  \mathcal{V}^N_d} \| \mu - \mu_* \|_{W^*}
\end{align}

By construction, if $|\mu|(\mathbb{R}^n)< \infty$ then Corollary \ref{cor:var_approx_W} will imply that $(\mu_N)$ converges to $\mu$ in the metric $d_{W^*}$. We only need to ensure that such a projection is well defined, which is the object of the following proposition:

\begin{prop} \label{thm:exist_min_proj}
Suppose all assumptions in proposition \ref{prop:ctru_kernel} and remark \ref{rk:assumption_kernel} hold. We further assume  that the functions $\rho$ and $\gamma$ defining the kernels are non-negative.
Then for any $\mu \in \mathcal{V}_d$ and $N \in \mathbb{N}$, there exists $\mu_{N} \in \mathcal{V}^N_d$ such that $\mu_{N} = \operatorname*{argmin}_{\nu \in  \mathcal{V}^N_d} \| \mu - \mu_* \|_{W^*}$ 
\end{prop}

\begin{proof}
  Let $\mu_m = \sum_{i=1}^N r_i^m \delta_{(x_i^m,T_i^m)}$ be a minimizing sequence, $K_W: W^* \mapsto W$ be the dual operator and $f := K_W(\mu) $. Without loss of generality, we may assume that there is a $N_1 \leq N$ such that $\sup_{1\leq i \leq N_1} |x_i^m| $ remains bounded and $\inf_{ M+1\leq i \leq N} |x_i^m|$ tends to $\infty$ as $m \rightarrow \infty$.

Observe that $\sup_{1 \leq i \leq N} \{r_i^m\}$ must be bounded. If it's not bounded, then from the assumptions that $\rho, \gamma \geq 0$, we obtain 

\begin{align*}
    \|\mu_m -\mu_* \|_{W^*} 
    &\geq \| \mu_m  \|_{W^*} -\| \mu_*  \|_{W^*} 
                         = \sqrt{\sum_{i,j=1}^N r_i^m r_j^m \rho(|x_i^m - x_j^m|^2) \gamma(\langle T_i^m,T_j^m \rangle)} - \| \mu_* \|_{W^*} \\
                         &\geq \sqrt{\sum_{i=1}^N (r_i^m)^2 \rho(0) \gamma(1)} - \| \mu_* \|_{W^*} \rightarrow \infty
\end{align*}
as $m \rightarrow \infty$, which is absurd.

Since $r_i^m,$ $T_i^m$, $f(x_i^m,T_i^m)$ and 
\begin{align*}
    A_m := (\rho(|x_i^m -x_j^m|^2) \gamma(\langle T_i^m, T_j^m \rangle))_{1 \leq i,j \leq N}
\end{align*}
 are all bounded sequences of $m$, we may replace them by convergent subsequences, thus we could assume that 

\begin{align*}
    &\lim_{m \rightarrow \infty} r_i^m = r_i, \ \ \lim_{m \rightarrow \infty} T_i^m = T_i, \ 
    \lim_{m \rightarrow \infty} f(x_i^m,T_i^m) = f_i, \ \ \lim_{m \rightarrow \infty} A_m = A.
\end{align*}
Since $\rho \in C_0(\mathbb{R})$, the matrix $A$ must has the following form:

\begin{align*}
    A = \left(\begin{array}{cc}
       B_1  & \boldsymbol{0} \\
       \boldsymbol{0}  & B_2
    \end{array} \right),
\end{align*}
where $B_1$ and $B_2$ are $N_1$-by-$N_1$ and $N-N_1$-by-$N-N_1$ semi-positive definite matrices. Combining this with the assumption $f \in C_0(\mathbb{R}^n \times \widetilde{G}^n_d)$, we obtain

\begin{align*}
    &\lim_{m \rightarrow \infty} \| \mu_m -\mu_* \|_{W^*}^2 
    = \boldsymbol{r'}^T B_1 \boldsymbol{r'} + \boldsymbol{r''}^T B_2 \boldsymbol{r''} - 2 \boldsymbol{f'}^T  \boldsymbol{r'} + \| \mu_* \|_{W^*}^2,
\end{align*}
where $\boldsymbol{r'} = (r_1,\cdots,r_{N_1})$, $\boldsymbol{r''} = (r_{N_1+1},\cdots,r_{N})$, and $\boldsymbol{f'} = (f_1,\cdots,f_{N_1})$. Since $\sup_{1 \leq i \leq N_1} |x_i^m|$ is bounded we can assume that  $\lim_{m \rightarrow \infty} x_i^m = x_i, \ 1 \leq i \leq N_1$. Let $\mu := \sum_{i=1}^{N_1} r_i \delta_{(x_i,u_i)}$, then 
\begin{align*}
     \| \mu -\mu_* \|_{W^*}^2 &= \boldsymbol{r'}^T B_1 \boldsymbol{r'} - 2 \boldsymbol{f'}^T  \boldsymbol{r'} + \| \mu_* \|_{W^*}^2 
     \leq \lim_{m \rightarrow \infty} \| \mu_m -\mu_* \|_{W^*}^2.
\end{align*}
Hence $\mu$ is a minimizer. 
\end{proof}
However, in general this projection is not unique. We also point out that the existence is a priori not guaranteed if kernels $\rho$ and $\gamma$ take negative values. It is so far an open question to determine to what extent one could generalize the result of Proposition \ref{thm:exist_min_proj}, one particular but important case being the one of current metrics obtained for $\gamma(t)=t$ which is not covered by our result.

As written in the proof of Proposition \ref{thm:exist_min_proj}, \eqref{eq:projection_problem} is equivalent to the optimization problem: 
\begin{align}
    (r_i,x_i,T_i) = \operatorname*{argmin}_{(w_i,y_i,S_i)} \| \sum_{i=1}^N w_i \delta_{(y_i,S_i)} - \mu_* \|_{W^*}^2
\end{align}
Any solution must satisfy first order optimality conditions obtained by differentiating $\|\mu_N-\mu\|_{W^*}$ with respect to the $(r_k,x_k,T_k)$. In particular, we have
\begin{align*}
    0 &= \frac{\partial \| \mu_N-\mu_* \|_{W^*}^2}{\partial r_k} \\
    &= 2 \bigg( \sum_{i=1}^N r_i \rho(|x_i-x_k|^2) \gamma(\langle T_i,T_k \rangle)  - \int_{\mathbb{R}^n \times \widetilde{G}^n_d} \rho(|x_k-x|^2) \gamma(\langle T_k,T \rangle) d \mu_*(x,T) \bigg).
\end{align*}
which gives $\langle \mu_N-\mu_*,\mu_N \rangle_{W^*} = 0$. It shows that for any $N \in \N$, $\|\mu_N\|_{W^*} \leq \|\mu_*\|_{W^*}$.

\subsection{$\Gamma$-convergence of registration functionals}
Ultimately, our purpose is to use the previous approximating discrete varifolds $\mu_N$ to approximate the diffeomorphic registration problem \eqref{eq:matching_var}. The natural question that arises is whether replacing the source varifold $\mu_0$ by its projections $\mu_N$ in \eqref{eq:matching_var} still leads to reasonable approximations (at least asymptotically) of optimal deformation fields for the original problem. In this section, we address this by showing a $\Gamma$-convergence property for these variational problems. We point out that our setting and the following proof differ quite a bit from previous results of the same type that were dealing with the specific case of surface triangulations such as \cite{Arguillere2016b}.      
   
To obtain such convergence results for solutions of variational problems, one usually requires the approximating sequence to possess certain nice properties. Specifically, assuming $\mu \in \mathcal{V}_d$ with compact support and finite mass and $\{\mu_N\} \subset \mathcal{V}^N_d$ such that $\lim_{N \rightarrow \infty} \|\mu_N - \mu \|_{W^*} = 0$, we will need that $\bigcup_{N} \textrm{supp}(\mu_N) \subset K$ for some compact set $K \subset \mathbb{R}^n$ or that $\sup_{N} |\mu_N|(\mathbb{R}^n) < \infty$. Unfortunately, this does not hold in general since convergence in $d_{W^*}$ does not allow to control the support nor the total mass of the sequence $\mu_N$.

Yet, provided that $\bigcup_{N} \textrm{supp}(\mu_N) \subset K$, we can actually retrieve the boundedness of the total mass. We assume in what follows that the kernels are such that $\rho(0)>0$ and $\gamma(1)>0$. 
   \begin{lemma}\label{lemma:unif_bdd_weight}
      Let $\{\mu_N\}$ be a sequence of discrete varifolds with finite mass such that there exists a compact $K \subset \R^n$ with $\rm{supp}(|\mu_N|) \subset K$ for all $N$. We assume that $\{\|\mu_N\|_{W^*}\}$ is bounded. Then $\{|\mu_N|(\mathbb{R}^n)\}$ is bounded. 
   \end{lemma}
   
\begin{proof}
We prove it by contradiction. Assume that $(|\mu_N|(\mathbb{R}^n))_{N \geq 1}$ is unbounded. Then, up to extracting a subsequence, we can assume that $|\mu_N|(\mathbb{R}^n) \rightarrow +\infty$. Let's write $\mu_N = \sum_{i=1}^{p_N} r_{i,N} \delta_{(x_{i,N},T_{i,N})}$. Thus $|\mu_N|(\mathbb{R}^n) = \sum_{i=1}^{p_N} r_{i,N} \rightarrow +\infty$. Since $\rho$ and $\gamma$ are continuous and $\rho(0),\gamma(1)>0$, we can find compact subsets $A \subset K$ and $B \subset \tilde{G}^n_d$ with diameters small enough, so that: $\inf_{x,y \in A} \rho(|x-y|^2)>m>0$, $\inf_{u,v \in B} \gamma(\langle u, v \rangle) > m' >0 $ and $\lim_{N \rightarrow \infty} \sum_{i\in \mathcal{I}_N} r_{i,N} = \infty$, 
where $\mathcal{I}_N := \{i: (x_{i,N},u_{i,N}) \in A \times B\}$. It follows that, as $N \rightarrow \infty$,
\begin{align*}
\|\mu_N\|_{W^*}^2 &= \sum_{i=1}^{p_N} \sum_{j=1}^{p_N} r_{i,N} r_{j,N} \rho(|x_{i,N}-x_{j,N}|^2) \gamma(\langle T_{i,N},T_{j,N} \rangle) \\
&\geq mm' \sum_{i,j \in \mathcal{I}_N}r_{i,N} r_{j,N} \\
&= mm' \left( \sum_{i \in \mathcal{I}_N}r_{i,N} \right)^2 \rightarrow \infty,
\end{align*}
which is a contradiction. 
\end{proof}

Lemma \ref{lemma:unif_bdd_weight} suggests that one should enforce the uniform compactness of the supports of the $\mu_N$. To do so in the context of the projection approach of the previous sections, we consider solving the optimization problem \eqref{eq:projection_problem} with the additional constraint that the support of $\mu_N$ stays in a compact set containing $\rm{supp}(|\mu|)$. We still have to verify the convergence of the resulting sequence:  

\begin{prop}\label{prop:nice_app_seq}
   Let $\mu_0$ be a varifold with finite mass and $K$ be a compact set in $\mathbb{R}^n$ which contains $\rm{supp}(|\mu_0|)$. Construct the approximating sequence of $\mu_0$ by solving the following constrained optimization problem:
   \begin{align*}
       &\mu_{K,N} = \operatorname*{argmin}_{\nu \in  \mathcal{V}^N_d} \| \nu - \mu \|_{W^*} \\
       &subject \ to \ \rm{supp}(|\nu|) \subset K.
   \end{align*}
   Then $\mu_{K,N}$ converges to $\mu_0$ in $d_{W^*}$ and, if the kernel $k$ is $C_0$-universal, it also converges in $d_{BL}$. 
\end{prop}
\begin{proof}
   Thanks to Theorem \ref{thm:varifold_approximation} and Lemma \ref{lemma:unif_bdd_weight}, we immediately get that $\|\mu_{K,N} - \mu_0\|_{W^*} \rightarrow 0$ as $N \rightarrow \infty$ and $\sup_N(|\mu_{K,N}|(\mathbb{R}^n))< \infty$. Moreover, if $k$ is $C_0$-universal, then by Proposition \ref{prop:dW_finer_weakstar} it implies that $\mu_{K,N} \overset{\ast}{\rightharpoonup} \mu_0$. Since $ \bigcup_N \rm{supp}(|\mu_{K,N}|) \subset K$ and $\sup_N(|\mu_{K,N}|(\mathbb{R}^n))< \infty$, weak-* convergence implies that $\mu_{K,N}$ converges to $\mu_0$ in $d_{BL}$ by Proposition \ref{prop:dBL_weakstar}. 
\end{proof}

We are now able to state the main result of this section. We assume that the source/template varifold $\mu_0$ is compactly supported and we fix $K$ is a compact subset of $\R^n$ that contains $\rm{supp}(|\mu_0|)$. Then for any $N \in \N, \ N\geq 1$, $\mu_{K,N}$ is defined as in Proposition \ref{prop:nice_app_seq} and we introduce the following energy functionals $E_N : \ L^2([0,1],V) \rightarrow \R_+$:
\begin{align}
\label{eq:energy_optcontrol_approx}
          &E_N(v) \doteq \frac{1}{2} \int_{0}^{1} \|v_t\|_V^2 dt + \lambda \|\mu_{K,N}(1) - \mu_{tar} \|_{W^*}^2 \nonumber\\
          &subj \ to \ \left\{ \begin{array}{l}
              \partial_t \varphi^v_t = v_t \circ   \varphi^v_t, \  \varphi^v_0 = id  \\
              \mu_{K,N}(t) = ( \varphi^v_t)_{\#} \mu_{K,N} 
          \end{array} \right.
\end{align}
which are the equivalent to the energy $E$ of the original problem \eqref{eq:matching_var} but replacing the template varifold $\mu_0$ by its approximations $\mu_{K,N}$. 

\begin{theorem}\label{thm:convergence_sol}
      With the above notations, we assume that the reproducing kernel $k$ of $W$ is $C_0$-universal and satisfies all the conditions of Proposition \ref{thm:exist_min_proj}. We also assume the continuous embedding $V \hookrightarrow C_0^2(\mathbb{R}^n,\mathbb{R}^n)$. Then, the sequence of functionals $E_N$ $\Gamma$-converges to $E$ for the weak topology on $L^2([0,1],V)$. Consequently, if $v_N$ is a global minimizer of $E_N$ for each $N\geq 1$, then $(v^N)$ is bounded in $L^2([0, 1],V)$ and every cluster point for the weak topology of $L^2([0,1],V)$ is a global minimum of $E$.  
\end{theorem}

\begin{proof}
%Since $\|\varphi_1^v \cdot \mu_N - \mu_{tar} \|_{W^*}^2$ is bounded from below, hence $\|v_N\|^2_{L^2([0,1],V)}$ is bounded. Let $\bar{v}$ be a cluster point of $\{v^N\}$ then there is a subsequence of $\{ v^N\}$ that converges to $\bar{v}$ weakly in $L^2([0,1],V)$. For brevity, we assume $v^N$ converges to $\bar{v}$ weakly. 

We first show that whenever $v^N$ converges to $\bar{v}$ weakly in $L^2([0,1],V)$, we have
\begin{align*}
    E(\bar{v}) \leq \liminf_{N \rightarrow \infty} E_N(v^N).
\end{align*}
Since $v \mapsto \int_0^1 \|v\|_{V}^2 dt$ is lower semicontinuous with respect to the weak topology, we only need to prove the following,
\begin{align}\label{eq:conv_deform_1}
    \lim_{N \rightarrow \infty} \|(\varphi_1^{v^N})_{\#} \mu_{K,N} - \varphi_1^{\bar{v}} \cdot \mu_0 \|_{W^*} = 0.
\end{align}
% From triangular inequality,
%\begin{align*}
%&\|(\varphi_1^{v^N})_{\#} \mu^N - (\varphi_1^{\bar{v}})_{\#} \mu \|_{W^*} \\
%&\leq  \underbrace{\| (\varphi_1^{v^N})_{\#}  \mu^N - (\varphi_1^{v^N})_{\#} \mu \|_{W^*}}_{\MakeUppercase{\romannumeral 1}} \\
%&+ \underbrace{\| (\varphi_1^{v^N})_{\#}  \mu - (\varphi_1^{v})_{\#} \mu \|_{W^*}}_{\MakeUppercase{\romannumeral 2}}.
%\end{align*}
For all $\omega \in W$ with $\| \omega \|_{W} \leq 1$, we have

\begin{align*}
\left|\left((\varphi_1^{v^N})_{\#} \mu_{K,N} - (\varphi_1^{v^N})_{\#} \mu_0 | \omega \right) \right| 
&= \bigg{|} \int_{K \times \tilde{G}^n_d} J_S \varphi_1^{v^N}(x) 
 \omega (\varphi_1^{v^N}(x),d_x \varphi_1^{v^N} \cdot S ) d (\mu_{K,N} - \mu_0) \bigg{|} \\
&\leq C_1  \int_{K \times \tilde{G}^n_d} \sup_{N \geq 1} J_S \varphi_1^{v^N} d (\mu_{K,N} - \mu_0) \\
&\leq C_1 \int_{\mathbb{R}^n \times \tilde{G}^n_d} g(x,T) d (\mu_{K,N} - \mu_0),
\end{align*}
where $g \in C_c(\mathbb{R}^n \times \tilde{G}^n_d)$ and $\sup_{N \geq 1} J_S \varphi_1^{v^N} \leq g(x,S)$, for all $(x,S) \in K \times \tilde{G}^n_d$. Similar to the computation done in the proof of Theorem \ref{thm:exist_opt_control}, we see that

\begin{align*}
&\left|\left((\varphi_1^{v^N})_{\#} \mu_0 - (\varphi_1^{\bar{v}})_{\#} \mu_0 | \omega \right) \right| 
\leq C_2 \| (\varphi_1^{v^N} - \varphi_1^{\bar{v}})|_K \|_{1,\infty}.
\end{align*}
Taking supremum over all $\omega \in W$ with $\|\omega\|_W \leq 1$, we obtain the following inequality,
\begin{align*}
&\|(\varphi_1^{v^N})_{\#} \mu_{K,N} - (\varphi_1^{\bar{v}})_{\#} \mu_0 \|_{W^*} 
\leq  C_1  (\mu_{K,N}-\mu_0| g) 
+C_2 \| (\varphi_1^{v^N} - \varphi_1^{\bar{v}})|_K \|_{1,\infty}.
\end{align*}
From Proposition \ref{prop:nice_app_seq}, $\mu_{K,N}$ converges to $\mu_0$ in the narrow topology. Hence the right hand side in the equation above converges to $0$ as $N \rightarrow \infty$. This proves \eqref{eq:conv_deform_1}. 

Second, we need to show that for each $\bar{v} \in L^2([0,1],V)$, there exists a sequence $v^N$ converging to $\bar{v}$ weakly such that
\begin{align*}
    E(\bar{v}) \geq \limsup_{N \rightarrow \infty} E_N(v^N).
\end{align*}
In fact, it suffices here to take $v^N$ to be the constant sequence $v^N = \bar{v}$ since, by a similar argument to the proof of \eqref{eq:conv_deform_1}, it leads to  
 \begin{align}\label{eq:conv_deform_2}
 \lim_{N \rightarrow \infty}\|(\varphi_1^{\bar{v}})_{\#} \mu^N - (\varphi_1^{\bar{v}})_{\#} \cdot \mu \|_{W^*} = 0
 \end{align}
and thus implies that
  \begin{align*}
      \limsup_{N \rightarrow \infty} E_N(v^N) = \lim_{N \rightarrow \infty} E_N(\bar{v}) = E(\bar{v}).
  \end{align*}
\end{proof}

Note that we stated the result of Theorem \ref{thm:convergence_sol} in the situation where only the source varifold $\mu_0$ is approximated by the projection approach that we presented in the previous sections but one can easily extend it to the scenario in which both source and target are replaced by discrete approximating sequences, the conclusion being the same in that case.

\section{Numerical considerations}
\label{sec:numerics}
Having introduced a variational formulation for the varifold registration problem together with an approach for projecting onto the space of discrete varifolds with fixed number of Diracs, we now turn more specifically to the numerical implementation of methods for solving those problems. The first hurdle, which we start by addressing in Section \ref{ssec:frame_compression}, is to define an adequate framework for representing and computing with elements of the oriented Grassmannian.

\subsection{Frame representation for metric computation and quantization}
\label{ssec:frame_compression}
In order to come up with a computationally effective representation of $\widetilde{G}_d(\R^n)$ and by extension of discrete oriented varifolds, we consider a slightly different setting than the Pl\"{u}cker embedding idea of Remark \ref{rem:Grassmannian}, primarily because the dimension of the embedding vector space $\Lambda^d(\R^n)$ may become prohibitively large in practice. We may instead choose to represent an element $T \in \widetilde{G}_d(\R^n)$ by an oriented frame $(u^{(1)},\ldots,u^{(d)}) \in \R^{n \times d}$ of independent vectors for which $T=\text{Span}(u^{(1)},\ldots,u^{(d)})$. Such a representation is of course not unique since elements of $\widetilde{G}_d(\R^n)$ are equivalence classes of oriented frames but we leave to the next section the more thorough analysis of the additional invariances that this representation will imply.

We will in fact go one step further by also incorporating the weight of Dirac varifolds in this frame representation itself, which is done as follows. Let $\mu$ be a discrete varifold of the form $\mu = \sum_{i=1}^{N} r_i \delta_{(x_i,T_i)}$. For each $i$, we consider a frame $\{u_i^{(1)},\cdots,u_i^{(d)}\}$ such that
\begin{equation}
\label{eq:frame_representation_mu}
    T_i = \frac{u_i^{(1)} \wedge \cdots \wedge u_i^{(d)}}{|u_i^{(1)} \wedge \cdots \wedge u_i^{(d)}|}  \textrm{ and } r_i = |u_i^{(1)} \wedge\cdots\wedge u_i^{(d)}|.
\end{equation}
In other words, the oriented space spanned by the frame $\{u_i^{(1)},\cdots,u_i^{(d)}\}$ corresponds to $T_i$ while its $d$-volume matches the weight $r_i$. Given such a choice of frame for each $i$, we can then identify $\mu$ with the (non-unique) state variable $q = (x_i,u^{(1)}_i,\cdots,u^{(d)}_i)_{i=1,\cdots,N}$ in the vector space $\mathbb{R}^{Nn(d+1)}$. Conversely, such a frame $q$ with $(u^{(1)}_i,\cdots,u^{(d)}_i)$ a matrix of rank $k$ for all $i$, corresponds to the (unique) discrete oriented varifold defined by the relations of \eqref{eq:frame_representation_mu}; we will denote it by $\mu^q$ in what follows.   

In this representation, the kernel metrics for discrete varifolds expressed in \eqref{eq:norm_discrete_varifolds} can be explicitly written as
\begin{equation}
\label{eq:norm_discrete_varifolds_frame}
       \langle \mu, \mu' \rangle_{W^*}^2 =\sum_{i=1}^N\sum_{j=1}^M  r_i r'_j \rho(|x_i-x'_j|^2) \gamma\left(\frac{1}{r_i r'_j}\det(u_i^{(k)} \cdot {u'}_j^{(l)})_{k,l}\right) 
\end{equation}
where $r_i = |u_i^{(1)} \wedge\cdots\wedge u_i^{(d)}| = \sqrt{\det(u_i^{(k)} \cdot u_i^{(l)})_{k,l}}$. Note that this expression does not depend on the choice of frames that satisfy the conditions of \eqref{eq:frame_representation_mu} for $\mu$ (and similarly for $\mu'$). In the case where $\mu'$ is a more general non-discrete varifold in $\mathcal{V}_d$, the computation of $\langle \mu, \mu' \rangle_{W^*}^2$ involves integrals over $\R^n \times \widetilde{G}_d(\R^n)$ of the kernel functions, which requires introducing specific quadrature schemes for approximating them. We do not address those issues in more details in this work as it needs particular discussion depending on the nature, regularity and dimension of the varifolds under consideration. Provided such adequate quadrature schemes have been defined, the $W^*$ metric then formally reduces to an expression equivalent to \eqref{eq:norm_discrete_varifolds_frame} in which the $x_j',u'_j$ and $r'_j$ are now the quadrature nodes and associated weights of the scheme.

In this setting, the solution to the projection problem \eqref{eq:projection_problem} can be computed by an iterative descent strategy on the vector $q=(x_i,u^{(1)}_i,\cdots,u^{(d)}_i)_{i=1,\cdots,N}$. The gradient of $q \mapsto \|\mu^q-\mu_*\|_{W^*}^2$ can be computed by direct differentiation of expressions like \eqref{eq:norm_discrete_varifolds_frame} with respect to the $x_i$ and $u_i^{(l)}$. In practice, computations of varifold kernel metrics for different classes of kernels and gradients of the metrics can be conveniently implemented with automatic differentiation pipelines. In our MATLAB implementation, we make use of the recent KeOps library \cite{libkp} which allows to generate CUDA functions for the low-level kernel sum evaluations and their automatic differentiation. The optimization itself is done using a limited memory BFGS algorithm from the HANSO library \cite{hanso} which we typically initialize by taking a random subset of $N$ Diracs composing the varifold $\mu_*$. Note that one of the main downside of this projection algorithm, in contrast with the previously mentioned approach of fixing a dictionary and solving a convex sparse decomposition problem, is that we can provide no general guarantees of convergence to a global minimum of \eqref{eq:projection_problem}. Results of this algorithm are discussed below in Section \ref{ssec:results_approx_reg}.

\subsection{Discrete registration model}
\label{ssec:discrete_registration}
This frame representation also provides a convenient setting to express the diffeomorphism action and registration problem on discrete varifolds. Indeed, let $\varphi$ be a diffeomorphism of $\R^n$ and $\mu \in \mathcal{V}_d^N$, the pushforward action $\varphi_{\#} \mu$ in \eqref{eq:def_pushforward} is equivalent to the following action in the frame model:
\begin{align*}
    \varphi_{\#} q := (\varphi(x_i),d_x \varphi (u^{(1)}_i),\cdots,d_x \varphi (u^{(d)}_i))_{i=1,\cdots,N} .
\end{align*}
Now, this allows us to rewrite the former infinite-dimensional optimal control problem by considering instead the finite-dimensional state variable $q \in \R^{Nn(d+1)}$. In the next paragraphs, we give a direct derivation of the optimality conditions in this discrete setting, in order to arrive at simpler and more explicit equations than the general abstract derivations presented in Section \ref{ssec:general_PMP}. Note that the resulting Hamiltonian equations we obtain are eventually very similar to the ones appearing in the 1st-order jets model studied in \cite{Sommer2013,Jacobs2013}, although there are a few notable differences due to the specific extra invariances attached to the varifold framework (c.f.  \cite{hsieh2019diffeomorphic} for a more detailed discussion in the $d=1$ case).      

Following once again the Pontryagin maximum principle approach, the Hamiltonian for this discrete representation is given by:
\begin{equation}
\label{eq:def_Hamiltonian}
H(q,p,v) \doteq
\sum_{i=1}^N \left[ p_i^{x}\cdot v(x_i) + \sum_{k=1}^d  p_i^{u_k}\cdot d_{x_i} v(u_i^{(k)}) \right]  - \frac{1}{2}\|v\|^2_V
\end{equation}
with $p^x,p^{u_k} \in \R^n$ denoting respectively the costates for the position $x$ and frame vector $u^{(k)}$ variables.
The PMP then shows that optimal trajectories of the registration problem are governed by the dynamical system:
\begin{align} 
\label{eq:ham_eq_discrete}
\left\{ \begin{array}{l}
\dot{x}_i = v_t(x_i) \\
\dot{u}_i^{(k)} = d_{x_i} v(u_i^{(k)}) \\
\dot{p}_i^{x} =  - d_{x_i}v^T p_i^{x} - \sum_{k=1}^d d_{x_i}^{(2)}v(\cdot,u_i^{(k)})^T p_i^{u_k}  \\
\dot{p}_i^{u_k} = - d_{x_i}v^T p_i^{u_k} \\
\end{array} \right.
\end{align}
while optimal vector fields $v$ satisfy
\begin{align}
\label{eq:optimal_discrete_v}
    v_t(\cdot) &= \sum_{i=1}^N \left( K(x_i(t),\cdot)p^x_i(t)  
    + \sum_{k=1}^d \partial_1 K(x_i(t),\cdot)(u_i^{(k)}(t)) \cdot p_i^{u_k} \right). 
\end{align}
Plugging the above expression of $v$ with respect to $(p,q)$ in the Hamiltonian \eqref{eq:def_Hamiltonian} gives the reduced Hamiltonian $H_r(p,q) \doteq H(p,q,v)$ which writes:
\begin{align}
    \label{eq:reduced_Hamiltonian}
   H_r(p,q) = \frac{1}{2}\sum_{i,j=1}^{N} \Big[ &p_i^x \cdot K(x_i,x_j) p_j^x  + p_i^x \cdot \sum_{k=1}^d \partial_1 K(x_j,x_i)(u_j^{(k)}) \cdot p_j^{u_k} \nonumber\\
   &+ \sum_{k=1}^{d}  p_i^{u_k} \cdot \partial_2 K(x_j,x_i)(u_j^{(k)}) p_j^x + \sum_{k,l=1}^{d} p_i^{u_k} \cdot \partial^2_{1,2} K(x_j,x_i)(u_j^{(l)},u_i^{(k)}) p_j^{u_l} \Big]
\end{align}
and \eqref{eq:ham_eq_discrete} becomes a coupled system in the variables $q$ and $p$ called the reduced Hamiltonian equations. Consequently, the set of optimal paths is entirely determined by the initial values $(q(0),p(0))$ and the value of the reduced Hamiltonian $H_r(p(t),q(t)) = \frac{1}{2} \|v_t\|_V^2$ is conserved along an optimal trajectory.  

There are in addition several other conserved quantities in such a system as evidenced by the following lemma: 
\begin{lemma}
\label{lemma:conservation_forward_eq}
For any $i=1,\ldots,N$, the matrix 
\begin{align*}
D^i(t)  \doteq  \left( \langle u_i^{(k)}(t), p_i^{u_{\ell}}(t) \rangle \right)_{1 \leq k, \ell \leq d},
\end{align*}
is constant in time.
\end{lemma}
\begin{proof}
Using the Hamiltonian equations written above, we have for all $k,l=1,\ldots,d$
\begin{equation*}
\frac{d}{dt} \left(D^i(t) \right)_{k,\ell} = \langle d_{x_i}v(u_i^{(k)}(t)), p_i^{u_{\ell}} \rangle - \langle u_i^{(k)}, d_{x_i}v^T p_i^{u_{\ell}} \rangle = 0.
\end{equation*}
Hence $D^i(t)$ is a constant matrix. \qed
\end{proof}
Note that, at this point, all those equations are fundamentally modelling the deformation of the frames $\{x_i,(u_i^{(k)})\}$ but are not yet taking into account the invariances that result from the representation of the discrete oriented varifolds as oriented frames. Those extra invariances can be derived from the boundary conditions of the PMP:
\begin{equation}
\label{eq:bd_condition_discrete_PMP}
p(1) = - \partial_q g(q)|_{q=q(1)}, \ \text{with} \ g(q) = \lambda \|\mu^q - \mu_{tar}\|_{W^*}^2.
\end{equation}
As a clear consequence of \eqref{eq:frame_representation_mu}, $\mu^q$ and thus $g(q)$ are independent of the choices of the frame vectors $(u_i^{(k)})_{k=1,\dots,d}$ that span the same oriented vector spaces $T_i$ with the same $d$-volumes $r_i$. This in turn leads to a set of conditions satisfied by the different components of the final costate $p(1)$ and, with Lemma \ref{lemma:conservation_forward_eq}, of the full path $p(t)$. These are summed up by the following result:

\begin{prop}
\label{prop:invariance_momentum}
Let $(q(t),p(t))$ be optimal trajectory, then for all $i$, the matrices $D^i(t)$ as defined above are constant scalar matrices. In particular, we have $p_i^{u_k}(t) \perp \text{Span}(\{u_i^{(\ell)}(t)\}_{\ell\neq k})$ for all $t \in [0,1]$, $i=1,\ldots,N$ and $k=1,\ldots,d$.
\end{prop}
This result, which proof can be found in Appendix, is particularly interesting from a computational point of view as it allows to partly alleviate the redundancy introduced by the frame representation of Grassmannians. Indeed, we see that the costates $p(t)$ actually lie in affine subspaces of $\R^{Nn(d+1)}$ of lower dimensions $N(n + d(n-d)+1)$, which is precisely the dimension of the 'true' state space $(\R^n \times \widetilde{G}_d(\R^n)\times \R)^N$. 

\subsection{Registration algorithm}
\label{ssec:registration_algorithm}
Based on the optimality equations of the previous section, we can now easily design an algorithm to solve the discrete registration problem. As mentioned earlier, optimal trajectories are completely determined, through the Hamiltonian equations \eqref{eq:ham_eq_discrete} and \eqref{eq:optimal_discrete_v}, by the initial conditions $q(0)=q_0$, which is known, and $p(0)$. One of the standard class of methods in optimal control, known as \textit{shooting methods}, consist in directly optimizing the cost function over $p(0)$, which has been the approach of choice in many past works on shape registration such as \cite{Vialard2012b,Sommer2013,Charon2}. We adopt a similar strategy for our particular problem.

The main issue is to compute the gradient of the total energy $E$ with respect to the initial costate $p(0)$. The regularization term being equal to $H_r(p(0),q(0))$ thanks to the conservation of the reduced Hamiltonian, its gradient can be obtained by direct differentiation of \eqref{eq:reduced_Hamiltonian}. The fidelity term $g(q(1))$ on the other hand depends indirectly on the initial costate $p(0)$ via the integration of the forward reduced Hamiltonian equations. As standard for this type of optimal control problems, c.f. \cite{Vialard2012b} or \cite{arguillere14:_shape}, the gradient of $g(q(1))$ with respect to $p(0)$ can be computed by flowing backward in time the adjoint Hamiltonian system 
\begin{equation}
\label{eq:adjoint_Hamiltonian}
    \overset{\bold{\ldotp}}{Z}(t) = -dF(q(t),p(t))^{T} Z(t) 
\end{equation}
where $F(q,p)=(\partial_p H_r(p,q), -\partial_q H_r(p,q))$, $Z(t)=(\tilde{q}(t),\tilde{p}(t)) \in \mathbb{R}^n \times (\mathbb{R}^{n})^d$ the adjoint variables of the system, together with the end-time conditions $\tilde{q}(1) = -\partial_q g(q)|_{q=q(1)}$ and $\tilde{p}(1)=0$. Although being a linear system of ODEs, the adjoint equations can be tedious to derive and implement, in particular given the rather intricate expression of the reduced Hamiltonian function considered here. Instead, the differential appearing on the right hand side of \eqref{eq:adjoint_Hamiltonian} can be approximated efficiently based on the finite difference trick proposed in \cite{arguillere14:_shape} (Section 4.1). Indeed, it can be rewritten as follows:
\begin{equation*}
    dF(q,p)^{T} Z = \begin{pmatrix} 
    \partial_p(\partial_q H_r) \cdot \tilde{q} - \partial_q (\partial_q H_r) \cdot \tilde{p}\\
    \partial_p(\partial_p H_r)\cdot \tilde{q} - \partial_q (\partial_p H_r) \cdot \tilde{p}
    \end{pmatrix}
\end{equation*}
which only involves directional derivatives of the components of the function $F$ in the directions of $\tilde{q}$ and $\tilde{p}$. We then specifically approximate the above by centered finite difference 
\begin{equation*}
    dF(q,p)^{T} Z \approx \begin{pmatrix} \alpha \\ -\beta \end{pmatrix}, \ \ \text{with } \begin{pmatrix} \alpha \\ \beta \end{pmatrix} = \frac{F(q-\epsilon \tilde{p},p+\epsilon \tilde{q})-F(q+\epsilon \tilde{p},p-\epsilon \tilde{q})}{2\epsilon}
\end{equation*}
for some small $\epsilon>0$, which only requires at each time $t$ two evaluations of the same function $F$ that appears in the forward reduced Hamiltonian equations.   

With the above approach to compute the gradient with respect to $p(0)$, the registration algorithm then consists of essentially the same steps as the aforementioned works:   
\begin{algorithmic}[1]
\Repeat
\State From $(q(0),p(0))$ compute $(q(t),p(t))$ by forward integration of the reduced Hamiltonian system given by \eqref{eq:ham_eq_discrete} and \eqref{eq:optimal_discrete_v}.
\State Compute $g(q(1))$ and $-\partial_q g(q)|_{q=q(1)}$.
\State Integrate backward the adjoint Hamiltonian equations \eqref{eq:adjoint_Hamiltonian} to obtain $\partial_{p(0)} g(q(1))$.
\State Deduce the gradient of the full cost function with respect to $p(0)$.
\State Update $p(0)$.
\Until{convergence}
\end{algorithmic}
For the numerical ODE integration steps of lines 2 and 4, we use a standard RK4 scheme with regular time samples in $[0,1]$, where we typically take $T=15$ time steps in most of the experiments that we present in the next section. Note that one can easily replace the RK4 scheme by even higher order or adaptive step methods although in practice we have found this to be unnecessary for the types of ODEs involved here. The optimization update in line 6 follows the limited memory BFGS algorithm, specifically the implementation provided by the HANSO library \cite{hanso}. One can further take additional advantage of the dimensionality reduction provided by Proposition \ref{prop:invariance_momentum} by restricting each of the components $p_i^{u_k}(0)$ to the linear subspace $\text{Span}(\{u_i^{(\ell)}(0)\}_{\ell\neq k})^{\perp}$. Lastly, as in Section \ref{ssec:frame_compression}, all kernel summation and differentiation operations appearing in both the varifold fidelity terms and Hamiltonian equations are coded in CUDA using the KeOps library \cite{libkp}. The full implementation of the varifold approximation and diffeomorphic registration approach is available at \url{https://github.com/charoncode/Var_LDDMM} together with the scripts and data of some of the simulations presented in the next section.

\section{Results}
\label{sec:results}
We now present some results of the previous algorithms on discrete varifolds of dimension $d=1$ and $d=2$. In all these experiments, we choose the deformation kernel $K$ of $V$ to be a diagonal Gaussian kernel $K(x,y)=\exp(-\frac{|x-y|^2}{\sigma_V^2}) Id$. The kernel function $\rho$ is a Gaussian of scale $\sigma_{\rho}$. The choice of these scales is adapted to the sizes of the shapes in each of the experiment. We will not discuss these questions more in detail here, since this is not our main topic and it has been more thoroughly analyzed in previous works such as \cite{Vialard2012,Charon2017,hsieh2019diffeomorphic}. The function $\gamma$ is chosen, depending on the situation, in the different classes of functions discussed in detail in \cite{Charon2017}, the main distinction being whether the considered varifolds are rectifiable or not according to the conditions given by Theorem \ref{thm:dist_rectifiable_var} and Theorem \ref{thm:dist_general_var} and whether the shapes carry a relevant orientation or not. In particular, one can use $\gamma(t) = t^2$ to recover an orientation invariant fidelity metric, or $\gamma(t) = e^{-\frac{2}{\sigma_{\gamma}^2}(1-t)}$ which leads to an orientation sensitive distance that satisfy the conditions of Theorem \ref{thm:dist_general_var}. All simulations are run on a desktop computer equipped with a NVIDIA Quadro P5000 graphics card.

\begin{figure}[h]
%\fbox{\rule{0pt}{2in} \rule{.9\linewidth}{0pt}}
    \begin{center}
    \hspace*{-1cm}
    \begin{tabular}{cccccc}
    \rotatebox{90}{varifold deform}
    &\includegraphics[trim = 18mm 18mm 18mm 18mm ,clip,width=3.1cm]{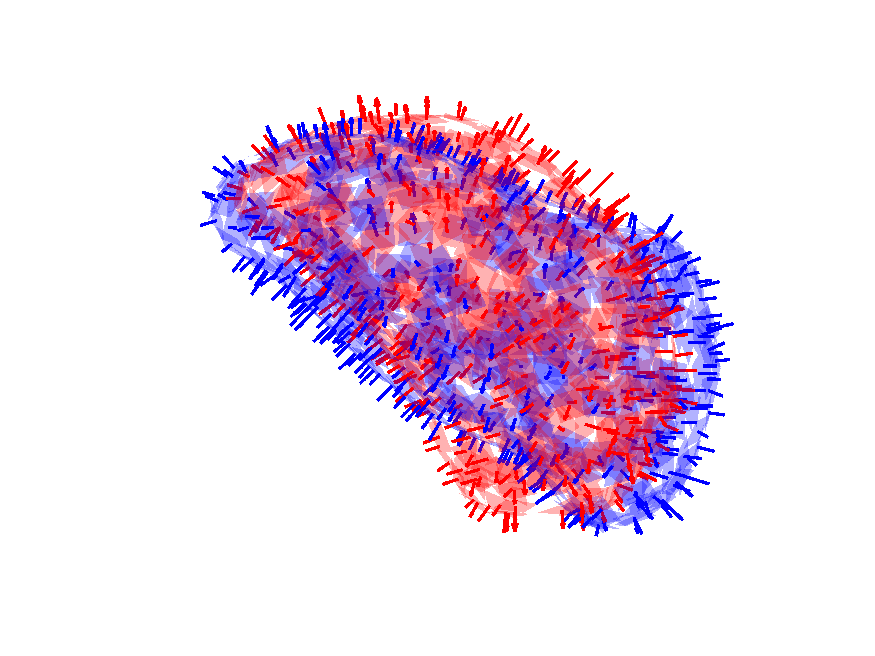}
    &\includegraphics[trim = 18mm 18mm 18mm 18mm ,clip,width=3.1cm]{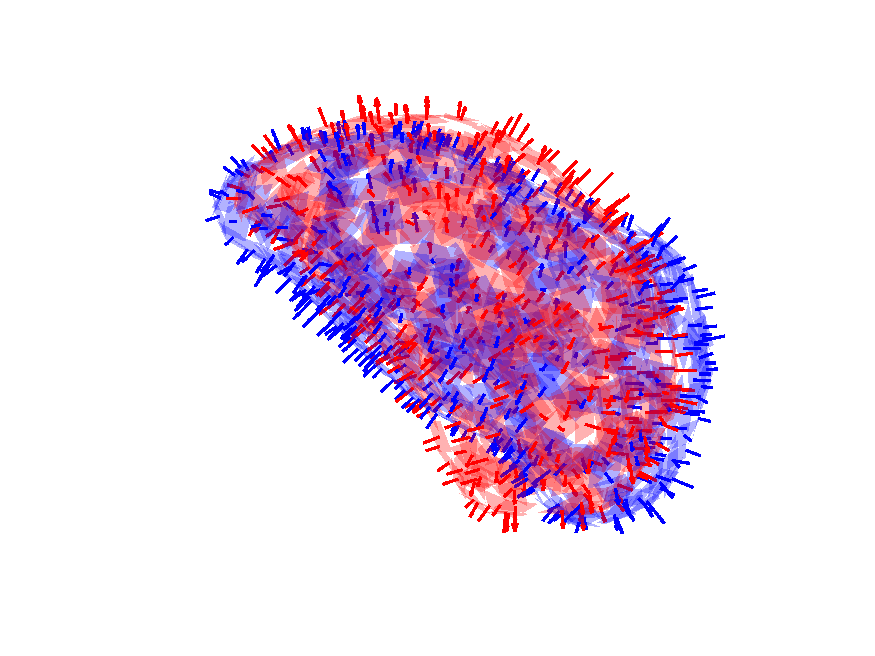} 
    &\includegraphics[trim = 18mm 18mm 18mm 18mm,clip,width=3.1cm]{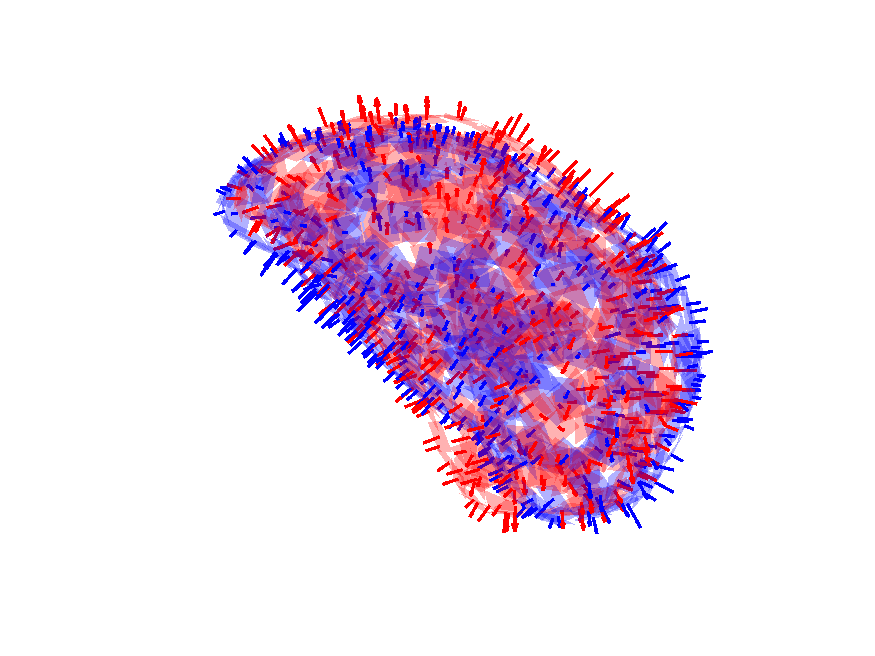}
    &\includegraphics[trim = 18mm 18mm 18mm 18mm ,clip,width=3.1cm]{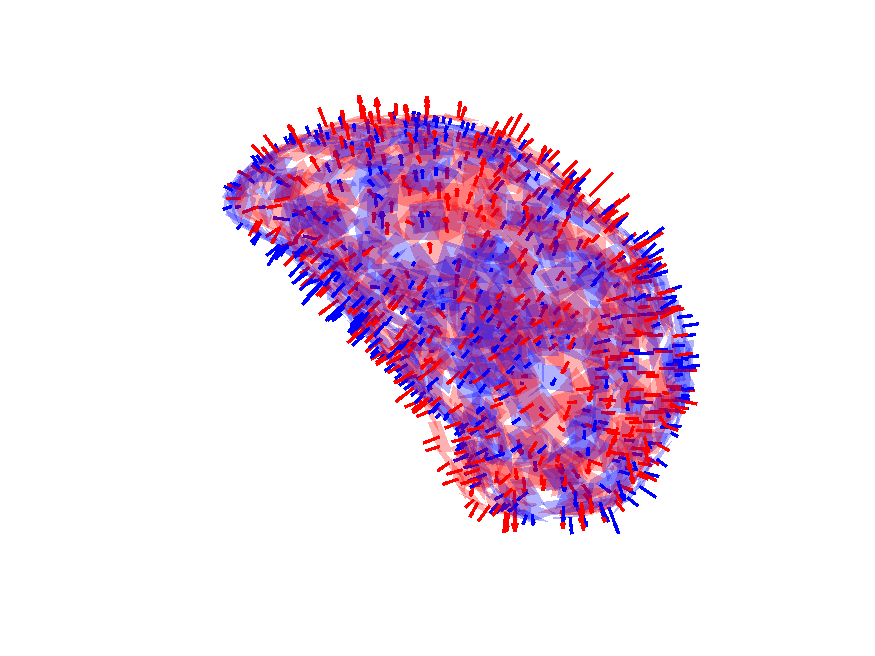} 
    &\includegraphics[trim = 18mm 18mm 18mm 18mm ,clip,width=3.1cm]{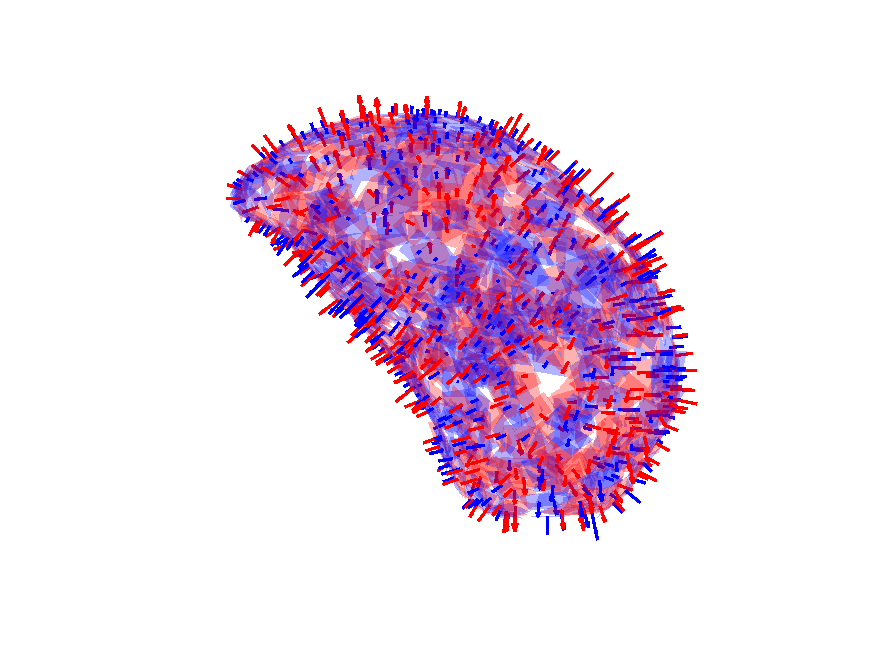}\\
    
    \rotatebox{90}{var-LDDMM}
    &\includegraphics[trim = 18mm 18mm 18mm 18mm ,clip,width=3.1cm]{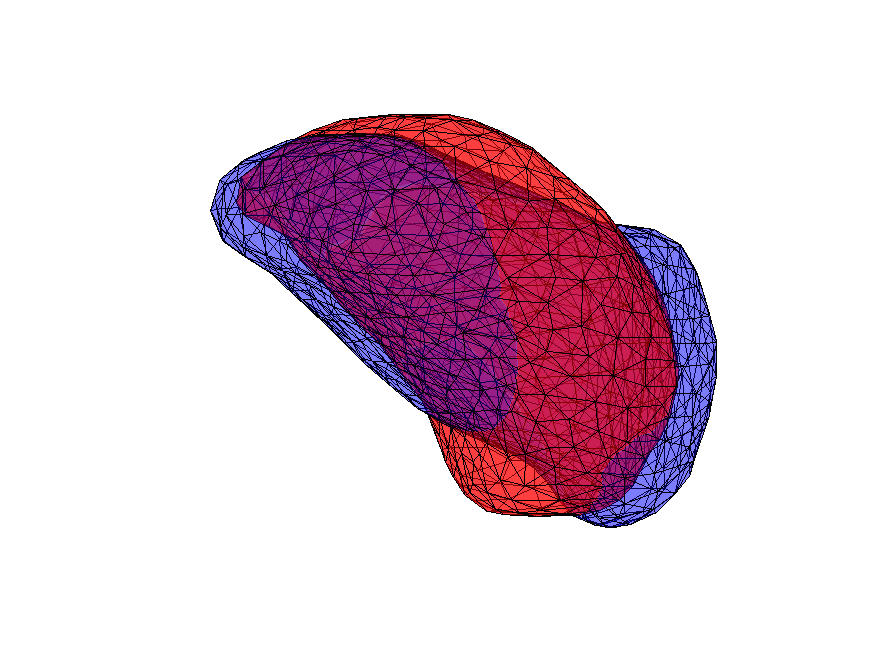}
    &\includegraphics[trim = 18mm 18mm 18mm 18mm ,clip,width=3.1cm]{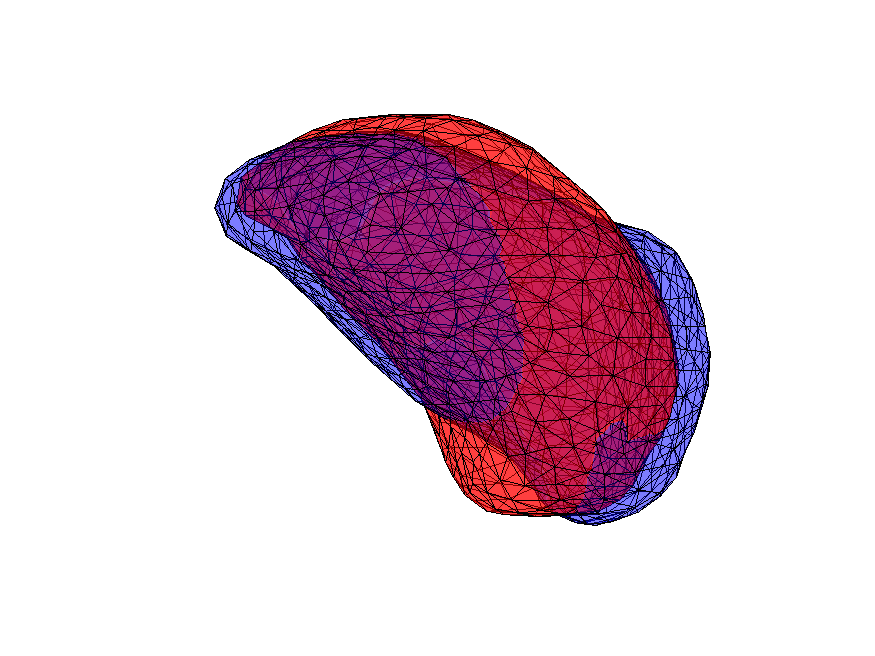} 
    &\includegraphics[trim = 18mm 18mm 18mm 18mm,clip,width=3.1cm]{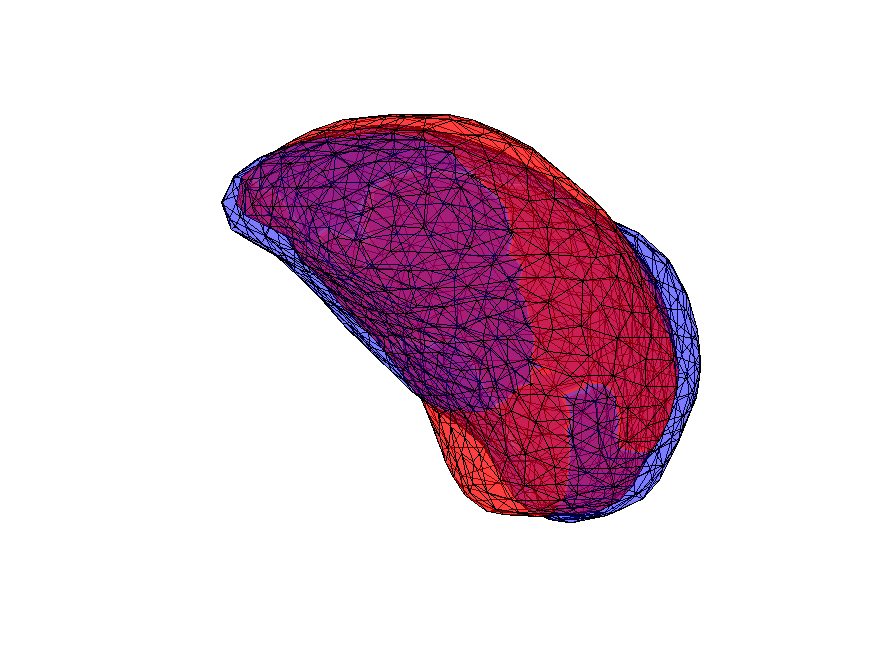}
    &\includegraphics[trim = 18mm 18mm 18mm 18mm ,clip,width=3.1cm]{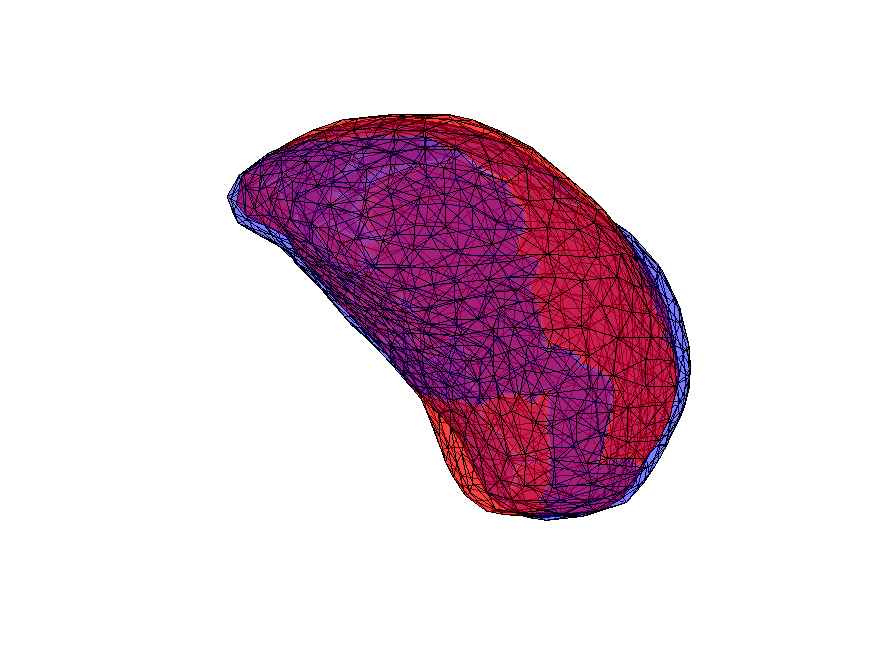} 
    &\includegraphics[trim = 18mm 18mm 18mm 18mm ,clip,width=3.1cm]{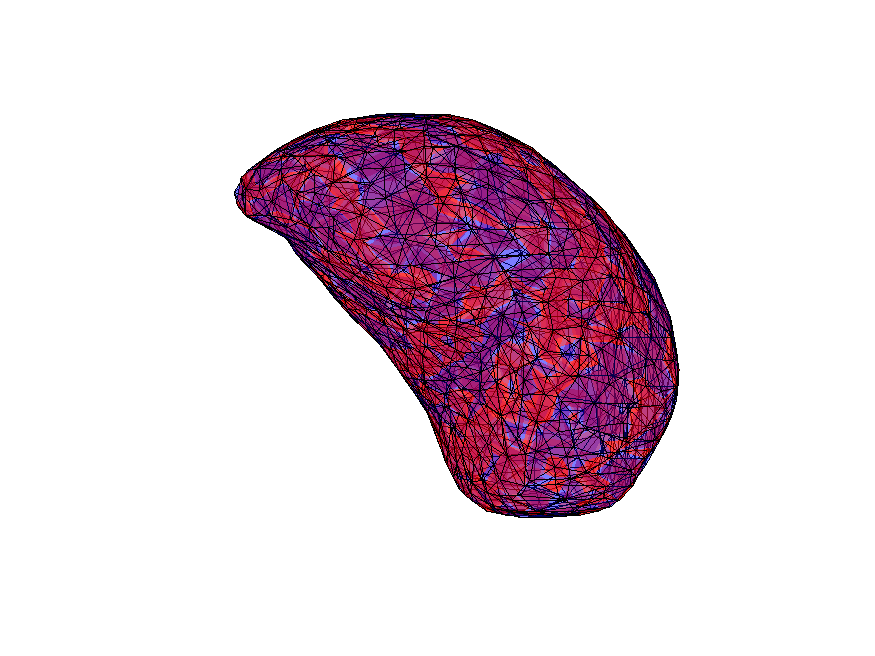}\\
    
    \rotatebox{90}{mesh-LDDMM}
    &\includegraphics[trim = 18mm 18mm 18mm 18mm ,clip,width=3.1cm]{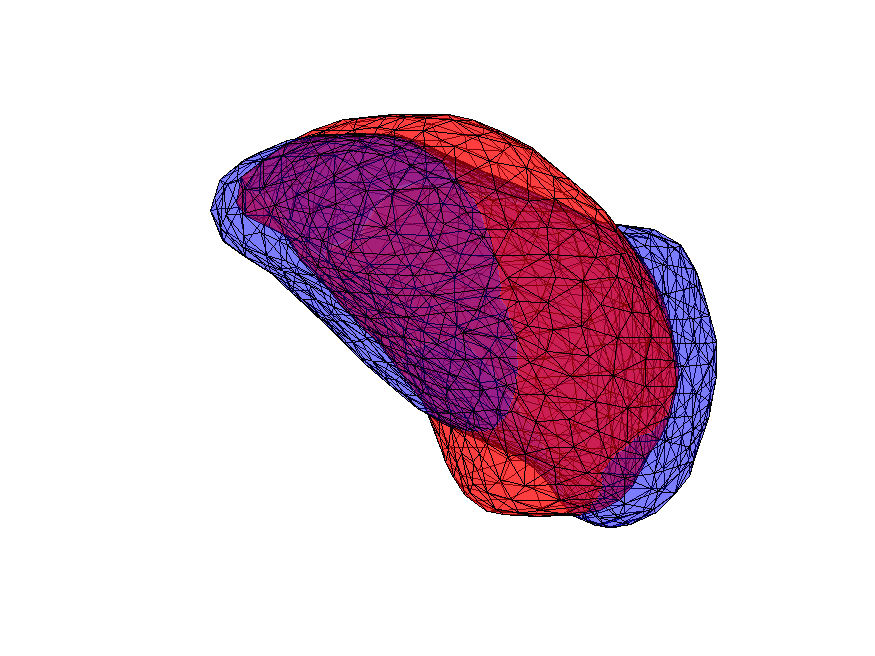}
    &\includegraphics[trim = 18mm 18mm 18mm 18mm ,clip,width=3.1cm]{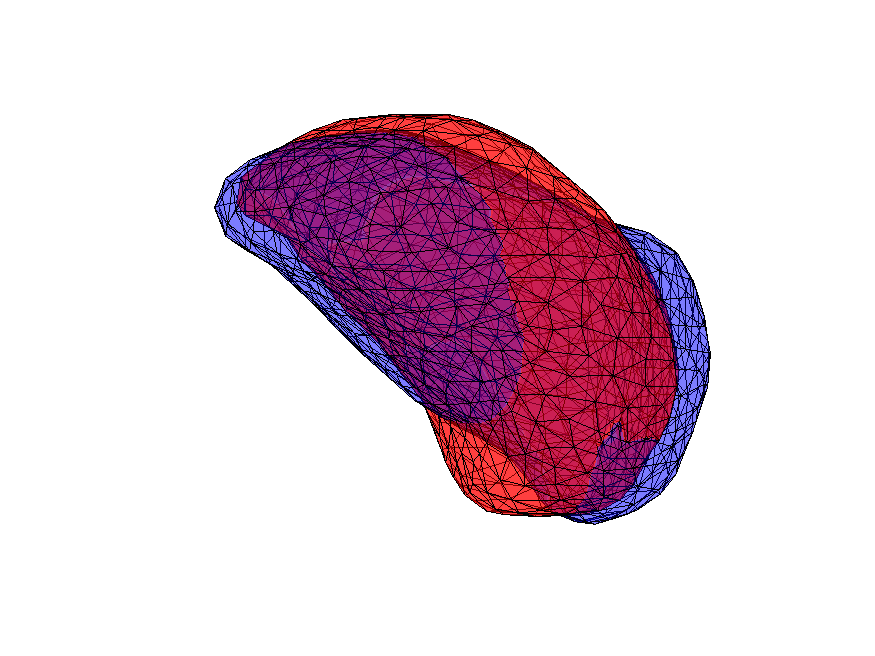} 
    &\includegraphics[trim = 18mm 18mm 18mm 18mm,clip,width=3.1cm]{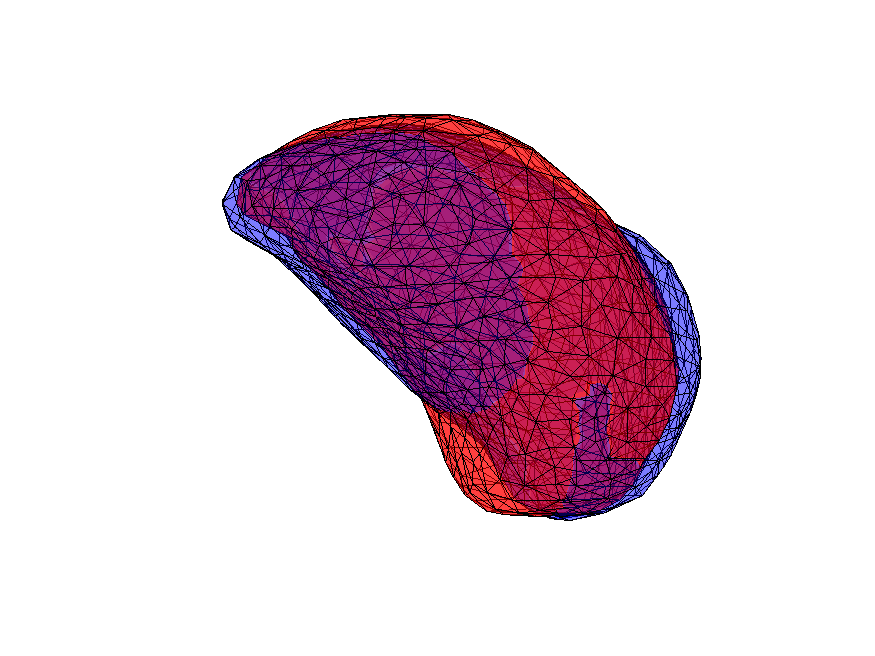}
    &\includegraphics[trim = 18mm 18mm 18mm 18mm ,clip,width=3.1cm]{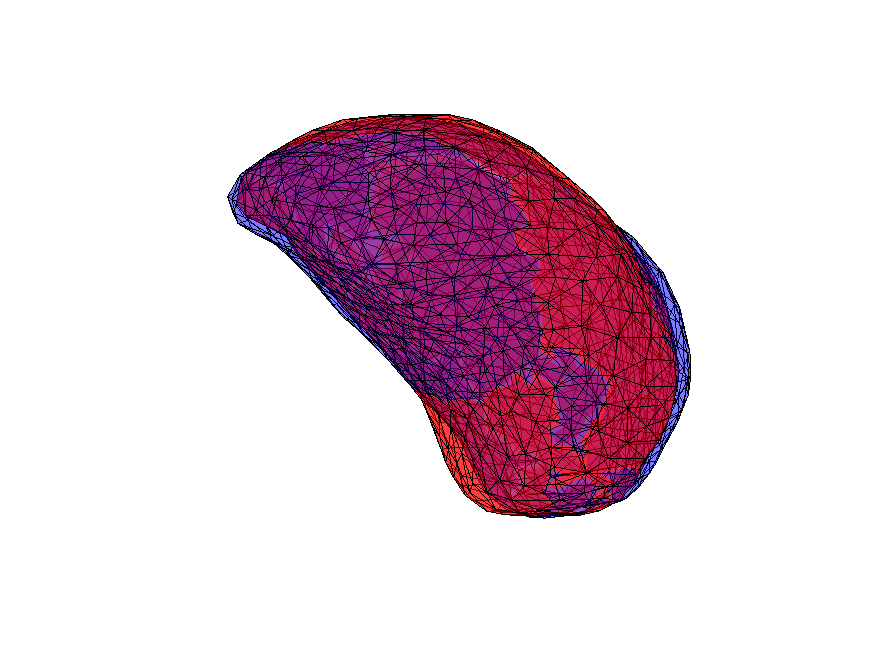} 
    &\includegraphics[trim = 18mm 18mm 18mm 18mm ,clip,width=3.1cm]{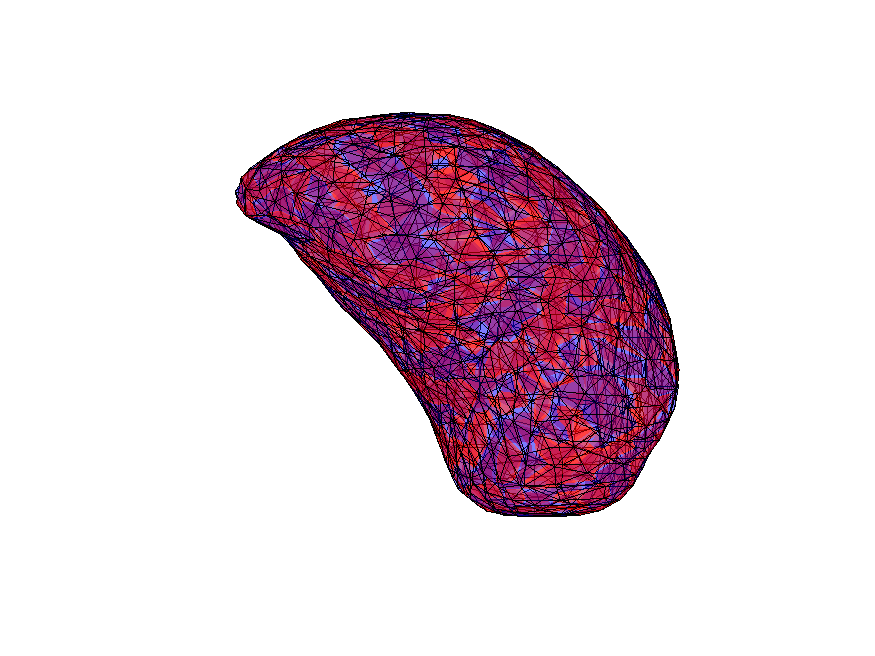}\\
    &$t=0$ & $t=1/5$ & $t=7/15$  & $t= 4/5$ & $t = 1$
    \end{tabular}
    \caption{Surface registration of two amygdalas (data courtesy of S. Ardekani) using discrete varifold LDDMM (1st and 2nd row) and surface mesh LDDMM (3rd row). The first row depicts the evolution of the deformed tangent spaces along the geodesic. The parameters used are the same for both methods; namely a weighting constant $\lambda=10$ between the regularization and fidelity term, a deformation scale $\sigma_V = 4.75$, a scale $\sigma_\rho = 3$ for the spatial kernel of the fidelity term and a Gaussian function on the sphere of scale $\sigma_{\gamma} = 1$ for the function $\gamma$.}
    \label{fig:Amygdala}
    \end{center}
\end{figure}

\subsection{Diffeomorphic registration}
\begin{figure*}[h]
%\fbox{\rule{0pt}{2in} \rule{.9\linewidth}{0pt}}
    \begin{center}
    \hspace*{-1cm}
    \begin{tabular}{cccc}
    &\includegraphics[trim = 25mm 25mm 25mm 25mm ,clip,width=4.2cm]{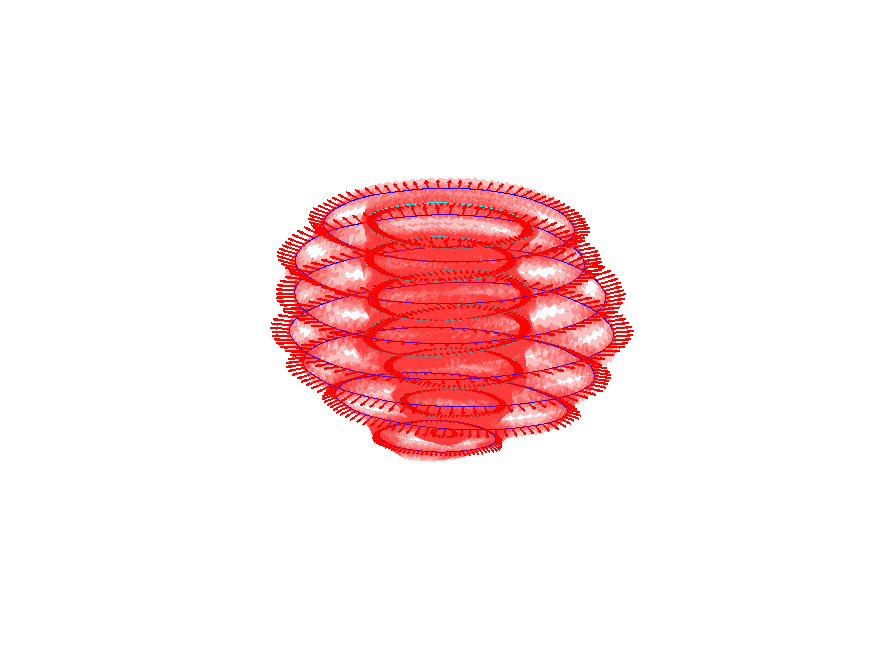} &
    &\includegraphics[trim = 25mm 25mm 25mm 25mm ,clip,width=4.2cm]{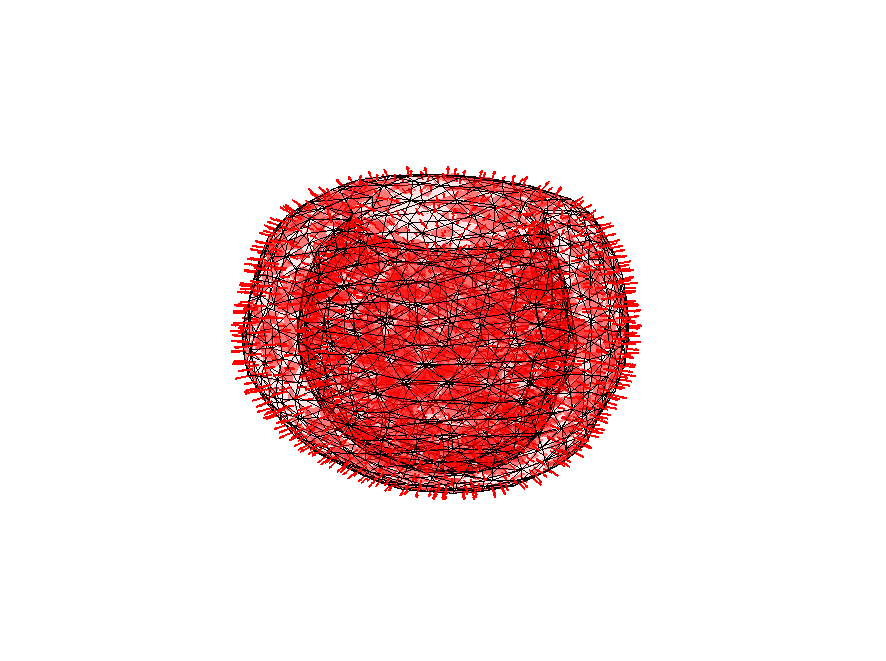} \\
    &2-varifold associated to curve set &  & 2-varifold associated to mesh surface
    \end{tabular}    
    \vskip3ex
    \hspace*{-1.5cm}
    \begin{tabular}{ccccc}
    \rotatebox{90}{curves to surf}
    &\includegraphics[trim = 25mm 25mm 25mm 25mm ,clip,width=4.1cm]{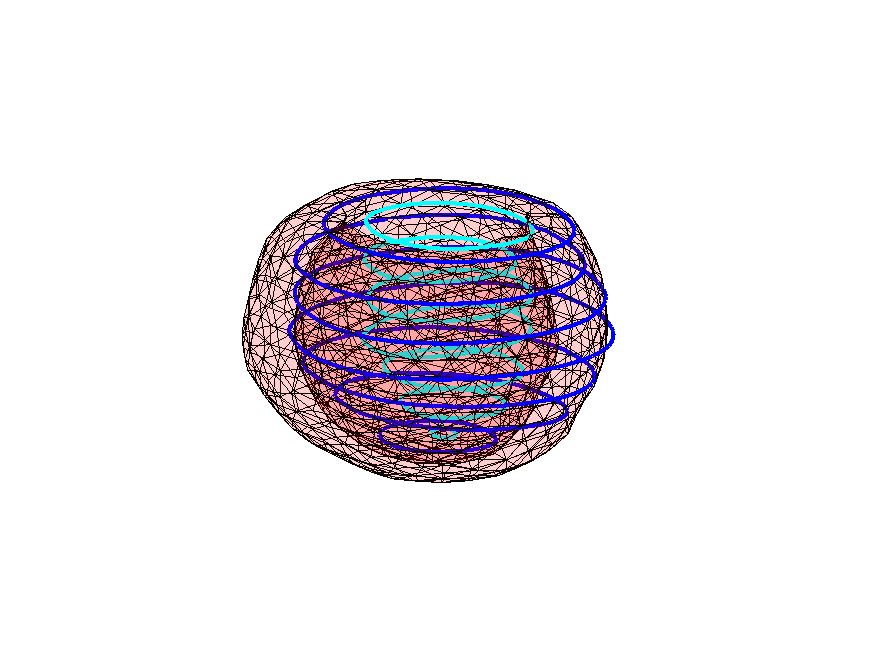}
    &\includegraphics[trim = 25mm 25mm 25mm 25mm ,clip,width=4.1cm]{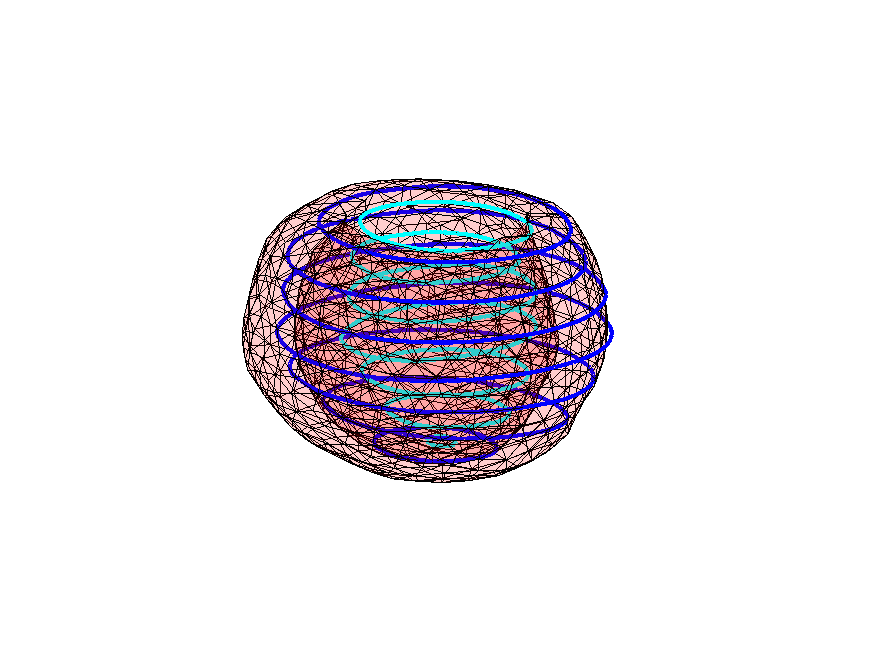} 
    &\includegraphics[trim = 25mm 25mm 25mm 25mm ,clip,width=4.1cm]{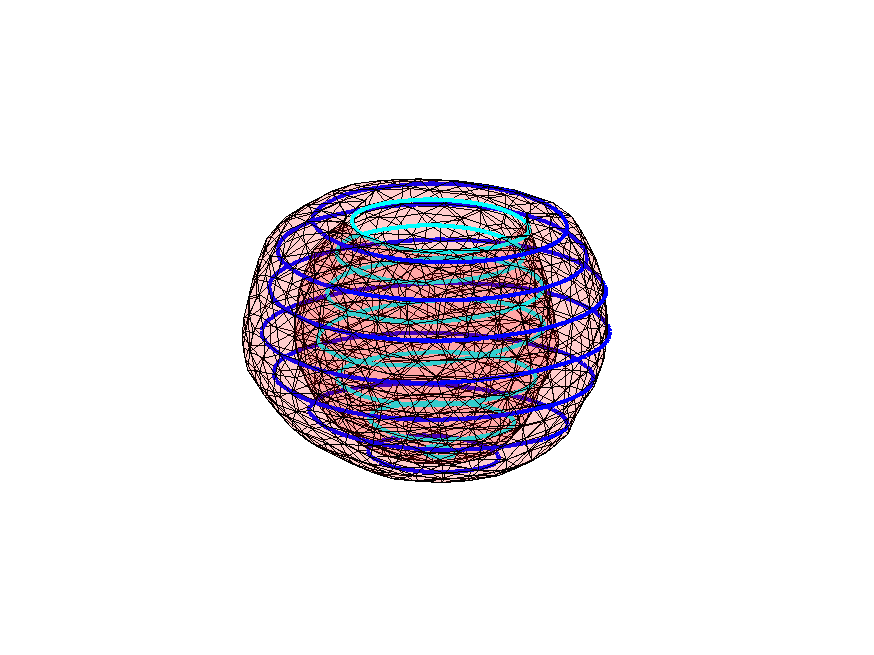}
    &\includegraphics[trim = 25mm 25mm 25mm 25mm ,clip,width=4.1cm]{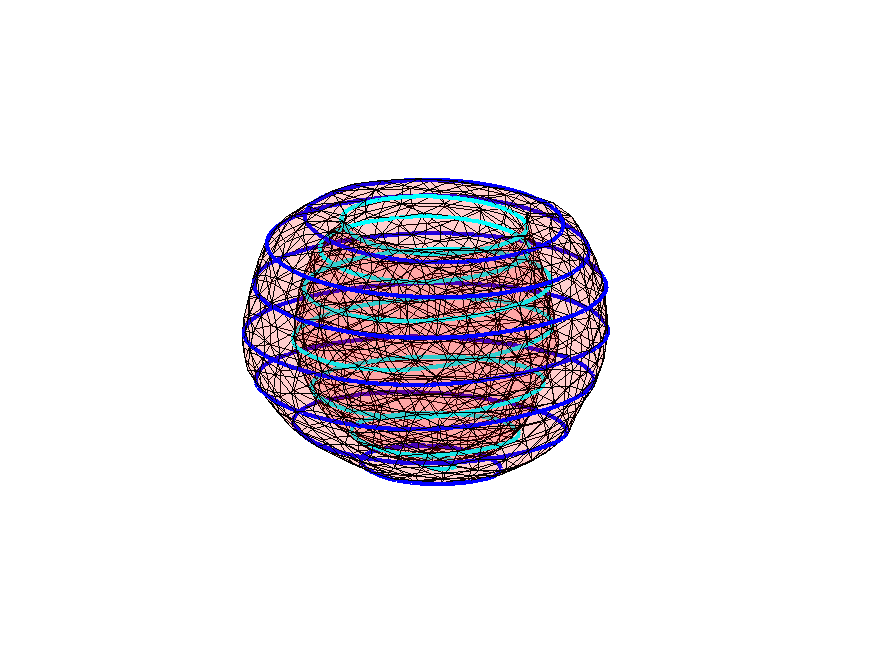} \\
    
    \rotatebox{90}{surf to curves}
    &\includegraphics[trim = 25mm 25mm 25mm 25mm ,clip,width=4.1cm]{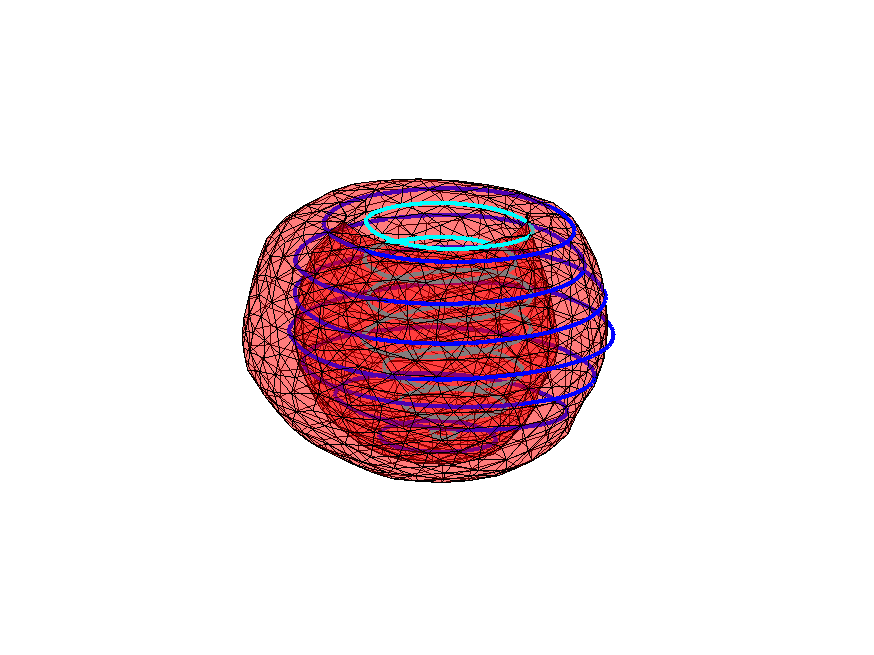}
    &\includegraphics[trim = 25mm 25mm 25mm 25mm ,clip,width=4.1cm]{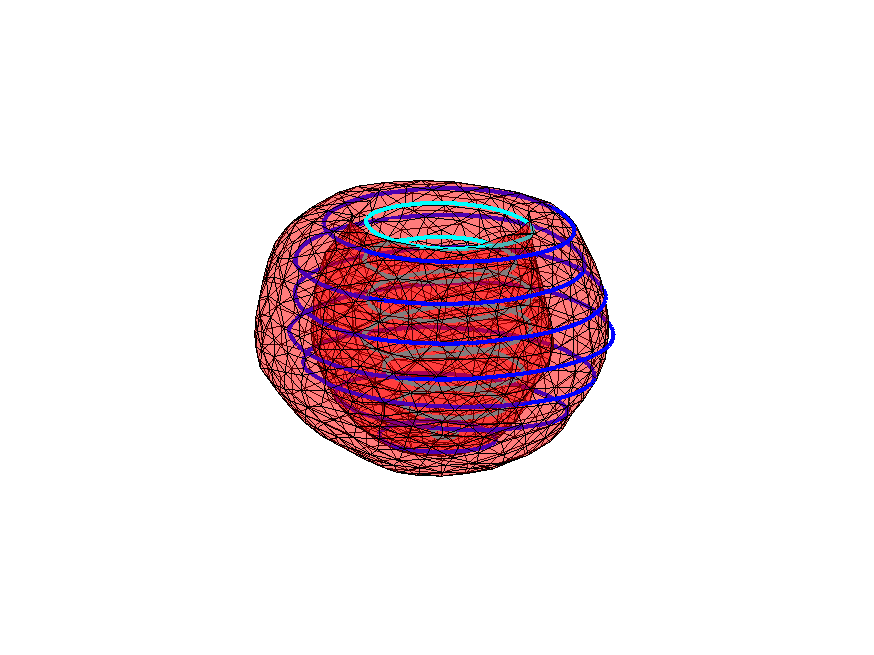} 
    &\includegraphics[trim = 25mm 25mm 25mm 25mm ,clip,width=4.1cm]{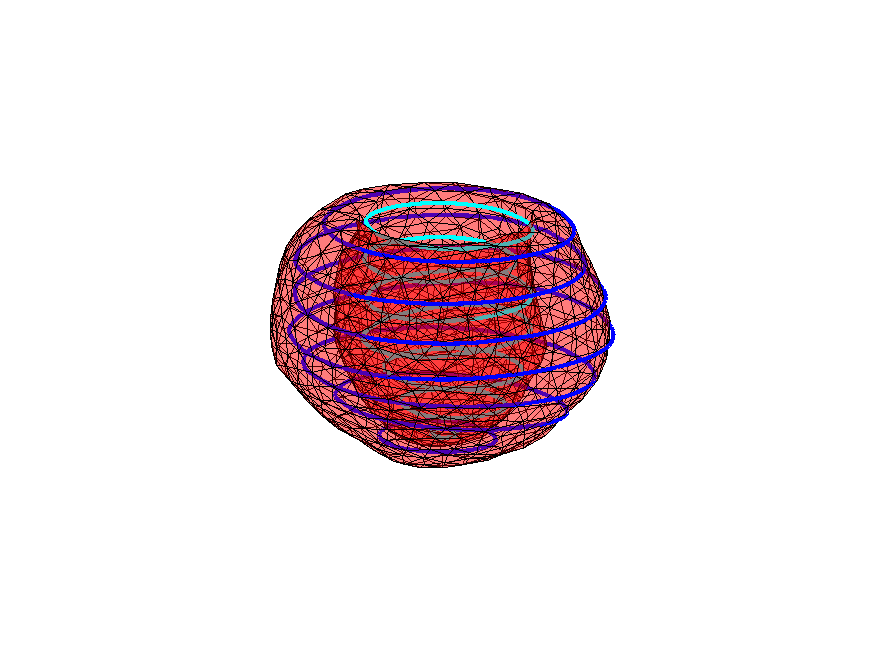}
    &\includegraphics[trim = 25mm 25mm 25mm 25mm ,clip,width=4.1cm]{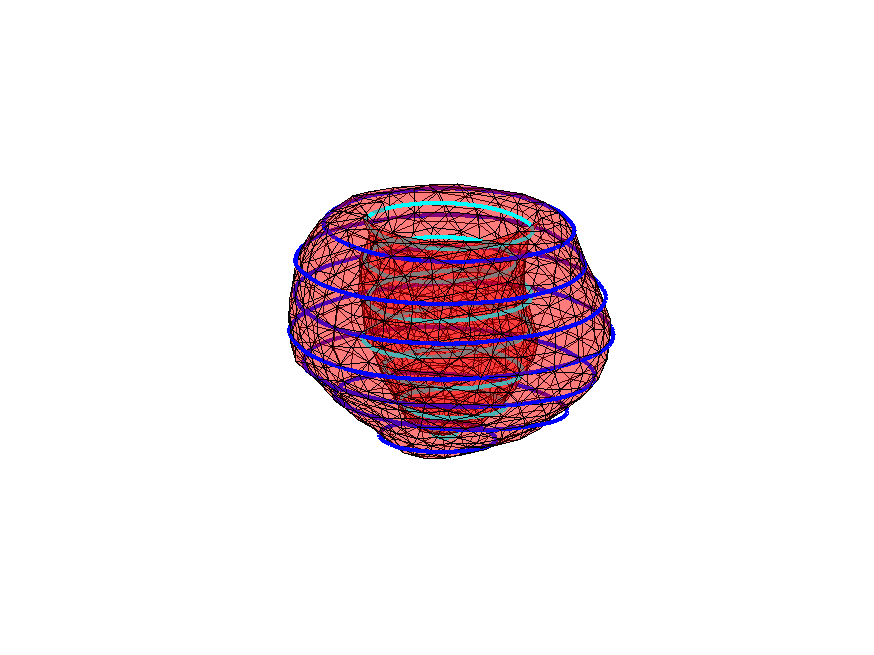}\\
    &$t=0$ & $t=0.3$ & $t=0.6$  & $t= 1$ 
    \end{tabular}
    \caption{Registration between two shapes of hearts of different nature. On top row: illustration of the 2-varifolds associated to the sectional contour curves (left) for a first subject and to a triangulated surface (right) for the second subject. In the second and third rows are shown the two results of varifold registration of surface to contour curves and contour curves to surface respectively.}
    \label{fig:heart_curves_surf}
    \end{center}
\end{figure*}

We start with results of registration obtained from the algorithm of Section \ref{ssec:registration_algorithm}. In this section, we will mostly focus on examples involving 2-varifolds, the reader may refer to \cite{hsieh2019diffeomorphic} for additional examples in the case $d=1$. First, as a sanity check, we compare our 2-varifold registration approach applied to triangulated surfaces with the previous LDDMM mesh surface matching implementation of \cite{Charon2,Charon2017} using the same kernel size parameters, in which case we expect both approaches to be theoretically equivalent as pointed out in the last paragraph of Section \ref{ssec:diffeom_reg}. Shown in Fig. \ref{fig:Amygdala} are triangulated surfaces of amygdala segmented from two different subjects of the BIOCARD database \cite{BIOCARD2015}, containing 563 vertices, 1122 triangles and 488 vertices, 972 triangles respectively. Following the simple procedure outlined at the beginning of Section \ref{ssec:discrete_approx}, we obtain discrete 2-varifolds (one Dirac for each triangle). The first row in the figure shows the optimal deformation estimated with our approach through the evolution of the discrete varifold of the source shape (red) to the target varifold (blue). Discrete varifolds are here displayed in the form of tangent patches and normal vectors (instead of 2-frames) for the purpose of better visualization. Now, the estimated vector fields $v_t$ define a path of dense deformations of the full space which we can also apply to deform the original triangulated surface, which we show on the second row of Fig. \ref{fig:Amygdala}. This is very comparable to the result of the surface mesh LDDMM registration approach displayed on the third row. In terms of computation times, the varifold registration takes a total of 494s (0.99s per iteration of BFGS) against 92.5s (0.18s per iterations) for the surface LDDMM algorithm. This difference comes from mainly two factors: the fact that the numerical complexities are quadratic in the number of Diracs (i.e. triangles) for varifold matching as opposed to the number of vertices for surface LDDMM, and from the increased dimensionality of the Hamiltonian systems in our model.

In Fig. \ref{fig:heart_curves_surf}, we consider a more challenging registration scenario which was originally studied in \cite{Ardekani2011}. Here, one of the two shape is a triangulated surface of a heart membrane segmented from high resolution CT imaging while the second one only consists of a sparse set of cross-sectional curves of the heart contour obtained from lower resolution clinical cardiac MRI data. The varifold framework of this paper leads to an alternative registration approach to the one proposed in \cite{Ardekani2011} that relies on a tailored closest point fidelity cost for the surface to curve set comparison. In our case, we instead represent both shapes as 2-varifolds and register them using the exact same varifold registration algorithm as in the previous example. The triangulated surface is again associated to a discrete 2-varifold in the same way as above. As for the set of cross-sectional curve set, we first obtain its 1-varifold representation $\{x_i,u_i^{(1)}\}$ which involve the tangent vectors $u_i^{(1)}$ to the curve that passes through $x_i$. We then complete it into a 2-varifold by adding a second "vertical" (i.e. inter-sectional) frame vector $u_i^{(2)}$, which can be estimated in this case by simply finding the projection of $x_i$ onto the corresponding curve in the section immediately above (note that this does involve any attempt to estimate an actual surface mesh of the data). We show the 2-varifolds associated to each shape in the first row of Fig. \ref{fig:heart_curves_surf} and as well as the result of the 2-varifold registration both from curve set to surface and surface to curve set. In each case, we have again applied the estimated deformation between varifolds on the original shapes for visualization.

\begin{figure*}
%\fbox{\rule{0pt}{2in} \rule{.9\linewidth}{0pt}}
    \begin{center}
    \hspace*{-1cm}
    \begin{tabular}{ccc}
    \rotatebox{90}{\phantom{aaaa} noisy point clouds }
    &\includegraphics[trim = 15mm 15mm 15mm 15mm ,clip,width=5.6cm]{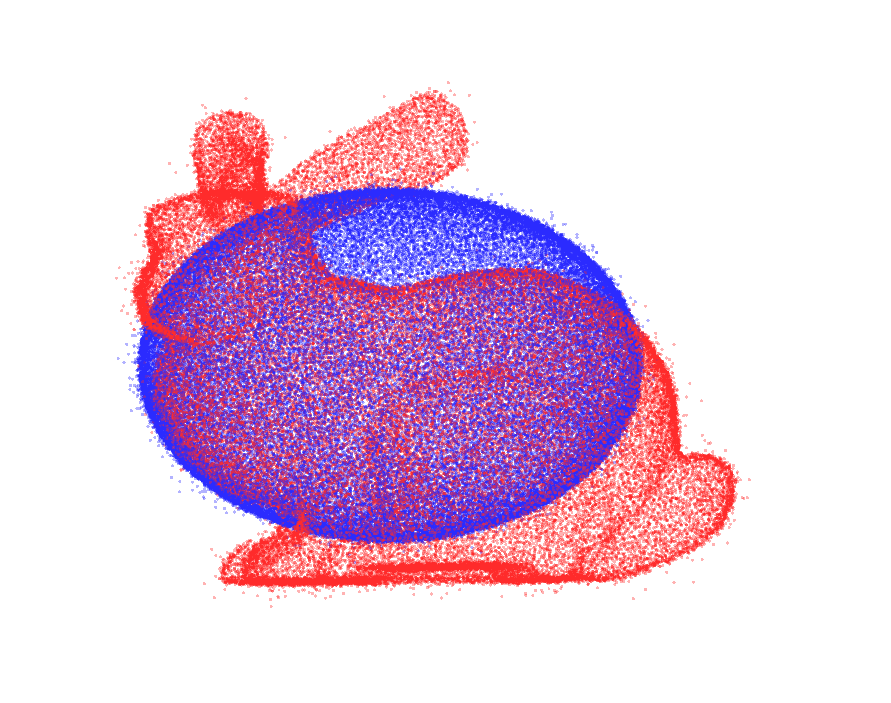}
    &\includegraphics[trim = 15mm 15mm 15mm 15mm ,clip,width=5.6cm]{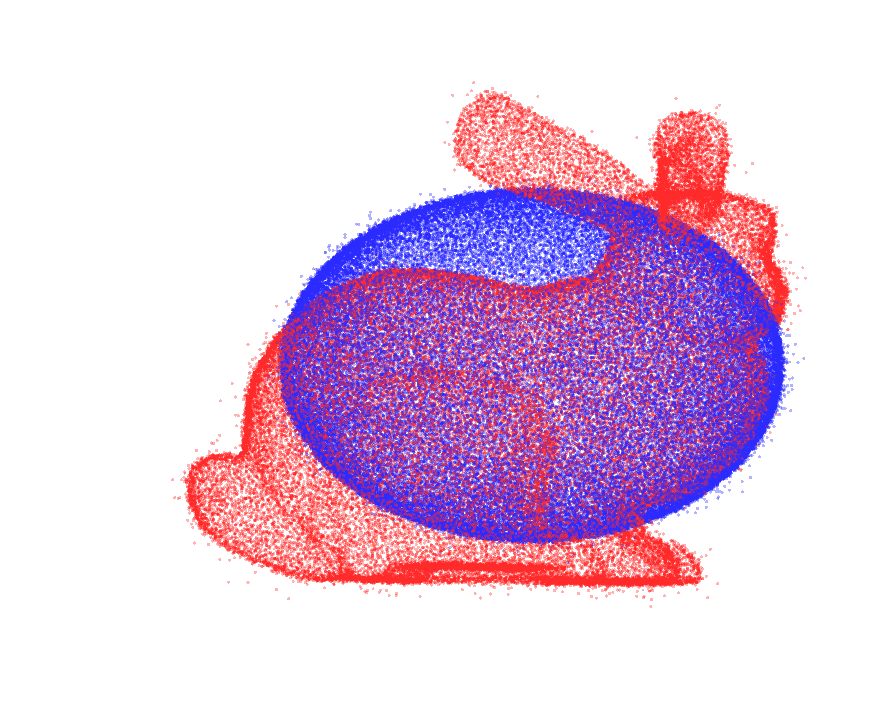} \\
    
    \rotatebox{90}{approx 2-varifold from GMRA}
    &\includegraphics[trim = 15mm 15mm 15mm 15mm,clip,width=5.6cm]{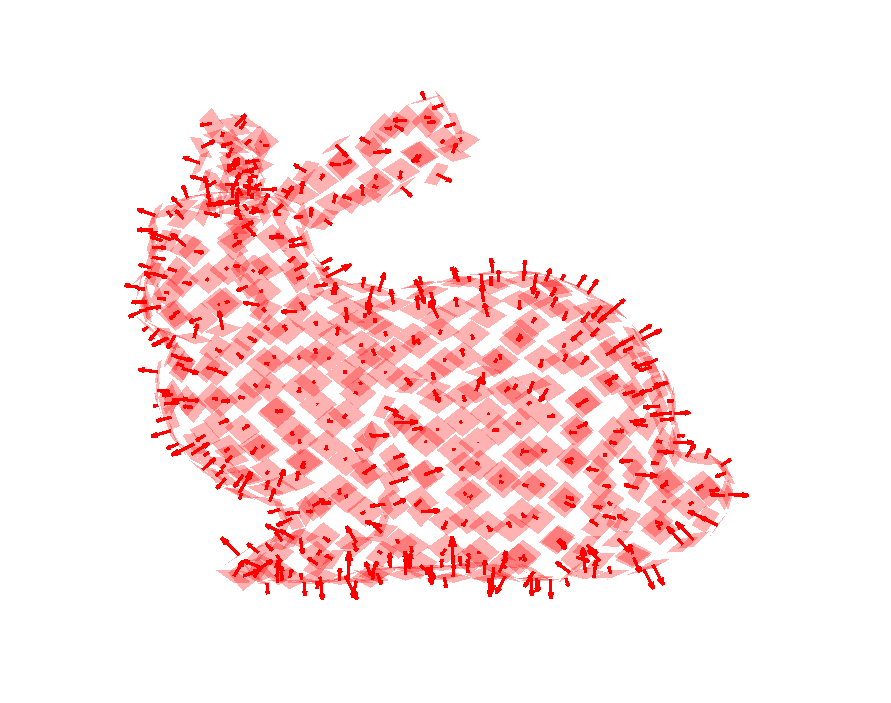}
    &\includegraphics[trim = 15mm 15mm 15mm 15mm ,clip,width=5.6cm]{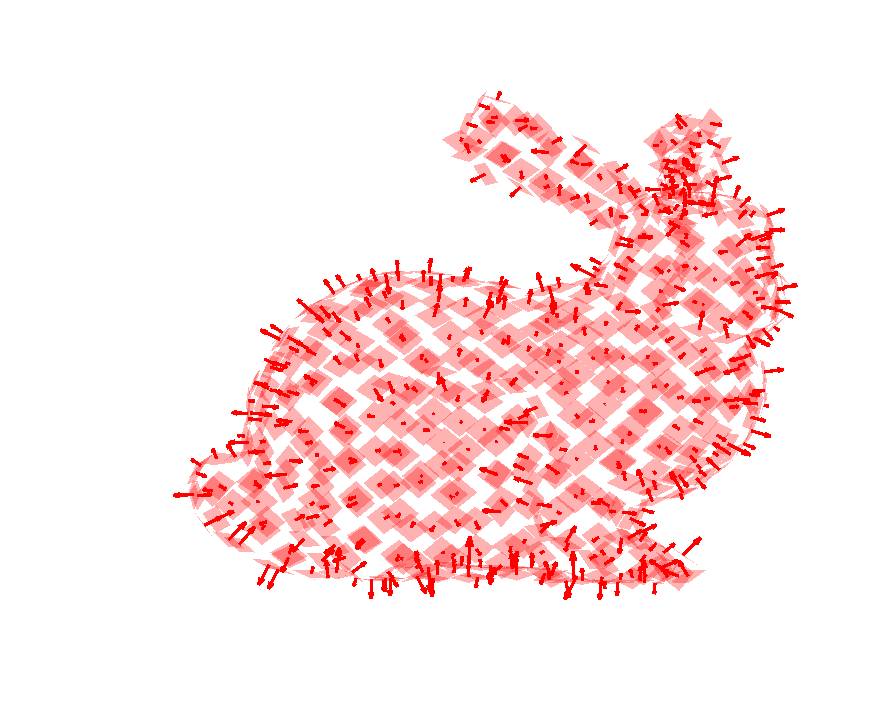} \\
    
    \rotatebox{90}{\phantom{aa} point cloud LDDMM}
    &\includegraphics[trim = 15mm 15mm 15mm 15mm,clip,width=5.6cm]{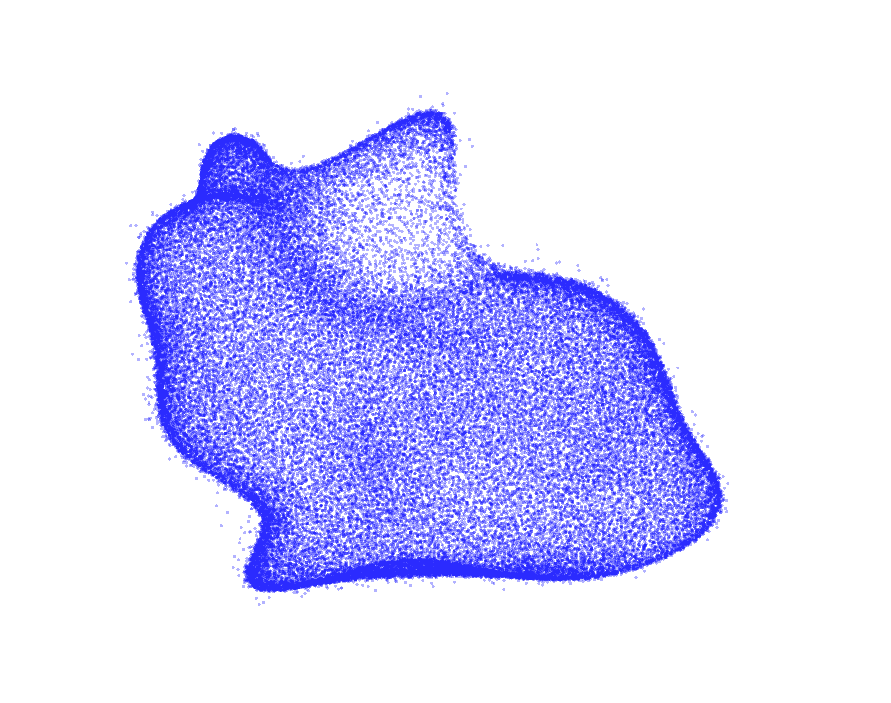}
    &\includegraphics[trim = 15mm 15mm 15mm 15mm ,clip,width=5.6cm]{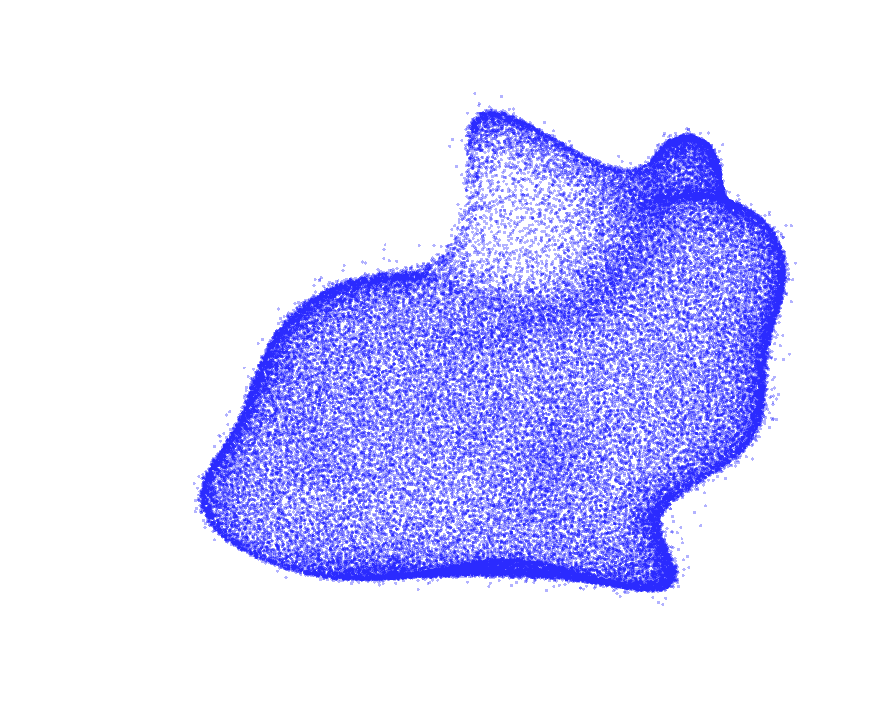} \\
    
    \rotatebox{90}{\phantom{aaaaa} varifold LDDMM}
    &\includegraphics[trim = 15mm 15mm 15mm 15mm,clip,width=5.6cm]{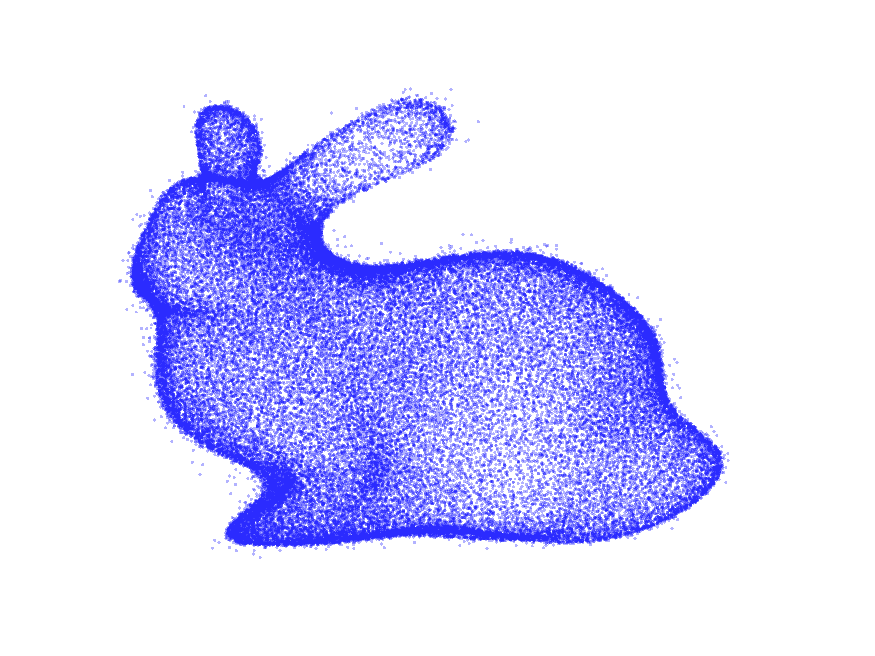}
    &\includegraphics[trim = 15mm 15mm 15mm 15mm ,clip,width=5.6cm]{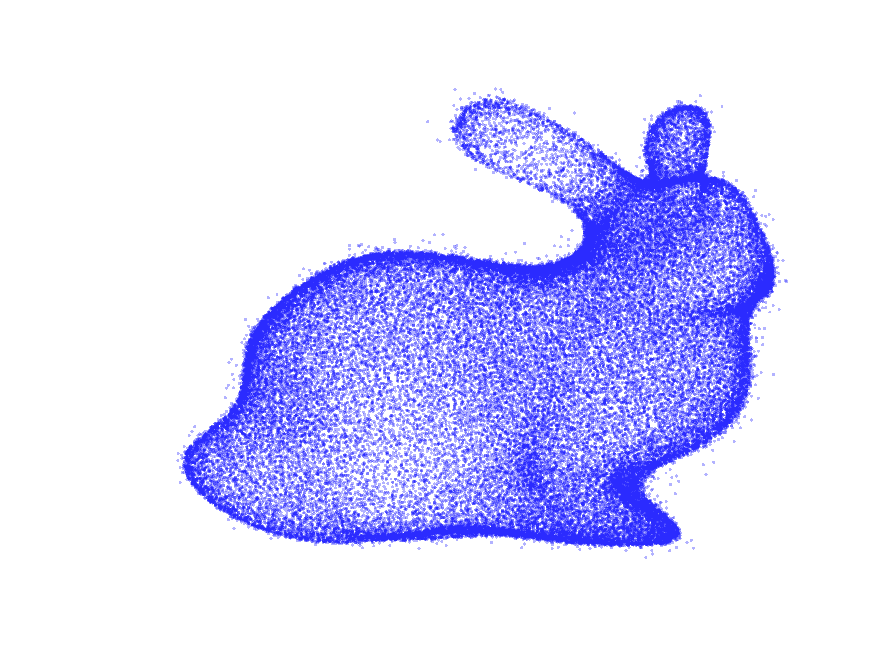} \\
    \end{tabular}
    \caption{Registration of noisy point clouds. Top row: the source (blue) and target (red) point clouds with respectively 58962 and 54834 points. Second row: illustration of the target 2-varifold obtained by GMRA with only 512 Diracs. Third row: result of direct registration of the raw point clouds. Bottom row: registration estimated from the approximate 2-varifolds.}
    \label{fig:pt cloud bunny}
    \end{center}
\end{figure*}

Along the same lines, we finally look into the case of even less structured data objects. Specifically, as displayed on the first row of Fig. \ref{fig:pt cloud bunny}, we consider two noisy point clouds which are obtained by first randomly selecting vertices from the groundtruth surfaces (with replacement) and then adding some Gaussian noises ($\sigma=0.028$) to the position of each sampled point. A first possible registration approach could be to treat such point clouds as standard measures of $\R^3$ (i.e. 0-varifolds) and follow the simple point distribution LDDMM algorithm for unlabelled point sets proposed in \cite{Glaunes2004}. The result shown on the third row of Fig. \ref{fig:pt cloud bunny} illustrates the shortcomings of such a model for this type of data. Indeed, one can see that, in the absence of any tangential information, many details of the target shape are not well-recovered. Furthermore, this point set model is not robust to sampling changes and imbalances which results in the mismatches observed below the ear region. An arguably more adequate method would be to exploit the fact that these point clouds are close to their underlying surfaces. However, due to noise and the presence of outliers, estimating triangulations of the point clouds with standard meshing algorithms can prove particularly challenging and inefficient. Instead, our approach consists in directly learning the 2-varifold structure from the point clouds based on the geometric multi-resolution analysis (GMRA) framework developed in \cite{allard2012multi}. Here, we fix a specific scale and GMRA then provides local partitions with estimates of tangent planes to the point clouds which eventually gives us an approximate representation as a 2-varifold illustrated on the second row of Fig. \ref{fig:pt cloud bunny}. Besides its robustness and numerical efficiency, such manifold learning algorithm is also particularly well suited for our proposed registration framework since it naturally leads to approximations in the form of 2-varifolds (and generally not meshes). In the last row of Fig. \ref{fig:pt cloud bunny}, we show the deformed point cloud resulting from the deformation estimated by the 2-varifold registration algorithm. It obviously outperforms the direct point cloud registration described above both in terms of quality of matching but also computation time (10 mins vs 39 mins in total).

\begin{figure*}[h]
%\fbox{\rule{0pt}{2in} \rule{.9\linewidth}{0pt}}
\hspace*{-1cm}
     \begin{tabular}{cccc}
    \includegraphics[width=4.9cm]{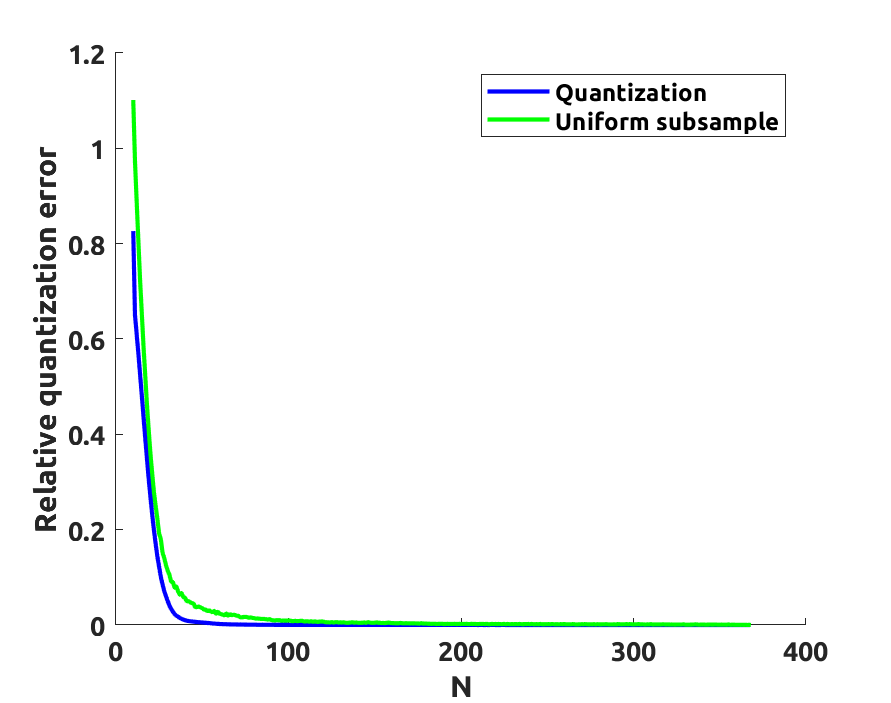}
    &\includegraphics[trim = 30mm 15mm 30mm 15mm ,clip,width=3.6cm]{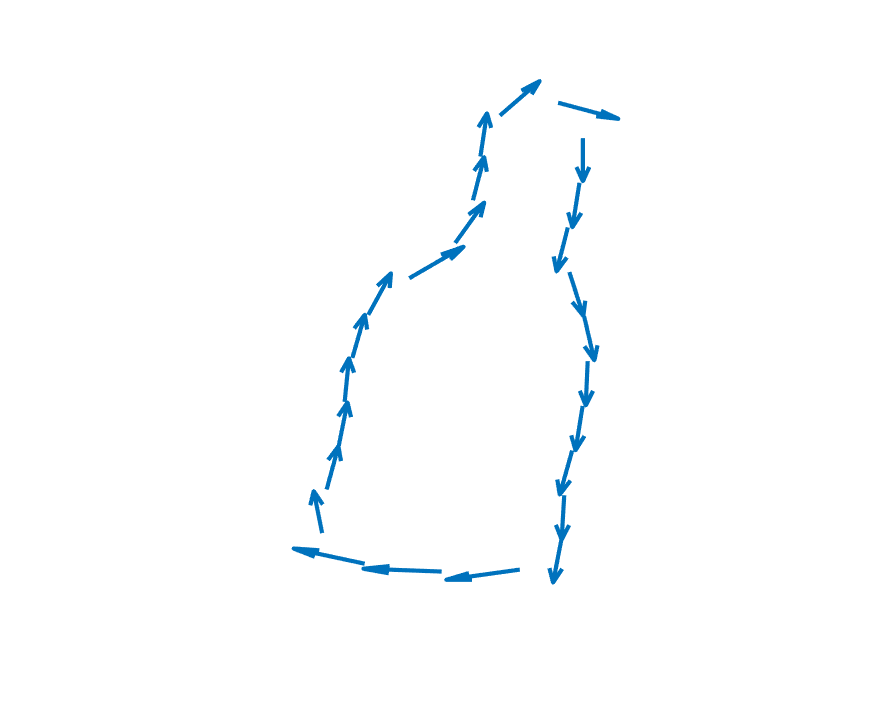}
    &\includegraphics[trim = 30mm 15mm 30mm 15mm ,clip,width=3.6cm]{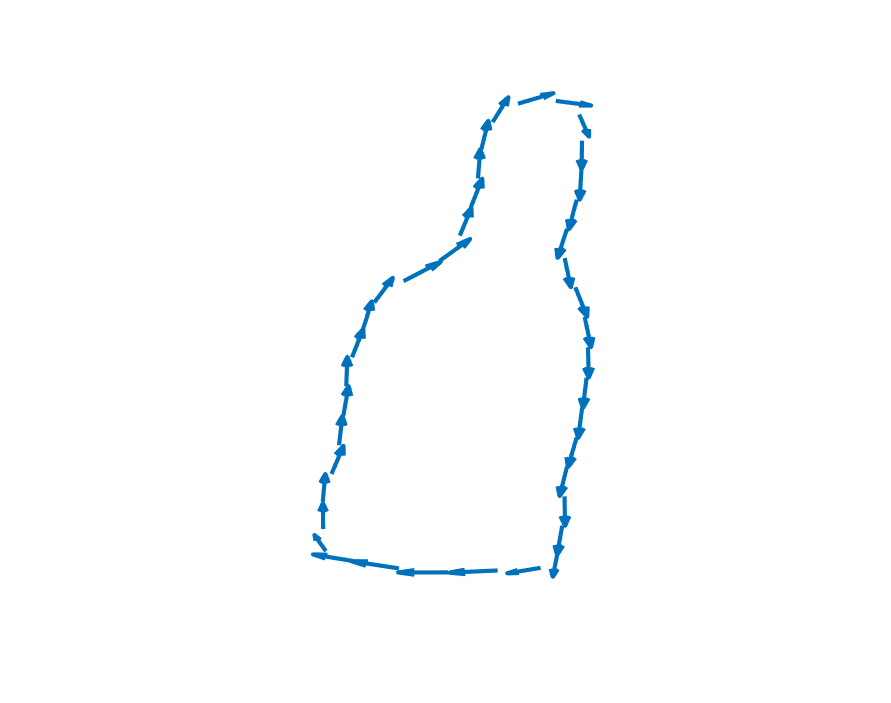}
    &\includegraphics[trim = 30mm 15mm 30mm 15mm ,clip,width=3.6cm]{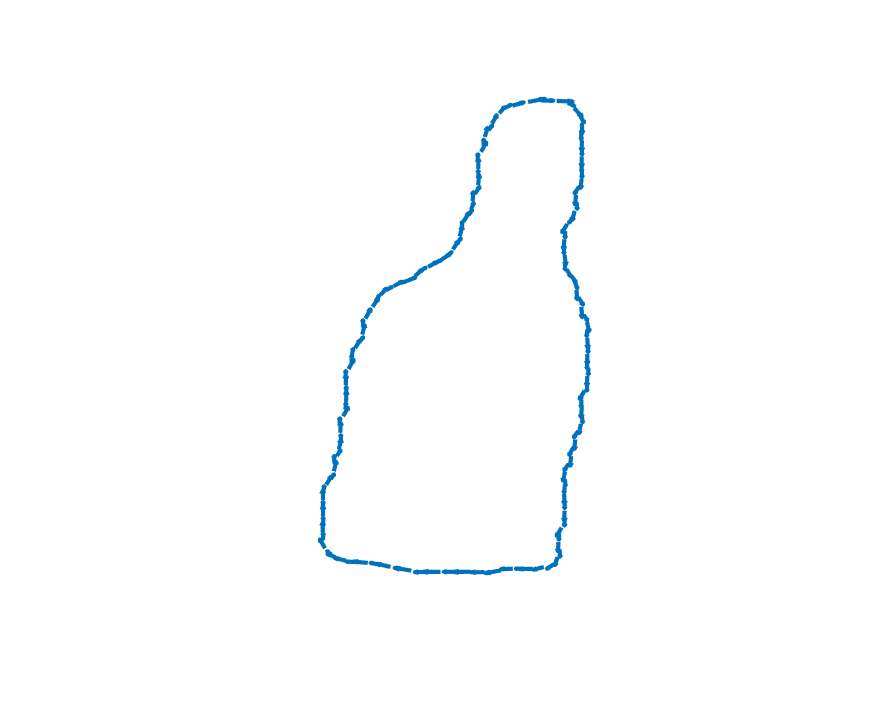} \\
    &$N=25$, rel err=$12.19\%$ &$N=40$, rel err=$1\%$ &$N=150$, rel err=$0.01\%$
    \end{tabular}
    \vskip3ex
    \hspace*{-1cm}
    \begin{tabular}{cccc}
    \includegraphics[width=4.9cm]{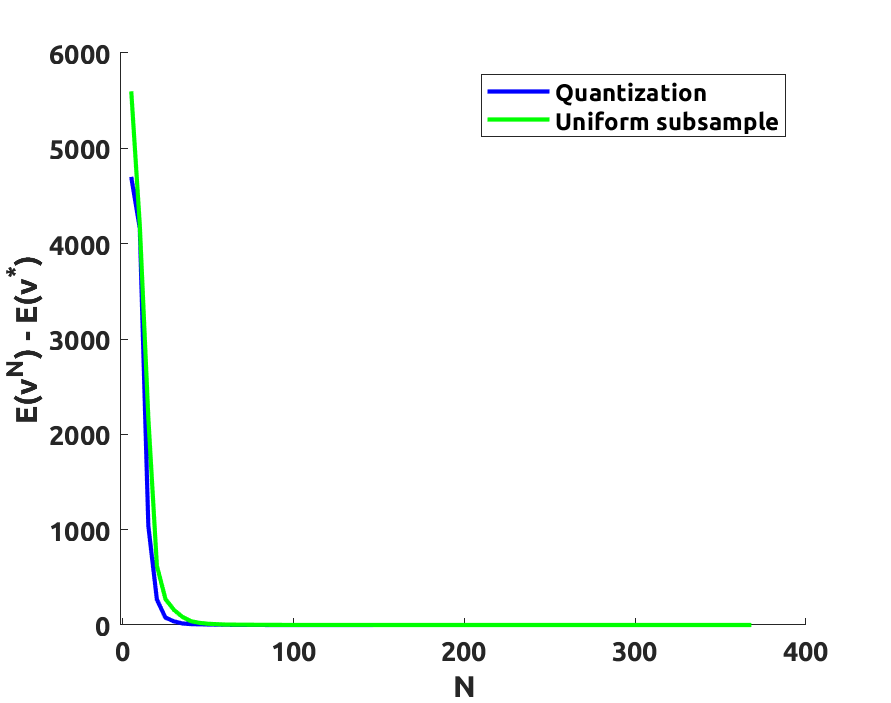}
    &\includegraphics[trim = 30mm 15mm 30mm 15mm ,clip,width=3.6cm]{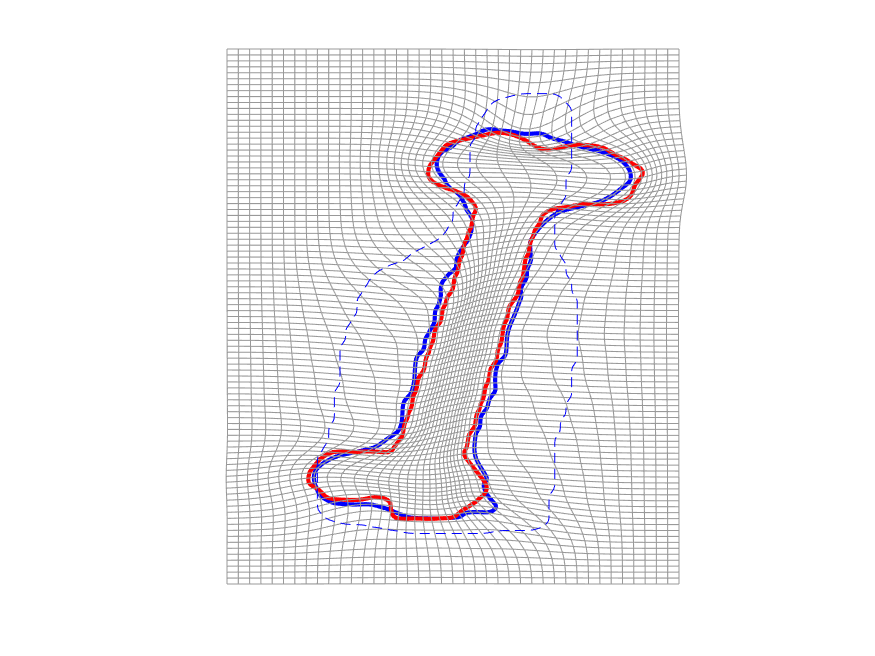} 
    &\includegraphics[trim = 30mm 15mm 30mm 15mm ,clip,width=3.6cm]{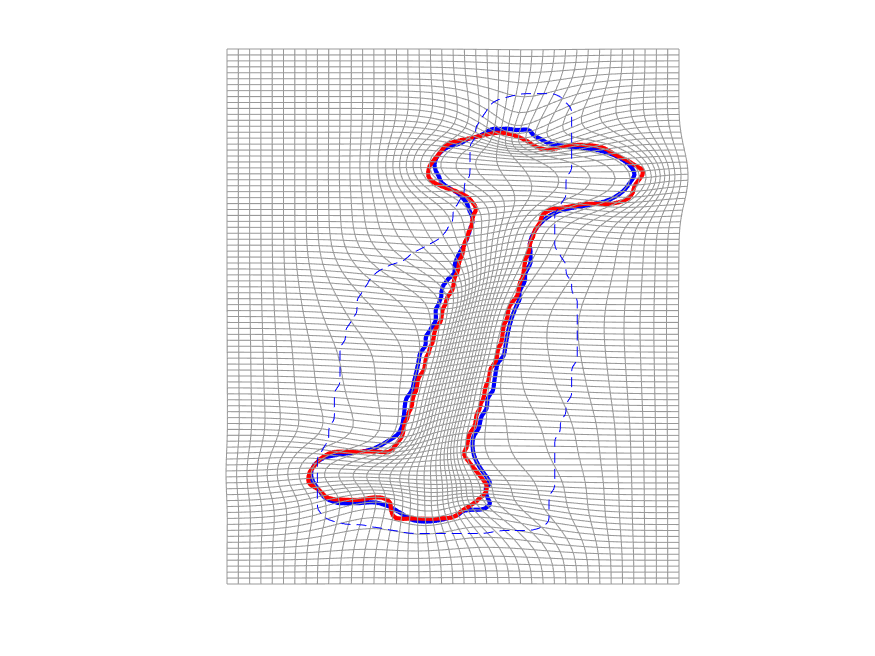} 
    &\includegraphics[trim = 30mm 15mm 30mm 15mm ,clip,width=3.6cm]{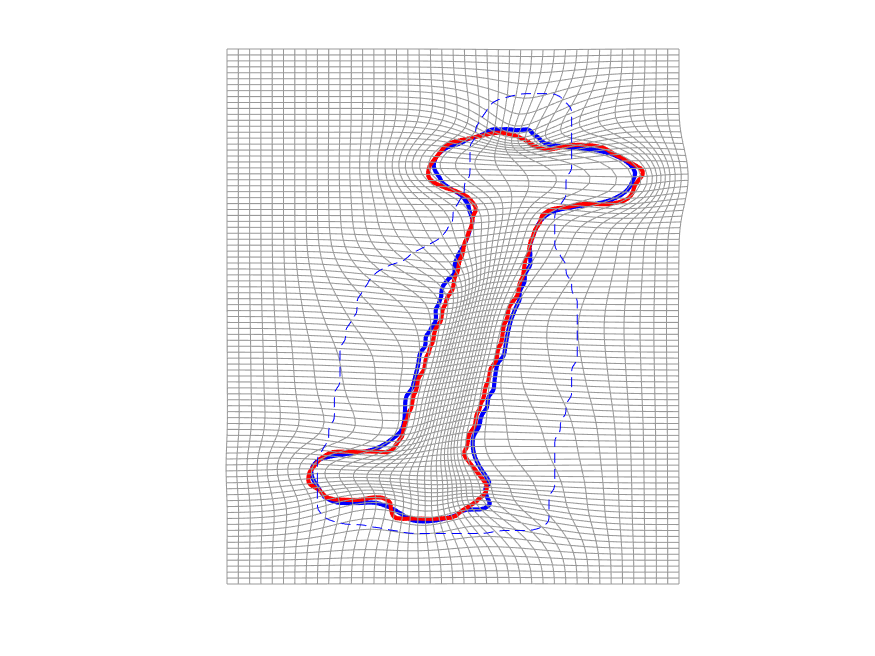} \\
    &$N=25$ &$N=40$ &$N=150$
    \end{tabular}   
  
    \caption{Compression and registration of 1-varifolds. The first row shows the results of the quantization algorithm on the 1-varifold associated to the source shape for different values of $N$; the relative quantization errors are plotted on the left (blue curve) and compared to the errors obtained with a uniform subsampling scheme (green curve). The second row shows the registration results using the apprroximated source in the first row. The plot on the left of the second row shows the difference to the groundtruth optimal energy when solving the registration problem from the approximate source given by the varifold quantization (blue) and the direct subsampling approach (green).}
    \label{fig:Bone_Bottle_quantization_registration}
\end{figure*}

\subsection{Approximation and registration}
\label{ssec:results_approx_reg}
In this second part, we examine some results of the varifold quantization procedure proposed in Section \ref{sec:approximation_discrete_var}, and in particular its interplay with the registration algorithm. Specifically, we wish to numerically validate the statements of Corollary \ref{cor:var_approx_W} and Theorem \ref{thm:convergence_sol}. We shall consider the following protocol. Starting from a highly sampled shape (that we treat as the groundtruth) for which the associated varifold $\mu_0$ is composed of a very high number of Diracs, we compute the compressed varifolds given by the $\mu_N$ of \eqref{eq:projection_problem} for increasing values of $N$ and evaluate the resulting quantization error in terms of the $d_{W^*}$ metric. Then we solve the registration problems to a fixed target $\mu_{tar}$ from the source varifolds given by the $\mu_N$ in lieu of $\mu_0$, and compare the estimated solutions to the registration of the groundtruth. For comparison, we will evaluate the total energy $E(v^N)$ of the estimated deformation fields $v^N$ for the original problem, i.e. 
\begin{equation*}
    E(v^N) = \int_0^1 \|v^N_t\|_{V}^2 dt + \lambda \|(\varphi_1^{v^N})_{\#} \mu_0 - \mu_{tar} \|_{W^*}^2. 
\end{equation*}
We shall also compare this overall approach against the alternative idea of directly subsampling the original meshes and registering those subsampled shapes with point set mesh LDDMM.

\begin{figure*}[h]
%\fbox{\rule{0pt}{2in} \rule{.9\linewidth}{0pt}}
    \begin{center}
    \begin{tabular}{cccc}
    &\includegraphics[trim = 30mm 15mm 30mm 15mm ,clip,width=5cm]{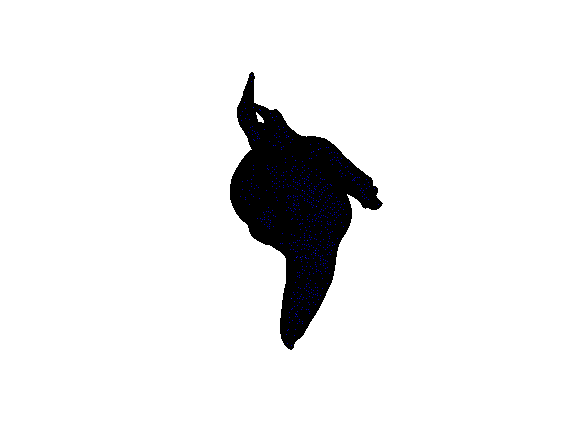} &
    &\includegraphics[trim = 30mm 15mm 30mm 15mm ,clip,width=5cm]{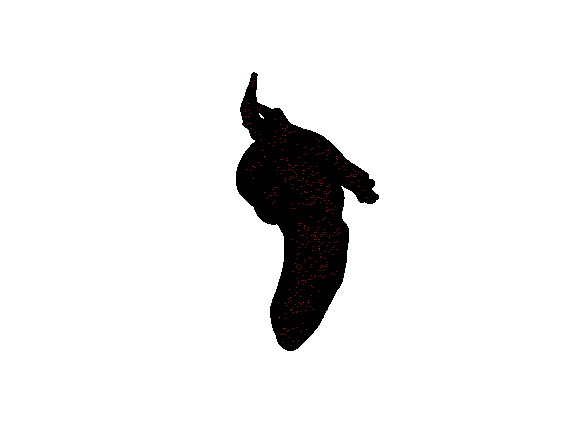} \\
    &Source surface (42448 triangles) &  & Target surface (50352 triangles)
    \end{tabular}    
    \vskip3ex
    \hspace*{-1.5cm}
    \begin{tabular}{ccccc}
    &\includegraphics[width=5.5cm]{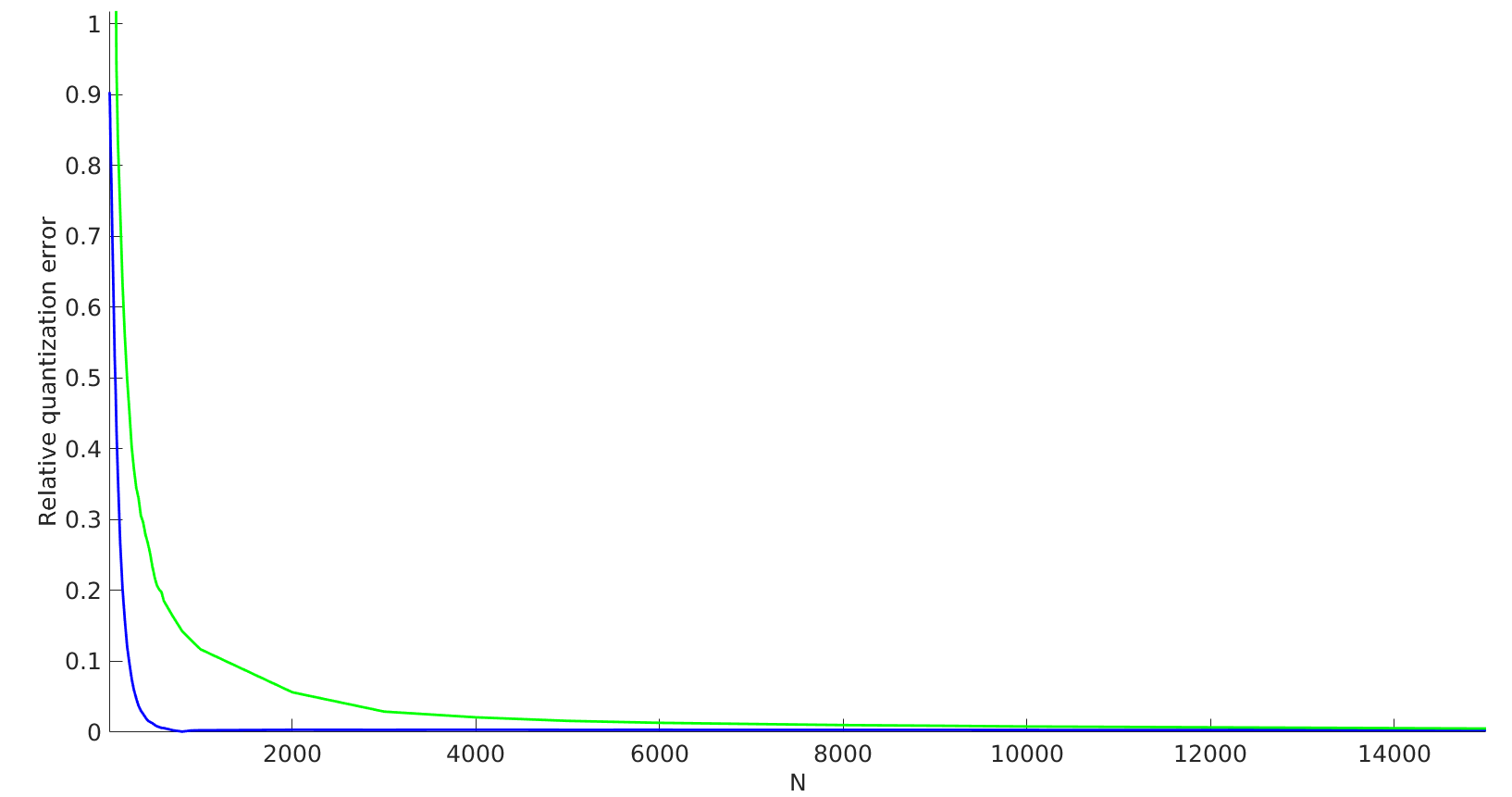}
    &\includegraphics[trim = 30mm 15mm 30mm 15mm ,clip,width=3.8cm]{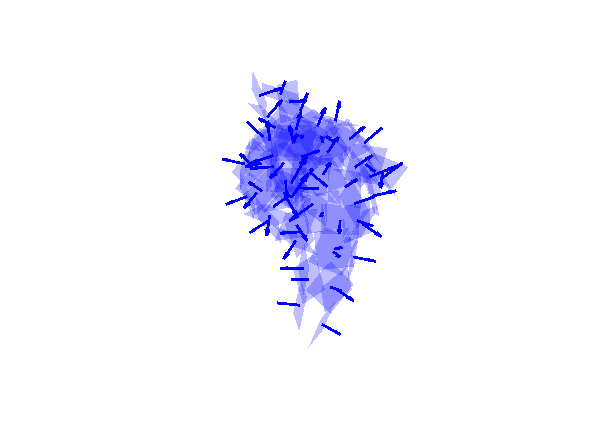} 
    &\includegraphics[trim = 30mm 15mm 30mm 15mm ,clip,width=3.8cm]{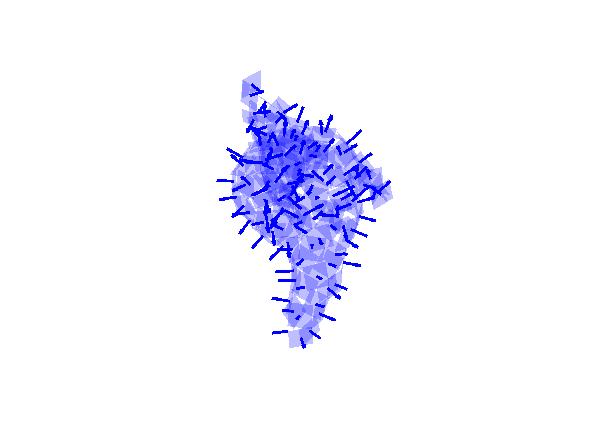}
    &\includegraphics[trim = 30mm 15mm 30mm 15mm ,clip,width=3.8cm]{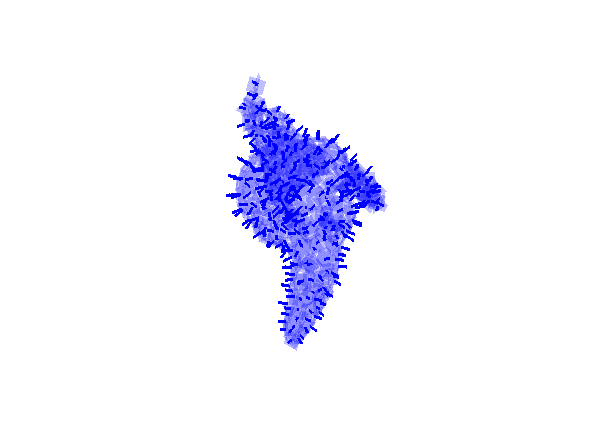} \\
    &Relative quantization error plot &$N=65$, rel err=$28\%$ &$N=125$, rel err=$7.1\%$ &$N=375$, rel err=$0.07\%$
    \end{tabular}
    \vskip3ex
    \begin{tabular}{c}
    \includegraphics[width=7cm]{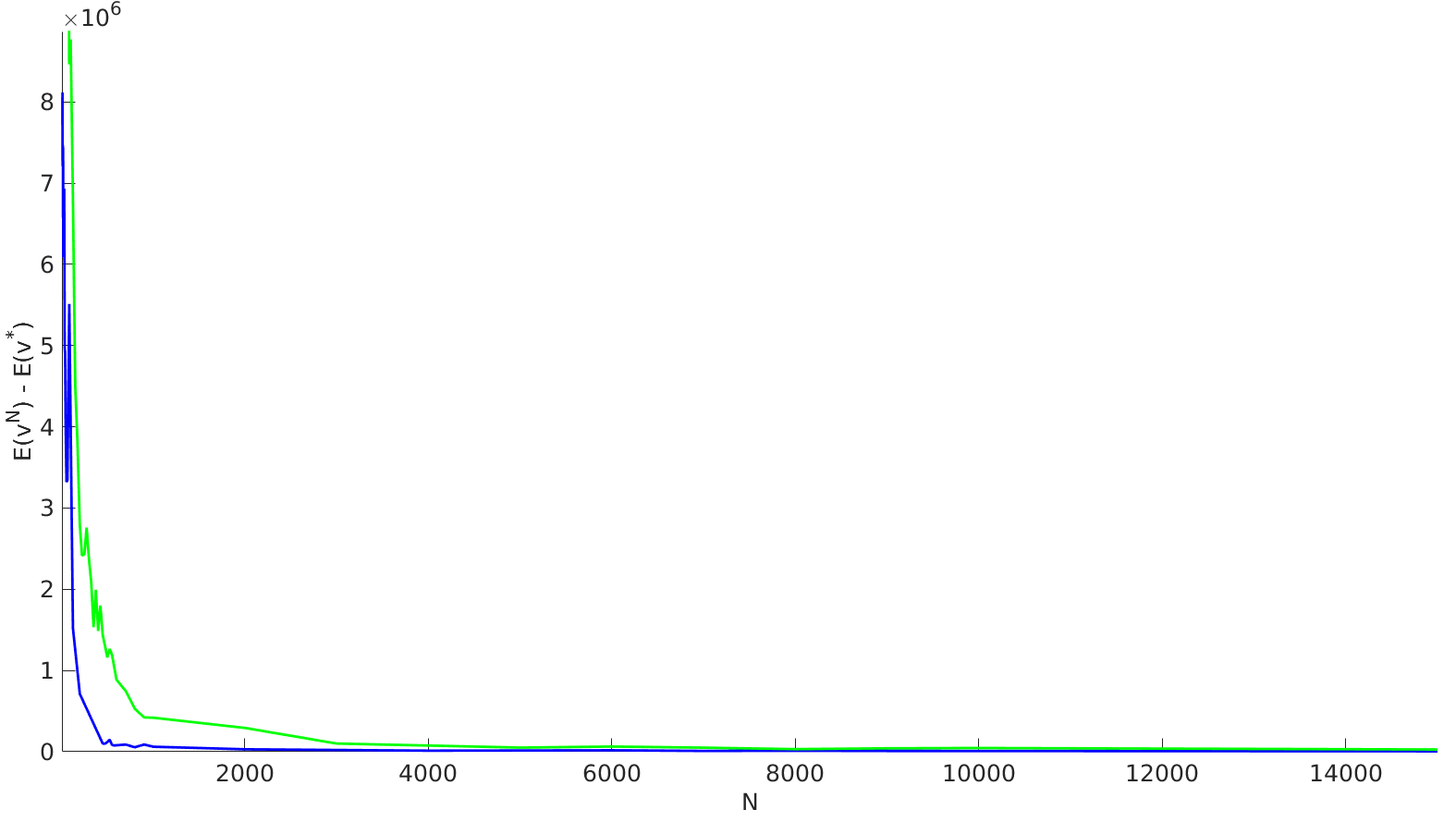} \\
    Total registration energy differences: $E(v^n)-E(v^*)$.
    \end{tabular}
    \caption{Compression and registration of 2-varifolds. On top, the source and target triangulated surfaces. The second row shows the results of the quantization algorithm on the 2-varifold associated to the source shape for different values of $N$; the relative quantization errors are plotted on the left (blue curve) and compared to the errors obtained with a mesh subsampling scheme (green curve). The plot on the third row shows the difference to the groundtruth optimal energy when solving the registration problem from the approximate source given by the varifold quantization (blue) and the mesh subsampling approach (green).}
    \label{fig:heart_quantization_registration}
    \end{center}
\end{figure*}

We begin with a 1-varifold toy example given by the curves shown in Fig. \ref{fig:Bone_Bottle_quantization_registration} from the Kimia database. These very simple curves segmented from binary images have a relatively high number of points to start with (368 vertices and edges). We look first at how well they can be approximated with smaller number of Diracs through the quantization approach described above. The upper row shows the plot of the relative approximation error $\|\mu_N - \mu_0\|_{W^*}/\|\mu_0\|_{W^*}$ of the source curve as a function of $N$ (blue) as well as the same error in varifold norm when instead the curve is uniformly subsampled (green). Consistent with the fact that varifold quantization should provide the optimal error rate at a given $N$, we observe that the error is indeed smaller than with the subsampling approach. We also display a few of the quantized $\mu_N$ for several values of $N$. As a second step, we compute the optimal deformations from the reduced shapes to the fixed target and compare their registration energies to the "groundtruth" $E(v^*)$ estimated from the full resolution source shape. The corresponding plots for the quantization versus subsampling methods are shown on the lower row in blue and green respectively. It suggests again a faster convergence to the optimal energy $E(v^*)$ with the quantization strategy, although the difference between the two methods is rather tenuous in this example.     

Those effects can be much more significant in the two-dimensional case. We emphasize it with the triangulated heart surfaces of Fig. \ref{fig:heart_quantization_registration} (data courtesy of C. Chnafa, S. Mendez and F. Nicoud, University of Montpellier). The source surface has a total of 42448 triangles leading to the same number of Diracs for the source 2-varifold $\mu_0$ and thus compressing the representation may be in that case quite critical from a computational standpoint. Indeed, computing the groundtruth matching at full resolution takes more than 7 hours (68s per iteration) in this case. We again compare two approaches: our quantization algorithm applied to $\mu_0$ versus directly subsampling the triangulated surface itself (we use here the reducepatch function in MATLAB to reduce the initial mesh to a given number of triangles). For both methods, we compute the relative approximation error $\|\mu_N - \mu_0\|_{W^*}/\|\mu_0\|_{W^*}$ with different values of $N$, the number of Diracs (resp. triangles) of the compressed varifold (resp. mesh). This is shown on the left second row in Fig. \ref{fig:heart_quantization_registration}. Unsurprisingly, we see that the quantization approach leads to a much faster decrease in the error as a function of $N$ but that in addition we obtain a very good approximation of $\mu_0$ with only a small fraction of the initial number of Diracs. Some of the quantized varifolds $\mu_N$ are displayed in the figure. We also evaluate how well the solution of the registration problem to the target varifold or surface can be approximated based on the quantized source shapes. With $v^*$ being a numerical solution for the groundtruth and $v^N$ the solutions based on the quantized source shapes, the third row of Fig. \ref{fig:heart_quantization_registration} shows the difference of the energies $E(v^N)-E(v^*)$. We observe again a faster convergence towards the groundtruth optimal energy with the varifold quantization than with mesh subsampling.

\section{Discussion}
In this paper, we proposed a registration framework between varifolds that goes beyond the previous restrictions of such models to the registration of discrete or smooth submanifolds of $\R^n$. To achieve so, we studied a general class of distances between oriented varifolds based on reproducing kernels and derived a deformation model on the space $\mathcal{V}_d$, which are combined into an optimal control formulation of the registration problem between any two varifolds. We also examined the possibility to couple this approach with a quantization/compression methodology in order to eventually tackle the registration problem, in practice, on discrete varifolds with a relatively low number of Dirac masses.     

We showed that first of all this setting leads to an equivalent yet alternative formulation to the diffeomorphic registration of rectifiable sets such as continuous or discrete curves and surfaces; the resulting higher-order Hamiltonian systems in our model provides richer local patterns for the deformations but at the price of a higher numerical cost. From an application standpoint, however, the main advantage we expect from this framework is that it applies very naturally to more general geometric objects, in particular to typical situations where well-defined and reliable meshes are not available. We gave a taste of it through some of the examples of Section \ref{sec:results}, although future work on a larger scale will be needed in order to evaluate such benefits more thoroughly. Besides the cases mentioned here, there are also several types of data that could constitute interesting test applications for this setting. This includes for instance high-angular resolution diffusion MRI in which the data is effectively modeled as spatially distributed orientation probability distribution functions consistent with the Young measure representation of varifolds in \eqref{eq:var_mu_disintegration}, or the case of contrast-invariant image registration c.f. \cite{hsieh2019diffeomorphic}. 

At the theoretical level, there are several questions left open by this work which we believe can constitute interesting tracks for future work. One is to study the possibility of extending all or part of the results of Section \ref{sec:approximation_discrete_var} to more general kernel metrics (in particular currents) and determining tighter quantization error bounds. Moreover, the registration model at play in this paper is based on the pushforward group action of $\text{Diff}(\R^n)$ on $\mathcal{V}_d$. Yet, other group actions could be have been considered, as briefly evoked in Section \ref{ssec:deformation_models}, that involve different choices of reweighing factor, for which we could expect very different properties of the solutions to the registration problem. 

Lastly, some additional work on the numerical side is likely needed for potential future applications to large scale databases, most notably to generalize this work to the estimation of means and atlases over populations of many high resolution shapes. Indeed, as we pointed out, even with the ability to compress the size of varifolds in the registration pipeline using the quantization approach, the higher complexity of the dynamical equations involved in the registration model has a non-negligible numerical toll. This could be improved in the future by using more efficient computational schemes for the repeated evaluations of sums of kernels and derivative of kernels appearing in the Hamiltonian equations, possibly along the lines of fast multiple methods.         

\section*{Acknowledgements}
The authors would like to thank Benjamin Charlier, Siamak Ardekani, Laurent Younes and the BIOCARD team for sharing the data used in some of the examples of this paper. This work was supported by NSF grant No 1819131.

\section*{Appendix}
\textbf{Proof of Theorem \ref{thm:dist_rectifiable_var}}
\vskip1ex
We first prove that $\mathcal{H}^d(X\bigtriangleup Y) =0$. Let us denote by $W^{pos}$ and $W^G$ the RKHS associated to kernels $k^{pos}$ and $k^{G}$ respectively. Suppose that $X$ and $Y$ are rectifiable sets as above such that $\left\| \mu_X-\mu_Y \right\|_{W^*}=0$ and $\mathcal{H}^d(X\bigtriangleup Y) >0$. Without loss of generality, we may assume that $\mathcal{H}^d(X \setminus Y) >0$. From Lusin's theorem, there exists a subset $U$ of $X$ such that $T|_{U}$ is continuous and $\mathcal{H}^d(X\setminus U) <  \mathcal{H}^d(X\setminus Y)$. Let us denote by $E := U \cap (X \setminus Y)$, we see that $\mathcal{H}^d(E)>0$. Since for $\mathcal{H}^d \ a.e. \ x\in E$, 

\begin{align*}
\limsup\limits_{r \rightarrow 0} \frac{\mathcal{H}^d(B_r(x)\cap E )}{\frac{\pi^{\frac{d}{2}}}{\Gamma(\frac{d}{2}+1)}r^d} \geq \frac{1}{2^d},
\end{align*}
(cf \cite{evans2018measure}), there exists $x_0 \in E, \ \mathcal{H}^d(B_r(x_0) \cap E)>0$ for any $r >0$.  

Let $g:\widetilde{G}^n_d \rightarrow \R$ be defined by $g(\cdot) = k^G(T(x_0),\cdot)$. Since $x \longmapsto g(T(x))$ is continuous on $E$ and $g(T(x_0))>0$, there exists $r_0>0$ such that $\forall \ x \in B_{r_0}(x_0) \cap E, \ g(T(x))>0$. Let $A \doteq B_{r_0}(x_0) \cap E$ and $h(x) := \mathbf{1}_A(x)$, then $\mathcal{H}^d(A)>0$ and $g(T(x))>0, \ \forall \ x \in A$. Using the density of $C_c(\mathbb{R}^n)$ in $L^1(\mathbb{R}^n,\mathcal{H}^d \mres (X \cup Y))$ together with the fact that $k^{pos}$ is $C_0$-universal, there exist $\{f_j\}_{j=1}^{\infty} \subset C_c(\mathbb{R}^n)$ and $\{h_j\}_{j=1}^{\infty} \subset W^{pos}$ 
such that $\lim\limits_{j \rightarrow \infty} f_j =h$ in $L^1(\mathbb{R}^n,\mathcal{H}^d \mres (X \cup Y))$ and $\|f_j-h_j \|_{\infty}< \frac{1}{j}$. Now, since $h_j \otimes g \in W$ and $\mu_X = \mu_Y$ in $W^*$, we have

\begin{equation*}
0 = (\mu_X-\mu_Y)(h_j \otimes g) = \int_{X} h_j(x) g(T(x)) d \mathcal{H}^d(x) - \int_{Y} h_j(y) g(S(y)) d \mathcal{H}^d(y) \rightarrow \int_A g(T(x)) d \mathcal{H}^d(x) >0,
\end{equation*}
which is a contradiction. Hence we have $\mathcal{H}^d(X\bigtriangleup Y) = 0$
 
Next, we show that $T(x) = S(x)$ $\mathcal{H}^d$-$a.e.$. Let $F := \{x \in X | T(x) = -S(x)\}$ and assume that $\mathcal{H}^d(F)>0$. From Lusin's theorem, there exists subset $F' \subset F$ such that $T|_{F'}$ is continuous and $\mathcal{H}^d(F')>0$. Using the upper density argument as above, we can find $z_0 \in F'$ such that $\mathcal{H}^d(B_r(z_0) \cap F')>0$ for all $r>0$. Since the map $x \mapsto \langle T(x),T(z_0) \rangle$ restricted to $F'$ is continuous, there exists a $\delta_0>0$ satisfying:
\begin{align*}
   \langle T(x),T(z_0) \rangle >0, \ \forall x \in B_{\delta_0}(z_0) \cap F'.
\end{align*}
Define $B:= B_{\delta_0}(z_0) \cap F'$, $\eta(\cdot):= \gamma(\langle \cdot,T(z_0) \rangle)$ and $u(x) := \eta(T(x)) - \eta(S(x))$. Observe that, from the assumption $\gamma(t) \neq \gamma(-t), \ \forall t \in [-1,1]$,
\begin{align*}
u(x) = \eta(T(x)) - \eta(-T(x)) \neq 0, \ \forall x \in F'.
\end{align*}
From this, we may assume that $u(x)>0, \ \forall x \in F'$. Let $\{f_j'\}_j$ and $\{h_j'\}_j$ be sequences in $C_c(\mathbb{R}^n)$ and $W_{pos}$ such that $f_j'$ converges to $\mathbf{1}_B$ in $L^1(\mathbb{R}^n,\mathcal{H}^d \mres F)$ and $\|f_j'-h_j'\|_{\infty} < 1/j$. We obtain
 
 \begin{equation*}
   0 = (\mu_X - \mu_Y| h_j' \otimes \eta) = \int_X h_j'(x) u(x) d\mathcal{H}^d(x) \rightarrow \int_{B} u(x) d\mathcal{H}^d(x) > 0,
 \end{equation*}
which is impossible. \qed
 
\vskip3ex

\textbf{Proof of Theorem \ref{thm:exist_opt_control}}
\vskip1ex
 Thanks to the first term in $E$, any minimizing sequence of $E$ is bounded in $L^2([0,1],V)$. Let $\{v_j\}$ be a subsequence of such minimizing sequence which converges weakly to some $\bar{v}$ in $L^2([0,1],V)$. Using the results of \cite{younes2019shapes} Chapter 7.2, we know that
 \begin{align*}
    \lim_{j \rightarrow \infty} \|(\varphi_1^{v_j} - \varphi_1^{\bar{v}})|_K \|_{1,\infty} = 0. 
 \end{align*}
Furthermore, for any $\omega \in W$, we have
\begin{align*}
\left|\left( (\varphi_1^{v_j})_{\#} \mu_0 - (\varphi_1^{\bar{v}})_{\#} \mu_0 | \omega \right) \right| 
&=\bigg{|} \int_{K} J_S \varphi_1^{v_j}(x) \omega(\varphi_1^{v_j}(x),d_x \varphi_1^{v_j} \cdot S)  
- J_S\varphi_1^{\bar{v}}(x) \omega(\varphi_1^{\bar{v}}(x),d_x \varphi_1^{\bar{v}} \cdot S) d \mu_0 \bigg{|} \\
&\leq \int_{K} |J_S \varphi_1^{v_j}(x)| 
 \left|\omega(\varphi_1^{v_j}(x),d_x \varphi_1^{v_j} \cdot S) -  \omega(\varphi_1^{\bar{v}}(x),d_x \varphi_1^{\bar{v}} \cdot S) \right|  d \mu_0 \\
&+ \int_{K} \left|J_S \varphi_1^{v_j}(x) - J_S\varphi_1^{\bar{v}}(x) \right|  \left|\omega(\varphi_1^{\bar{v}}(x),d_x \varphi_1^{\bar{v}} \cdot S)  \right| d \mu_0 \\
\end{align*}
Now, using the embedding $W \hookrightarrow C_0^1(\mathbb{R}^n \times \widetilde{G}^n_d)$
\begin{align*}
\left|\left( (\varphi_1^{v_j})_{\#} \mu_0 - (\varphi_1^{\bar{v}})_{\#} \mu_0 | \omega \right) \right| 
&\leq \left( \int_{K} |J_S \varphi_1^{v_j}(x)| d \mu_0 \right) \| \omega\|_{1,\infty} \| (\varphi_1^{v^N} - \varphi_1^{\bar{v}})|_K \|_{1,\infty} 
+ C \| (\varphi_1^{v_j} - \varphi_1^{\bar{v}})|_K \|_{1,\infty} \\
&\leq C' \| (\varphi_1^{v_j} - \varphi_1^{\bar{v}})|_K \|_{1,\infty}.
\end{align*}
Taking supremum over all $\omega \in W$ with $\|\omega\|_W \leq 1$, we obtain that
\begin{align*}
\|(\varphi_1^{v_j})_{\#} \mu_0 - (\varphi_1^{\bar{v}})_{\#} \mu_0 \|_{W^*} \leq  C' \| (\varphi_1^{v_j} - \varphi_1^{\bar{v}})|_K \|_{1,\infty} \rightarrow 0
\end{align*}
as $j \rightarrow \infty$. Combining this with lower semicontinuity of $v \mapsto \|v\|_{L^2([0,1],V)}^2$, we finally obtain that
\begin{align*}
    E(\bar{v}) \leq \liminf_{j \rightarrow \infty} E(v_j)
\end{align*}
and hence $\bar{v}$ is a global minimizer. \qed

\vskip3ex

\textbf{Proof of Proposition \ref{prop:variation_g}}
\vskip1ex
Recall that for all $\phi \in \text{Diff}(\R^n)$, $g(\phi) = \lambda \|\phi_{\#} \mu_0 - \mu_{tar} \|_{W^*}^2$ which we may rewrite as
\begin{equation*}
    g(\phi) = \lambda (\phi_{\#} \mu_0 | K_W(\phi_{\#} \mu_0 -2\mu_{tar})) + \lambda \|\mu_{tar}\|_{W^*}^2.
\end{equation*}
Thus, the variation with respect to $\phi$ in the Banach space $\mathcal{B}$ writes
\begin{equation*}
    \partial_{\phi} g(\phi)  = \partial_{\phi} (\phi_{\#} \mu_0 | \omega_0)
\end{equation*}
where $\omega_0 \doteq 2\lambda  K_W(\phi_{\#} \mu_0 -\mu_{tar}) \in W$. Moreover
\begin{equation*}
 (\phi_{\#} \mu_0 | \omega_0) = \int_{\R^n \times \widetilde{G}_d^n} \omega_0(\phi(x),d_x \phi \cdot T) J_T \phi(x) d \mu_0(x,T).  
\end{equation*}
Taking the variation with respect to $\phi$ along any $u \in C_0^1(\R^n,\R^n)$, we obtain:
\begin{align}
\label{eq:variation_g_1}
  (\partial_{\phi} g(\phi) | u) &= \int_{\R^n \times \widetilde{G}_d^n} \partial_x \omega_0(\phi(x),d_x \phi \cdot T) \cdot u(x) J_T \phi(x) d \mu_0(x,T) \nonumber \\
  &+\int_{\R^n \times \widetilde{G}_d^n} \partial_T \omega_0(\phi(x),d_x \phi \cdot T) \cdot (d_x u |_T) J_T \phi(x) d \mu_0(x,T) \nonumber \\
  &+\int_{\R^n \times \widetilde{G}_d^n} \omega_0(\phi(x),d_x \phi \cdot T). \text{div}_T u(x). J_T \phi(x) d \mu_0(x,T)
\end{align}
where the last term follows from the differentiation of Gram determinant matrices while the notation $\partial_T$ in the second term is a shortcut notation for differentiation on the Grassmannian which we do not explicit further here, we however refer to the similar computations done in \cite{Charon2} and to the developments in Section \ref{sec:numerics} for more details. For the first term, we can rely on the Young measure decomposition $\mu_0 = |\mu_0| \otimes \nu_x$ introduced at the end of Section \ref{subsec:defofvf} which gives:
\begin{equation*}
    (1) = \int_{\R^n} \tilde{\alpha}(\phi,x) \cdot u(x) \ d |\mu_0|(x), \ \ \text{where} \ \ \tilde{\alpha}(\phi,x) = \int_{\widetilde{G}_d^n} \partial_x \omega_0(\phi(x),d_x \phi \cdot T) J_T \phi(x) d\nu_x(T). 
\end{equation*}
We can also rewrite the third term as:
\begin{equation*}
    (3) = \int_{\R^n \times \widetilde{G}_d^n} \tilde{\gamma}(\phi,x,T)  \ \text{div}_T u(x) \ d \mu_0(x,T), \ \ \text{with} \ \ \tilde{\gamma}(\phi,x,T) =  \omega_0(\phi(x),d_x \phi \cdot T) \ J_T \phi(x).     
\end{equation*}
As for the second term in \eqref{eq:variation_g_1}, for each $(x,T)$ the integrand involves a linear combination (depending on $\phi$) of the partial derivatives of $u$ along the subspace $T$ i.e. of the elements of the matrix $d_x u|_T \in \R^{n\times d}$. Thus, without attempting to specify this term explicitly, we can in general write it as $\tilde{\beta}(\phi,x,T) d_x u|_T$ where $\tilde{B}$ is a continuous map from $\mathcal{B} \times \R^n \times \tilde{G}_d^n$ into $\mathcal{L}(\R^{n\times d},\R)$ giving us 
\begin{equation*}
    (2) = \int_{\R^n \times \widetilde{G}_d^n} \tilde{B}(\phi,x,T) d_x u|_T \ d \mu_0(x,T).
\end{equation*}
The result of the theorem then follows by setting $\alpha(x) \doteq \tilde{\alpha}(\varphi_1^v,x)$, $\beta(x,T) \doteq \tilde{\beta}(\varphi_1^v,x,T)$ and $\gamma(x,T)=\tilde{\gamma}(\varphi_1^v,x,T)$. \qed

\vskip3ex

\textbf{Proof of Proposition \ref{prop:invariance_momentum}}
\vskip1ex
We can treat the case of each particle $i$ separately and thus, without loss of generality, we may directly assume that $N=1$. We write $q(t)=(x(t),u^{(1)}(t),\cdots,u^{(d)}(t))$, $p(t)=(p^x(t),p^{u_1}(t),\ldots,p^{u_d}(t))$ for the state and costate variables along an optimal trajectory and 
\begin{align*}
    U \doteq \textrm{Span}\{u^{(1)}(1),\cdots,u^{(d)}(1)\}.
\end{align*}
Consider the group of linear transformations, $G \doteq {\rm SL} (U) \oplus {\rm GL} (U^{\perp})$, i.e., for any $\mathrm{g} \in G$,
\begin{align*}
\mathrm{g}(x) = \mathrm{g}_{\parallelsum}(x_{U}) + \mathrm{g}_{\perp}(x_{U^{\perp}}),
\end{align*}
where $x_{U}$ and $x_{U^{\perp}}$ are the orthogonal projections of $x$ on $U$ and $U^{\perp}$, with   $\mathrm{g}_{\parallelsum} \in {\rm SL}(U)$ and $\mathrm{g}_{\perp} \in {\rm GL}(U^{\perp})$. The Lie algebra of $G$ is $\mathfrak{g} = \mathfrak{sl}(U) \times \mathcal{L}(U^\perp)$ and $\mathfrak{sl}(U)$ is the set of all zero trace linear transformations of $U$. Now, consider the action of $G$ on $\mathbb{R}^{(d+1)n}$ defined as:
\begin{align*}
\mathrm{g} \cdot q := (q_0,\mathrm{g}(q_1),\cdots,\mathrm{g}(q_d)).
\end{align*}
for any $q=(q_0,\ldots,q_{d}) \in \mathbb{R}^{(d+1)n}$. We see that $\mu^{\mathrm{g}\cdot q(1)} = \mu^{q(1)}$ for all $\mathrm{g} \in G$ and therefore $g(\mathrm{g}\cdot q(1)) = g(q(1))$. 

Now, if we let $\{\mathrm{g}_t\}$ be a smooth curve in $G$ that satisfies $\mathrm{g}_0 = id$ and $\frac{d}{d\tau}|_{\tau=0} \mathrm{g}_\tau = h \in \mathfrak{g}$, differentiating the equality $g(\mathrm{g}_\tau \cdot q(1)) = g(q(1))$ shows that for any $h \in \mathfrak{g}$, we have
\begin{equation*}
    0 = (p(1)|h \cdot q(1)) = \sum_{k=1}^{d} \langle p^{u_k}(1) , h(u^{(k)}(1)) \rangle 
\end{equation*}
Since $h \in \mathfrak{g}$, we must have that $h|_{U}$ is a zero trace linear map. For any $1 \leq i < j \leq d$, we may choose $h$ such that $h(u^{(i)}(1)) = -h(u^{(j)}(1))$ and $h(u^{(k)}(1)) = 0 , \ \forall k \notin\{i,j\}$, which leads to $\langle u^{(i)}(1), p^{u_i}(1) \rangle = \langle u^{(j)}(1), p^{u_j}(1) \rangle$. Consequently,
\begin{align*}
\langle u^{(1)}(1), p^{u_1}(1) \rangle = \cdots = \langle u^{(d)}(1), p^{u_d}(1) \rangle =\alpha 
\end{align*}
for some constant $\alpha$. In addition, for any $i \neq j$, we can also choose $h$ such that $h(u^{(i)}(1)) = u^{(j)}(1)$ and $h(u^{(k)}(1))= 0 , \ \forall k \notin\{i,j\}$, which gives $\langle u^{(i)}(1),p^{u_j}(1) \rangle = 0$. It results that $D(1) = \alpha.I_{d \times d}$.

Finally, since $D(t)$ is constant by Lemma \ref{lemma:conservation_forward_eq}, we obtain that
\begin{align*}
D(t) = \left( \begin{array}{ccc}
\alpha &  & 0  \\
         & \ddots    &  \\
    0    &  &  \alpha
\end{array} \right),
\end{align*} 
for all $t \in [0,1]$. \qed

\bibliographystyle{apacite}
\bibliography{biblio}

\end{document}